\documentclass[12pt]{amsart}

\usepackage{setspace}   

\usepackage[usenames]{color}
\usepackage{fullpage,url,mathrsfs,stmaryrd}
\usepackage{amsmath,amsthm,amssymb,mathrsfs,stmaryrd,color}
\usepackage[utf8]{inputenc}
\usepackage[T1]{fontenc}

\usepackage{relsize}
\usepackage[bbgreekl]{mathbbol}
\usepackage{amsfonts}

\DeclareSymbolFontAlphabet{\mathbb}{AMSb}
\DeclareSymbolFontAlphabet{\mathbbl}{bbold}
\newcommand{\prism}{{\mathlarger{\mathbbl{\Delta}}}}

\newcommand{\BA}{{\mathbb{A}}}

\newcommand{\BG}{{\mathbb{G}}}

\newcommand{\Id}{{\mathop{\operatorname{\rm Id}}}}

\usepackage{amssymb}
\usepackage[all]{xy}
\usepackage{mathrsfs}
\usepackage{enumerate}
\usepackage{amscd}

\usepackage{bbm}

\DeclareMathOperator{\bbig}{{big}}

\DeclareMathOperator{\Aff}{{Aff}}

\DeclareMathOperator{\alg}{{alg}}

\DeclareMathOperator{\com}{{com}}

\DeclareMathOperator{\pre}{{pre}}

\DeclareMathOperator{\Div}{{Div}}

\DeclareMathOperator{\Int}{{Int}}

\DeclareMathOperator{\Ob}{{Ob}}

\newcommand{\HHom}{\underline{\on{Hom}}}

\newcommand{\MMor}{\underline{\on{Mor}}}

\newcommand{\cA}{{\mathcal A}}

\newcommand{\cI}{{\mathcal I}}

\newcommand{\cK}{{\mathcal K}}

\newcommand{\cM}{{\mathcal M}}

\newcommand{\cO}{{\mathcal O}}

\newcommand{\sD}{{\mathscr D}}

\newcommand{\sF}{{\mathscr F}}

\newcommand{\sL}{{\mathscr L}}
\newcommand{\sM}{{\mathscr M}}
\newcommand{\sN}{{\mathscr N}}

\newcommand{\sR}{{\mathscr R}}

\newcommand{\sX}{{\mathscr X}}
\newcommand{\sY}{{\mathscr Y}}

\newcommand{\fD}{{\mathfrak D}}

\newcommand{\fF}{{\mathfrak F}}

\newcommand{\nc}{\newcommand}

\nc\wh{\widehat}

\nc\on{\operatorname}

\nc\Gr{\on{Gr}}

\nc\Fl{\on{Fl}}

\newtheorem{cor}[subsubsection]{Corollary}
\newtheorem{lem}[subsubsection]{Lemma}
\newtheorem{prop}[subsubsection]{Proposition}

\newtheorem{conj}[subsubsection]{Conjecture}
\newtheorem{thm}[subsubsection]{Theorem}

\newtheorem{quest}[subsubsection]{Question}

\theoremstyle{remark}
\newtheorem{rem}[subsubsection]{Remark}

\newcommand{\BF}{{\mathbb{F}}}
\newcommand{\BM}{{\mathbb{M}}}
\newcommand{\BN}{{\mathbb{N}}}
\newcommand{\BQ}{{\mathbb{Q}}}

\newcommand{\BZ}{{\mathbb{Z}}}

\DeclareMathOperator{\Lie}{{Lie}}

\DeclareMathOperator{\red}{{red}} \DeclareMathOperator{\Spf}{{Spf}}

\DeclareMathOperator{\Cone}{{Cone}}

\newcommand{\limto}{{\displaystyle\lim_{\longrightarrow}}}
\newcommand{\rightlim}{\mathop{\limto}}

\newcommand{\leftlim}{\mathop{\displaystyle\lim_{\longleftarrow}}}
\newcommand{\limfromn}{\leftlim\limits_{\raise3pt\hbox{$n$}}}
\newcommand{\limton}{\rightlim\limits_{\raise3pt\hbox{$n$}}}

\newcommand{\rightlimit}[1]{\mathop{\lim\limits_{\longrightarrow}}\limits%
                    _{\raise3pt\hbox{$\scriptstyle #1$}}}

\newcommand{\leftlimit}[1]{\mathop{\lim\limits_{\longleftarrow}}\limits%
                    _{\raise3pt\hbox{$\scriptstyle #1$}}}

\newcommand{\epi}{\twoheadrightarrow}
\newcommand{\iso}{\buildrel{\sim}\over{\longrightarrow}}
\newcommand{\mono}{\hookrightarrow}

\DeclareMathOperator{\prim}{{prim}}

\DeclareMathOperator{\Alg}{{Alg}}
\DeclareMathOperator{\Aut}{{Aut}}

\DeclareMathOperator{\dR}{{dR}}
\DeclareMathOperator{\Distr}{{Distr}}
\DeclareMathOperator{\HT}{{HT}}

\DeclareMathOperator{\End}{{End}}

\DeclareMathOperator{\gr}{{gr}} 
\DeclareMathOperator{\Hom}{{Hom}} \DeclareMathOperator{\Ind}{{Ind}}
\DeclareMathOperator{\inj}{{inj}} 
\DeclareMathOperator{\nilp}{{nilp}}

\DeclareMathOperator{\Ker}{{Ker}} \DeclareMathOperator{\id}{{id}}
\DeclareMathOperator{\im}{{Im}}

\DeclareMathOperator{\Mor}{{Mor}}

\DeclareMathOperator{\Polyd}{{Polyd}}

\DeclareMathOperator{\Spec}{{Spec}}

\DeclareMathOperator{\WCart}{{WCart}}

\theoremstyle{definition}

\numberwithin{equation}{section}

\newcommand{\Fr}{\operatorname{Fr}}

\newcommand{\Fun}{\operatorname{Fun}}

\begin{document}
\title[A 1-dimensional formal group]{A 1-dimensional formal group over the prismatization of $\Spf\BZ_p$}
\author{Vladimir Drinfeld}
\address{University of Chicago, Department of Mathematics, Chicago, IL 60637}

\begin{abstract}
Let $\Sigma$ denote the prismatization of $\Spf\BZ_p$. The multiplicative group over $\Sigma$ maps to the prismatization of $\BG_m\times\Spf\BZ_p$. We prove that the kernel of this map is the Cartier dual of some $1$-dimensional formal group over $\Sigma$. We obtain some results about this formal group (e.g., we describe its Lie algebra).
We give a very explicit description of the pullback of the formal group to the quotient stack $Q/\BZ_p^\times$, where $Q$ is the $q$-de Rham prism.
\end{abstract}


\maketitle

\tableofcontents

\section{Introduction}
Let $p$ be a prime. 

\subsection{Subject of this article}
In their remarkable work   \cite{BS} B.~Bhatt and P.~Scholze introduced the theory of \emph{prismatic cohomology} of $p$-adic formal schemes.
B.~Bhatt and J.~Lurie realized that the theory of \cite{BS} has a stacky reformulation; it is based on a certain \emph{prismatization functor,} which we denote\footnote{Bhatt and Lurie \cite{BL,BL2} write $\WCart_X$ instead of $X^\prism$ and $\WCart$ instead of $\Sigma:=(\Spf\BZ_p)^\prism$.} by 
$X\mapsto X^\prism$. This is a functor from the category of bounded $p$-adic formal schemes to that of stacks\footnote{Bhatt and Lurie also define a derived version of the prismatization functor. The difference between derived and non-derived prismatization is irrelevant for our article.}.

Following \cite{Prismatization}, we write $\Sigma:=(\Spf\BZ_p)^\prism$. 
The stack $\Sigma$ plays a fundamental role in the theory of prismatic cohomology. 

In general, there is no canonical map $X\times\Sigma\to X^\prism$. However, such a map exists if $X=\BG_m\times\Spf\BZ_p$. Moreover, this map is a faithfully flat group homomorphism (more precisely, a homomorphism from a commutative group scheme over $\Sigma$ to a Picard stack over~$\Sigma$). Let $G_\Sigma$ be its kernel; it is a flat affine commutative group scheme over $\Sigma$.

Our first main result (Theorem~\ref{t:dual to formal group}) says that $G_\Sigma$ is the Cartier dual of some 1-dimensional formal group over~$\Sigma$, which we denote
by $H_\Sigma$. Then $\Lie (H_\Sigma)=\HHom (G_\Sigma ,(\BG_a)_\Sigma )$ is a line bundle on $\Sigma$. It turns out to be inverse to the Breuil-Kisin-Tate module
$\cO_\Sigma \{ 1\}$ (see Theorem~\ref{t:2Hom (G_Sigma ,G_a)}). The corresponding homomorphism $G_\Sigma\to\cO_\Sigma  \{ 1\}$ 
is explicitly constructed in \cite{BL,BL2} and called the \emph{prismatic logarithm}; it is used in \cite{BL} to define the \emph{prismatic first Chern class}.

We obtain some results about the formal group $H_\Sigma$ (see \S\ref{ss:Results on H_Sigma}), but we are unable to describe it explicitly. However, in \S\ref{ss:H_Q}-\ref{ss:Z_p^times-action on H_Q} we give a very explicit description of the pullback of $H_\Sigma$ to the quotient stack $Q/\BZ_p^\times$, where $Q$ is the $q$-de Rham prism.

The author's study of $G_\Sigma$ and $H_\Sigma$ was motivated by the desire to understand certain aspects of \cite{BS} and \cite{BL,BL2}
(see Remark~\ref{r:why we care about G_Sigma , H_Sigma} and Appendix~\ref{s:prismatic cohomology of G_m} for more details). On the other hand, $H_\Sigma$ could be interesting from the topologist's point of view.

Let us note that the group scheme $G_\Sigma$ is also introduced in \cite{BL2} (under the name of~$G_{\WCart}$).

\subsection{Organization}
The main results are formulated in \S\ref{s:main results}. We also formulate there a question about $G_\Sigma\,$ and a conjecture about $H_\Sigma\,$ (see \S\ref{ss:dualizing the deformation of H_Sigma to its Lie algebra} and Conjecture~\ref{c:H_Sigma^alg}).

In \S\ref{s:formal groups} we discuss some general results and constructions related to formal groups.
In \S\ref{s:Proofs} we prove the results formulated in~\S\ref{s:main results}. 

In \S\ref{s:realizations of G_Q} we describe and compare several ``realizations'' of the group scheme $G_Q:=G_\Sigma\times_\Sigma Q$; the first one immediately follows from the definition of $G_Q$, and the others come from the description of its Cartier dual. 
A key role is played by the expressions $(1+(q-1)z)^{\frac{t}{q-1}}$ and $q^{\frac{pt}{q-1}}$; the second expression  is closely related to the \emph{$q$-logarithm} in the sense of \cite[\S 4]{ALB}.

In Appendix~\ref{s:prismatic cohomology of G_m} we explain how to compute the prismatic cohomology of the punctured affine line over $\Spf\BZ_p$ using some results formulated 
in \S\ref{s:main results}.

In Appendix~\ref{s:Cartier dual of G_m^sharp} we discuss the Cartier dual of the divided powers version of $\BG_m\,$. As explained in \S\ref{sss:A way to think about G_dR}, 
the end of Appendix~\ref{s:Cartier dual of G_m^sharp} is related to \S\ref{s:Proofs}. Appendix~\ref{s:Cartier dual of G_m^sharp} is closely related to the material from \cite{BL} about the ``Sen operator''.

In Appendices~\ref{s:dual of hat G_m} and \ref{s:rescaled G_m and its Cartier dual} we describe the Cartier dual of $\hat\BG_m$ and of  its ``rescaled'' version. 
This material is used in \S\ref{s:realizations of G_Q}. As noted by the reviewer, a substantial part of Appendices ~\ref{s:dual of hat G_m} and \ref{s:rescaled G_m and its Cartier dual} is contained in \cite{MRT}.

\subsection{Acknowledgements}
I thank Alexander Beilinson,  Bhargav Bhatt,  Lance Gurney, and Jacob Lurie for valuable discussions.

The author's work on this project was partially supported by NSF grant DMS-2001425.

\section{Formulations of the main results}  \label{s:main results}
We fix a prime $p$. Let $W$ denote the scheme of $p$-typical Witt vectors; this is a ring scheme over $\BZ$.

\subsection{Some conventions}
A ring in which $p$ is nilpotent is said to be \emph{$p$-nilpotent.}
A scheme $S$ is said to be \emph{$p$-nilpotent} if $p\in H^0(S,\cO_S)$ is locally nilpotent. 

Unless specified otherwise, the word ``stack'' will mean a stack of groupoids on the category of schemes equipped with the fpqc topology.

Schemes and formal schemes are particular classes of stacks. E.g., $\Spf\BZ_p$ is the functor that associates to a scheme $S$ the set with one element if $S$ is $p$-nilpotent and the empty set otherwise.

For us, $\BA^1:=\Spec\BZ [x]$. Given a stack $\sX$, we write $\BA^1_\sX:=\BA^1\times\sX$. E.g., $\BA^1_{\Spf\BZ_p}$ is the $\Spf$ of the $p$-adic completion of $\BZ_p [x]$.

Similarly, $\BG_a$, $\BG_m$, $W$ are group schemes over $\BZ$, from which $(\BG_a)_\sX$, $(\BG_m )_\sX$, $W_\sX$ are obtained by base change to $\sX$.

\subsection{$\delta$-schemes and $\delta$-stacks}
\subsubsection{Definitions} \label{sss:def of delta-stack}
A \emph{Frobenius lift} for a stack $\sX$ is a morphism $F:\sX\to\sX$ equipped with a 2-isomorphism between 
the endomorphism of $\sX\otimes\BF_p$ induced by $F$ and the Frobenius endomorphism of $\sX\otimes\BF_p$. A \emph{$\delta$-stack} is a stack $\sX$ equipped with a Frobenius lift.

We say ``$\delta$-structure'' instead of  ``$\delta$-stack structure''. We say ``$\delta$-morphism'' instead of  ``morphism of $\delta$-stacks''. 

A $\delta$-stack which is a scheme (resp.~formal scheme) is called a \emph{$\delta$-scheme} (resp.~\emph{formal $\delta$-scheme}).

\subsubsection{Comparison with $\delta$-rings}  \label{sss:caveat}
According to \cite[Def.~2.1]{BS}, a $\delta$-ring is a ring $A$ equipped with a map $\delta :A\to A$ satisfying certain identities. These identities ensure that the map $\phi :A\to A$ given by $\phi (a)=a^p+p\delta (a)$ is a ring homomorphism (and therefore a Frobenius lift). If $A$ is $p$-torsion-free then a $\delta$-ring structure on $A$ is \emph{the same} as a Frobenius lift for $A$ or equivalently, a $\delta$-structure on $\Spec A$ in the sense of \S\ref{sss:def of delta-stack}. If $A$ is not $p$-torsion-free then the two notions are different, so the definitions of \S\ref{sss:def of delta-stack} are not so good. However, they are convenient enough for this article (because the rings that appear in it are $p$-torsion-free).

\subsubsection{Group $\delta$-schemes and ring $\delta$-schemes}     \label{sss:group delta-scheme}
By a \emph{group $\delta$-scheme} over a $\delta$-stack $\sX$ we mean a group object in the category of $\delta$-stacks equipped with a schematic\footnote{A morphism of stacks $\sY\to\sX$ is said to be \emph{schematic} if $\sY\times_{\sX}S$ is a scheme for any scheme $S$ equipped with a morphism to $s\to\sX$.} $\delta$-morphism to $\sX$. The definition of 
 \emph{ring $\delta$-scheme}  is similar. 
 
 \subsubsection{Examples}  \label{sss:examples of group delta-schemes}
 (i) The endomorphism $F:\BG_m\to\BG_m$ defined by $F(x)=x^p$ makes $\BG_m$ into a group $\delta$-scheme over $\BZ$.
 
 (ii) The \emph{Witt vector Frobenius} $F:W\to W$ makes $W$ into a ring $\delta$-scheme over $\BZ$.

\subsection{The formal $\delta$-scheme $W_{\prim}$}   \label{ss:W_prim}
Let us recall the material from \cite[\S 4.1]{Prismatization}. The same material is contained in \cite{BL}, but the notation in \cite{BL} is different: our $W_{\prim}$ is denoted there by~$\WCart_0$. 

\subsubsection{A locally closed subscheme of $W$}   \label{sss:A locally closed subscheme}
Let $A\subset W\otimes\BF_p$ be the locally closed subscheme obtained by removing $\Ker (W\epi W_2)\otimes\BF_p$ from $\Ker (W\epi W_1)\otimes\BF_p$.
In terms of the usual coordinates $x_0, x_1,\ldots$ on the scheme $W$, the subscheme $A\subset W$ is defined by the equations $p=x_0=0$ and the inequality $x_1\ne 0$.

\subsubsection{Definition of $W_{\prim}$}    \label{sss:W_prim}
Define $W_{\prim}$ to be the formal completion of $W$ along the locally closed subscheme $A$ from \S\ref{sss:A locally closed subscheme}.
In other words, for any scheme $S$, an $S$-point of $W_{\prim}$ is a morphism $S\to W$ which maps $S_{\red}$ to $A$. 
If $S$ is $p$-nilpotent and if we think of a morphism $S\to W$ as a sequence of functions $x_n\in H^0(S,\cO_S )$ then the condition is that $x_0$ is locally nilpotent and $x_1$ is invertible. If $S$ is not $p$-nilpotent then $W_{\prim} (S)=\emptyset$. 

$W_{\prim}$ is a formal affine $\delta$-scheme (the $\delta$-structure is induced by the one on $W$, see \S\ref{sss:examples of group delta-schemes}). In terms of the usual coordinates $x_0, x_1,\ldots$ on $W$, the  coordinate ring of $W_{\prim}$ is the completion of $\BZ_p[x_0,x_1,\ldots ][x_1^{-1}]$ with respect to the ideal $(p,x_0)$ or equivalently,  the $p$-adic completion of $\BZ_p [x_1,x_1^{-1},x_2,x_3,\ldots][[x_0]]$.

\subsection{The $\delta$-stack $\Sigma$}   \label{ss:def of Sigma}
Let us recall the material from \cite[\S 4.2]{Prismatization}. The same material is contained in \cite{BL}, but the notation in \cite{BL} is different: our $\Sigma$ is denoted there 
by~$\WCart$ and called the \emph{Cartier-Witt stack}.

\subsubsection{Action of $W^\times$ on $W_{\prim}$}   \label{sss:Action of W^times on W_prim} 
The morphism
\begin{equation}   \label{e:action by division}
W^\times\times W_{\prim}\to W_{\prim}, \quad (u ,\xi)\mapsto u^{-1}\xi
\end{equation}
defines an action of  $W^\times$ on $W_{\prim}$ (``action by division''). The reason why we prefer it to the action by multiplication is explained in \cite[\S 4.2.6]{Prismatization}.   The difference between the two actions is irrelevant for most purposes. Note that $W^\times$ is a $\delta$-scheme, $W_{\prim}$ is a formal $\delta$-scheme, and \eqref{e:action by division} is a $\delta$-morphism.

\subsubsection{$\Sigma$ as a quotient stack}   \label{sss:Sigma as quotient}
The $\delta$-stack $\Sigma$ is defined as follows:
\begin{equation}   \label{e:Sigma as quotient}
\Sigma:=W_{\prim}/W^\times . 
\end{equation}
In other words, $\Sigma$ is the fpqc-sheafification of the presheaf of groupoids $$R\mapsto~W_{\prim} (R)/W(R)^\times.$$
It is also the Zariski sheafification of this presheaf (see \cite{BL} or \cite[\S 4.2.2]{Prismatization}).

\subsubsection{The $S$-points of $\Sigma$} 
Instead of using the definition from \S\ref{sss:Sigma as quotient}, one can use a direct description of the groupoid of $S$-points of $\Sigma$, where $S$ is any scheme (see \cite{BL} or \cite[\S 4.2.2]{Prismatization}).

\subsection{The group $\delta$-scheme $G'_\Sigma$ over $\Sigma$}   \label{ss:G'_Sigma}
\subsubsection{The group scheme $G'_{W_{\prim}}$}   \label{sss:G'_W_prim}
We are going to define a flat affine commutative group $\delta$-scheme $G'_{W_{\prim}}$ over $W_{\prim}$ equipped with a homomorphism 
\begin{equation}  \label{e:G' to W^times}
G'_{W_{\prim}}\to W^\times_{W_{\prim}}:=W_{\prim}\times W^\times
\end{equation}
of group $\delta$-schemes over $W_{\prim}$.

As a formal $\delta$-scheme, $G'_{W_{\prim}}:=W_{\prim}\times W$. The map
\[
G'_{W_{\prim}}\times_{W_{\prim}}G'_{W_{\prim}}\to G'_{W_{\prim}}, \quad (\xi ,x_1,x_2)\mapsto (\xi ,x_1+x_2+\xi x_1x_2)
\]
is a group operation (to check this, use that $\xi$ is topologically nilpotent).

The homomorphism \eqref{e:G' to W^times} is given by
\[
(\xi ,x)\mapsto (\xi, 1+\xi x).
\]

\subsubsection{The group scheme $G'_{\Sigma}$ over $\Sigma$}   \label{sss:G'_Sigma}
Recall that $\Sigma =W_{\prim}/W^\times$. The $\delta$-morphism
\[
W^\times\times (W_{\prim}\times W)\to W_{\prim}\times W; \quad (u ,\xi, x)\mapsto (u^{-1}\xi ,ux)
\]
defines an action of $W^\times$ on $G'_{W_{\prim}}:=W_{\prim}\times W$, which lifts the action \eqref{e:action by division} on $W_{\prim}$ and 
preserves the group structure on $G'_{W_{\prim}}$ and the map \eqref{e:G' to W^times}. 

So $G'_{W_{\prim}}$ descends to a commutative group $\delta$-scheme $G'_{\Sigma}$  over $\Sigma$ equipped with a $\delta$-homomorphism
\begin{equation}  \label{e: 2G' to W^times}
G'_{\Sigma}\to W^\times_\Sigma:=W^\times\times\Sigma .
\end{equation}
$G'_{\Sigma}$ is affine and flat over $\Sigma$ because $G'_{W_{\prim}}$ is affine and flat over $W_{\prim}\,$.

\subsubsection{Relation to the prismatization of $\BG_m$} \label{sss:BG_m^prism}
The Bhatt-Lurie approach to prismatic cohomology is based on the \emph{prismatization functor} $X\mapsto X^\prism$ from the category of $p$-adic formal schemes\footnote{A \emph{$p$-adic formal scheme} is a stack $X$ equipped with a schematic morphism $X\to\Spf\BZ_p\,$.} to that of
$\delta$-stacks algebraic over $\Sigma$, see \cite[\S 1.4]{Prismatization}. If $X$ is a scheme over $\BZ$ we set $X^\prism:=(X\hat\otimes\BZ_p)^\prism$, where $X\hat\otimes\BZ_p$ is the $p$-adic completion of $X$. 

In particular, one can apply the prismatization functor to $\BG_m=\BA^1\setminus \{ 0\}$. It is easy to check that $\BG_m^\prism$ has a natural structure of strictly commutative Picard stack over $\Sigma$, and one has a canonical isomorphism of of strictly commutative Picard stacks
\begin{equation}  \label{e:BG_m^prism as Cone}
\BG_m^\prism\iso\Cone (G'_{\Sigma}\to W^\times_\Sigma  ),
\end{equation}
where the meaning of ``Cone'' is explained in \cite[\S 1.3.1-1.3.2]{Prismatization}; moreover, the isomorphism~\eqref{e:BG_m^prism as Cone} is compatible with the $\delta$-structures. We skip the details because the isomorphism~\eqref{e:BG_m^prism as Cone} will be used only to \emph{motivate} the study of $G'_{\Sigma}$ and its subgroup $G_\Sigma$ introduced below.

\subsection{The group $\delta$-scheme $G_\Sigma$ over $\Sigma$}   \label{ss:G_Sigma}
\subsubsection{Teichm\"uller embedding}
We have the Teichm\"uller embedding $\BG_m\mono W^\times$ and the retraction $W^\times\epi\BG_m$ (to a Witt vector one assigns its $0$-th component). Both $\BG_m$ and 
$W^\times$ are group $\delta$-schemes over $\BZ$ (see \S\ref{sss:examples of group delta-schemes}). The Teichm\"uller embedding is a $\delta$-homomorphism. The retraction $W^\times\to\BG_m$ is a homomorphism but \emph{not} a $\delta$-homomorphism.

\subsubsection{Definition}
$G_\Sigma$ is the preimage of the subgroup $(\BG_m)_\Sigma\subset W^\times_\Sigma$ under the homomorphism \eqref{e: 2G' to W^times}. Equivalently,  $G_\Sigma$ is the kernel of the homomorphism
\begin{equation}  \label{e: G' to W^times/G_m}
G'_{\Sigma}\to (W^\times/\BG_m)_\Sigma
\end{equation}
that comes from \eqref{e: 2G' to W^times}.

\subsubsection{Pieces of structure on $G_\Sigma$}
$G_\Sigma$ is a commutative affine group $\delta$-scheme over $\Sigma$ equipped with a $\delta$-homomorphism 
\begin{equation}  \label{e: G to G_m}
G_{\Sigma}\to (\BG_m )_\Sigma\, .
\end{equation}

\subsubsection{Notation}
For a stack $\sX$ over $\Sigma$, we write $G_{\sX}$ (resp.~$G'_{\sX}$) for the pullback of $G_{\Sigma}$ (resp.~$G'_{\Sigma}$) to $\sX$.

\subsection{Results about $G_\Sigma$}  \label{ss:Results on G_Sigma}
\begin{prop}   \label{p:G' to W^times/G_m faithfully flat}
The homomorphism \eqref{e: G' to W^times/G_m} is faithfully flat.
\end{prop}

The proof is given in \S\ref{ss:G' to W^times/G_m faithfully flat}.

\begin{cor}   \label{c:G_Sigma is flat}
$G_\Sigma$ is flat over $\Sigma$.
\end{cor}

\begin{proof}
Follows from Proposition~\ref{p:G' to W^times/G_m faithfully flat} because $G_\Sigma$ is the kernel of \eqref{e: G' to W^times/G_m}.
\end{proof}

Combining Proposition~\ref{p:G' to W^times/G_m faithfully flat} with \eqref{e:BG_m^prism as Cone}, one gets the following

\begin{cor}   \label{c:2 BG_m^prism as Cone}
One has a canonical isomorphism of strictly commutative Picard stacks
\begin{equation}  \label{e:2 BG_m^prism as Cone}
\BG_m^\prism\iso\Cone (G_{\Sigma}\to (\BG_m)_\Sigma  ).
\end{equation}
compatible with the $\delta$-structures. \qed
\end{cor}

\begin{rem}  \label{r:why we care about G_Sigma , H_Sigma}
Combining Corollary~\ref{c:2 BG_m^prism as Cone} with our results on $G_\Sigma$ and its Cartier dual $H_\Sigma\,$, one can compute the derived direct images of the structure sheaf under the morphism 
$$(\BA^1\setminus\{ 0\})^\prism =\BG_m^\prism\to (\Spf\BZ_p)^\prism =\Sigma,$$
see Appendix~\ref{s:prismatic cohomology of G_m}. This is not really a new result but rather a new point of view\footnote{The prismatization functor and the groups $G_\Sigma$, $H_\Sigma$ do not appear in \cite{BS}.} on a key result of \cite{BS}. 
\end{rem}

\medskip

Corollary~\ref{c:G_Sigma is flat} can be strengthened as follows.

\begin{thm}   \label{t:dual to formal group}
$G_\Sigma$ is the Cartier dual of some 1-dimensional commutative formal group $H_\Sigma$ over~$\Sigma$.
\end{thm}

The proof is given in \S\ref{sss:dual to formal group proof}. 
The precise definition of a formal group over a scheme $\sX$ is given in \S\ref{ss:def of 1-dim formal group}; in the case that $\sX$ is a stack see \S\ref{sss:fpqc-locality}(ii). 
(According to these definitions, a formal group is locally on $\sX$ defined by a formal group law.)

\begin{cor}    \label{c:Hom (G_Sigma ,G_a)}
$\HHom (G_\Sigma ,(\BG_a)_\Sigma )$ is a line bundle over $\Sigma$.
\end{cor}

\begin{proof}
By Theorem~\ref{t:dual to formal group},  $\HHom (G_\Sigma ,(\BG_a)_\Sigma )=\Lie (H_\Sigma)$.
\end{proof}

Our next goal is to formulate Theorem~\ref{t:2Hom (G_Sigma ,G_a)}, which says that $\HHom (G_\Sigma ,(\BG_a)_\Sigma )$ is canonically isomorphic to~$\cO_\Sigma \{ -1\}$, i.e., the inverse of the Breuil-Kisin-Tate module\footnote{For the definition of $\cO_\Sigma  \{ 1\}$, see  \cite[\S 4.9]{Prismatization} or \cite{BL}; one of the equivalent definitions is essentially recalled in~\S\ref{sss:Hom (G_Sigma ,G_a) proof}. Let us note that in \cite{BL} our
$\cO_\Sigma  \{ 1\}$ is called the Breuil-Kisin line bundle and denoted by $\cO_{\WCart}  \{ 1\}$.} $\cO_\Sigma  \{ 1\}$.
To explain the word ``canonically'',
we have to discuss $\rho_{\dR}^*G_\Sigma$, where $\rho_{\dR}:\Spf\BZ_p\to\Sigma$ is the ``de Rham point'' of $\Sigma$.

\subsubsection{The ``de Rham pullback'' of $G_\Sigma$}   \label{sss:2G_dR}
The element $p\in W(\BZ_p )$ defines a morphism 
$$\Spf\BZ_p\to W_{\prim}.$$ 
The corresponding morphism $\Spf\BZ_p\to\Sigma$ is called the \emph{de Rham point} of $\Sigma$ and denoted by $\rho_{\dR}:\Spf\BZ_p\to\Sigma$. 

Let  $G_{\dR}:=\rho_{\dR}^*G_\Sigma$. By the definition of $G_\Sigma\,$, for any $p$-nilpotent ring $A$ we have
\begin{equation}  \label{e:2def of G_dR}
G_{\dR}(A):=\{ x\in W(A)\,|\, 1+px\in A^\times\subset W(A)^\times\},
\end{equation}
where $A^\times\subset W(A)^\times$ is the image of the Teichm\"uller embedding, and the group operation on $G_{\dR}(A)$ is given by
$(x_1,x_2)\mapsto x_1+x_2+px_1x_2$.

We have a canonical homomorphism 
\begin{equation}   \label{e:G_dR to G_a}
G_{\dR}\to (\BG_a)_{\Spf\BZ_p}, \quad x\mapsto p^{-1}\log (1+px_0):=\sum_{n=1}^\infty \frac{(-p)^{n-1}}{n}x_0^n,
\end{equation}
where $x_0$ is the $0$-th component of the Witt vector $x$ (the formula makes sense because the numbers $ \frac{(-p)^{n-1}}{n}$ are in $\BZ_p$ and converge to $0$). 

\begin{prop}    \label{p:G_dR}
The homomorphism \eqref{e:G_dR to G_a} induces an isomorphism $G_{\dR}\iso (\BG_a^\sharp)_{\Spf\BZ_p}\,$, where $\BG_a^\sharp$ is the divided powers version of $\BG_a$.
\end{prop}

The proposition is proved in \S\ref{ss:G_dR}.

\begin{cor}
The homomorphism \eqref{e:G_dR to G_a} induces an isomorphism 
\begin{equation}   \label{e:HHom (G_dR ,G_a)}
(\BG_a)_{\Spf\BZ_p}\iso\HHom (G_{\dR} ,(\BG_a)_{\Spf\BZ_p} ).
\end{equation}
\end{cor}

\begin{thm}   \label{t:2Hom (G_Sigma ,G_a)}
There is a unique isomorphism 
\begin{equation}       \label{e:2Hom (G_Sigma ,G_a)}
\cO_\Sigma \{ -1\}\iso\HHom (G_\Sigma ,(\BG_a)_\Sigma )
\end{equation}
 whose $\rho_{\dR}$-pullback is the isomorphism \eqref{e:HHom (G_dR ,G_a)}.
\end{thm}

In the theorem and the next corollary we tacitly use that $\rho_{\dR}^*\cO_\Sigma \{ 1\}$ is canonically trivial, see  \cite[\S 4.9]{Prismatization}.
The existence of \eqref{e:2Hom (G_Sigma ,G_a)} is proved in \S\ref{sss:Hom (G_Sigma ,G_a) proof}; uniqueness follows from the equality 
\begin{equation}  \label{e:Fun(Sigma)}
H^0(\Sigma , \cO_\Sigma)=\BZ_p\,, 
\end{equation}
which is proved in \cite[Cor.~4.7.2]{Prismatization}. Combining Theorem~\ref{t:2Hom (G_Sigma ,G_a)} with \eqref{e:Fun(Sigma)}, we get

\begin{cor}   \label{c:prismatic logarithm}
There is a unique homomorphism $G_\Sigma \to\cO_\Sigma \{ 1\}$, whose $\rho_{\dR}$-pullback is the homomorphism
\eqref{e:G_dR to G_a}. \qed
\end{cor}

A homomorphism $G_\Sigma \to\cO_\Sigma \{ 1\}$ with this property is explicitly constructed in \cite{BL} (see also \cite[\S 4]{BL2}); it is denoted there by $\log_\prism$ and called the \emph{prismatic logarithm}. The prismatic logarithm  is used in \cite{BL} to define the \emph{prismatic first Chern class}.

\subsubsection{Pullback of $G_\Sigma$ to the Hodge-Tate divisor}   \label{sss:G_Delta_0}
We have a homomorphism $W\to W_1=\BA^1$ (to a Witt vector it associates its $0$-th coordinate). It induces a morphism 
$$\Sigma=W_{\prim}/W^\times\to\BA^1/\BG_m.$$
Let $\Delta_0\subset\Sigma$ be the preimage of $\{ 0\}/\BG_m\subset\BA^1/\BG_m$. Then $\Delta_0$ is an effective Cartier divisor on $\Sigma$ 
(in the sense of \cite[\S 2.10-2.11]{Prismatization}).
It is called the \emph{Hodge-Tate divisor}. Let us note that in \cite{BL} this divisor is denoted by $\WCart^{\HT}$.

Let $G_{\Delta_0}$ be the pullback of $G_\Sigma$ to $\Delta_0$. Let $\sM$ be the conormal line bundle of $\Delta_0\subset\Sigma$. Let $\sM^\sharp$ be the divided powers version of $\sM$ (so $\sM^\sharp$ and $\sM$ are obtained from $\BG_a^\sharp$ and $\BG_a$ by twisting them with the same $\BG_m$-torsor on $\Delta_0$).

\begin{prop}   \label{p:G_Delta_0}
$G_{\Delta_0}$ is isomorphic to $\sM^\sharp$. Accordingly, the Cartier dual of $G_{\Delta_0}$ is isomorphic to the formal completion of the line bundle $\sM^*$ along its zero section.
\end{prop}

The proposition will be proved in \S\ref{sss:G_Delta_0 proof}.

Let us note that Proposition~\ref{p:G_Delta_0} agrees with Theorem~\ref{t:2Hom (G_Sigma ,G_a)} because the pullback of $\cO_\Sigma \{ 1\}$ to $\Delta_0$ is known to be canonically isomorphic to $\sM$ (e.g., see \cite[Lemma~4.9.7(ii)]{Prismatization} and \cite[\S 4.9.1]{Prismatization}).

\subsubsection{Warning}
Let $\bar\rho_{\dR}:\Spec\BF_p\to\Sigma$ be the restriction of $\rho_{\dR}:\Spf\BZ_p\to\Sigma$. Then  $\bar\rho_{\dR}$ lands into $\Delta_0\subset\Sigma$. So
one can compute $\rho_{\dR}^*G_\Sigma$ using either Theorem~\ref{t:2Hom (G_Sigma ,G_a)} or Proposition~\ref{p:G_Delta_0}. Thus we get two isomorphisms 
$\rho_{\dR}^*G_\Sigma\iso (\BG_a^\sharp)_{\BF_p}$. In \S\ref{sss:somewhat unexpected} we will see that they differ by a \emph{non-linear} automorphism of 
$(\BG_a^\sharp)_{\BF_p}$ (there are plenty of such automorphisms because the Cartier dual of $(\BG_a^\sharp)_{\BF_p}$ is $(\hat\BG_a)_{\BF_p}$).

\subsection{A question about $G_\Sigma$}  \label{ss:dualizing the deformation of H_Sigma to its Lie algebra}
\subsubsection{}
By \S\ref{ss:Deformation to Lie algebra}, any formal group has a canonical ``degeneration'' into its Lie algebra. In particular, 
we have a canonical formal group over $\Sigma\times\BA^1$ whose restriction to $\Sigma\times \{ 1\}$ is $H_\Sigma$ and whose restriction to $\Sigma\times \{ 0\}$ is $\Lie (H_\Sigma )$. By Theorem~\ref{t:2Hom (G_Sigma ,G_a)}, $\Lie (H_\Sigma )=\cO_\Sigma \{ -1\}$.

\subsubsection{}   \label{sss:dualizing the deformation of H_Sigma to its Lie algebra}
Passing to the Cartier dual, we get a canonical affine group scheme over $\Sigma\times\BA^1$ whose restriction to $\Sigma\times \{ 1\}$ is $G_\Sigma$ and whose restriction to $\Sigma\times \{ 0\}$ is $(\cO_\Sigma\{ 1\})^\sharp$ (i.e., the divided powers version of the line bundle $\cO_\Sigma\{ 1\}$).

\begin{quest}    \label{q:dualizing the deformation of H_Sigma to its Lie algebra}
How to give a direct construction of the group scheme from \S\ref{sss:dualizing the deformation of H_Sigma to its Lie algebra}?
\end{quest}

\subsection{Results about $H_\Sigma$}   \label{ss:Results on H_Sigma}
\subsubsection{Pieces of structure on $H_\Sigma$}  \label{sss:Pieces of structure on H_Sigma}
The $\delta$-structure on $G_\Sigma$ is a group homomorphism 
$$G_\Sigma\to F^*G_\Sigma$$ 
whose restriction to $\Sigma\otimes\BF_p$ is the geometric Frobenius. Dualizing this, we get a group homomorphism 
\begin{equation} \label{e:varphi}
\varphi :F^*H_\Sigma\to H_\Sigma
\end{equation}
whose restriction to $\Sigma\otimes\BF_p$ is the Verschiebung.

The homomorphism \eqref{e: G to G_m} yields a section 
\begin{equation}  \label{e:s}
s:\Sigma\to H_\Sigma \, .
\end{equation}
Since \eqref{e: G to G_m} is a $\delta$-homomorphism, we have
\begin{equation} \label{e:s & varphi}
\varphi (F^*s)=ps
\end{equation}
(when writing $ps$ we are using the additive notation for the group operation in $H_\Sigma$).

\begin{thm}   \label{t:H_Sigma}
Let $s:\Sigma\to H_\Sigma$ and $\varphi :F^*H_\Sigma\to H_\Sigma$ be as in \S\ref{sss:Pieces of structure on H_Sigma}. Then

(i) $s^{-1}(0_\Sigma)=\Delta_0$, where $0_\Sigma\subset H_\Sigma$ is the zero section and $\Delta_0\subset\Sigma$ is the Hodge-Tate divisor (see \S\ref{sss:G_Delta_0});

(ii) $\varphi :F^*H_\Sigma\to H_\Sigma$ factors as $F^*H_\Sigma\iso H_\Sigma (-\Delta_0)\to H_\Sigma$.
\end{thm}

Here $H_\Sigma (-\Delta_0)$ is the formal group obtained from $H_\Sigma$ by \emph{rescaling} via the invertible subsheaf $\cO_\Sigma (-\Delta_0)\subset\cO_\Sigma\,$, see \S\ref{ss:Rescaling}. If you wish, $H_\Sigma (-\Delta_0)$ is a formal group equipped with a homomorphism $H_\Sigma (-\Delta_0)\to H_\Sigma$ vanishing at $\Delta_0$ and universal with this property (see Lemma~\ref{l:univ property of sX(-D)} and Proposition~\ref{p:univ property of sX(-D)}).

A proof of Theorem~\ref{t:H_Sigma} is given in \S\ref{ss:H_Sigma proof}.

\begin{cor}  \label{c:zeros of p^ns}
The substack of zeros of the section $p^ns$ equals $\Delta_0+\ldots+\Delta_n$, where $\Delta_i:=(F^i)^{-1}(\Delta_0)$.
\end{cor}

\begin{proof}
Combine Theorem~\ref{t:H_Sigma} with \eqref{e:s & varphi}.
\end{proof}

\subsubsection{The ``de Rham pullback'' of $H_\Sigma$}
Let $H_{\dR}:=\rho_{\dR}^*H_\Sigma$, where $\rho_{\dR}:\Spf\BZ_p\to\Sigma$ is as in~\S\ref{sss:2G_dR}.  Then $H_{\dR}$ is a formal group over $\Spf\BZ_p$ equipped with the following pieces of structure. First, \eqref{e:s} induces 
$s_{\dR}:\Spf\BZ_p\to H_{\dR}$. Second, \eqref {e:varphi} induces  a homomorphism $\varphi_{\dR} :H_{\dR}\to H_{\dR}$ (here we use that $F\circ\rho_{\dR}=\rho_{\dR}$).

\begin{prop}   \label{p:H_dR}
(i) There exists a unique isomorphism 
\[
(H_{\dR}, s_{\dR})\iso ((\hat\BG_a)_{\Spf\BZ_p}\,  , p:\Spf\BZ_p\to (\hat\BG_a)_{\Spf\BZ_p}),
\]
where $(\hat\BG_a)_{\Spf\BZ_p}$ is the formal additive group over $\Spf\BZ_p\,$.

(ii) $\varphi_{\dR}$ equals $p\in\End H_{\dR}\,$.
\end{prop}

\begin{proof}
Uniqueness in (i) is obvious. Existence in (i) follows from Proposition~\ref{p:G_dR} because $\hat\BG_a$ is Cartier dual to $\BG_a^\sharp$ via the pairing
\[
\hat\BG_a\times\BG_a^\sharp\to\BG_m\,, \quad (u,v)\mapsto \exp(uv).
\]
Statement (ii) follows from (i) because $\varphi ( s_{\dR})=ps_{\dR}$ by \eqref{e:s & varphi}.
\end{proof}

\subsubsection{Proof of Theorem~\ref{t:2Hom (G_Sigma ,G_a)}}  \label{sss:Hom (G_Sigma ,G_a) proof}
$\HHom (G_\Sigma ,(\BG_a)_\Sigma )=\Lie (H_\Sigma)$ because $H_\Sigma$ is dual to $G_\Sigma\,$.
By Theorem~\ref{t:H_Sigma}(ii) and Proposition~\ref{p:H_dR}(i), $\Lie (H_\Sigma)$ is a line bundle on $\Sigma$ equipped with an isomorphism
$F^*\Lie (H_\Sigma)\iso \Lie (H_\Sigma)(-\Delta_0)$ and a trivialization of $\rho_{\dR}^*\Lie (H_\Sigma)$. So one has a canonical isomorphism
$\Lie (H_\Sigma)\iso\cO_\Sigma \{ -1\}$, see \cite[\S 4.9]{Prismatization}. The corresponding isomorphism 
$\HHom (G_\Sigma ,(\BG_a)_\Sigma )\iso\cO_\Sigma \{ -1\}$ has the desired property.
\qed

\subsection{The pullback of $H_\Sigma$ to the $q$-de Rham prism $Q$}   \label{ss:H_Q}
\subsubsection{Recollections on $Q$}  \label{sss:Recollections on Q}
Let $Q:=\Spf\BZ_p[[q-1]]$, where $\BZ_p[[q-1]]$ is equipped with the $(p,q-1)$-adic topology.  Define $F:Q\to Q$ by $q\mapsto q^p$. Then $(Q,F)$ is
a formal $\delta$-scheme. 
More abstractly, $Q$ is the formal $\delta$-scheme underlying the formal group $\delta$-scheme $(\hat\BG_m)_{\Spf\BZ_p}$ over $\Spf\BZ_p$,
and the $\delta$-structure on $Q$ comes from the $\delta$-structure on $\BG_m$ introduced in \S\ref{sss:examples of group delta-schemes}.

Let $\Phi_p$ denote the cyclotomic polynomial. The element 
\[
\Phi_p ([q])=1+[q]+\ldots [q^{p-1}]\in W(\BZ_p[[q-1]])
\]
defines a morphism $Q\to W$ and, in fact, a morphism $Q\to W_{\prim}$. This is a $\delta$-morphism because 
$F(\Phi_p ([q])=\Phi_p (q^p)$. Let $\pi:Q\to\Sigma$ be the composite morphism $Q\to W_{\prim}\to\Sigma$. 
It is known that $\pi$ is faithfully flat. For us, the \emph{$q$-de Rham prism} is the pair $(Q,\pi )$.

Set $(\Delta_0)_Q:=\Delta_0\times_\Sigma Q\subset Q$; by the definition of $\pi$, we have
\begin{equation}          \label{e:preimage of Delta_0 in Q}
(\Delta_0)_Q=\Spf \BZ_p[[q-1]]/(\Phi_p(q))\subset\Spf \BZ_p[[q-1]]=Q.
\end{equation}

More details about $(Q,\pi )$ can be found in \cite{BL} and \cite[Appendix B]{Prismatization}.

\subsubsection{Pieces of structure on $H_Q$}   \label{sss:structures on H_Q}
Let $G_Q:=\pi^*G_\Sigma$, $H_Q:=\pi^*H_\Sigma$. 

By definition, a section $Q\to G_Q$ is the same as an element $x\in W(\BZ_p[[q-1]])$ such that $1+x\Phi_p ([q])$ is Teichm\"uller.
We will use the section $\sigma :Q\to G_Q$ corresponding to $x=[q]-1$ (then $1+x\Phi_p ([q])=[q^p]$).  It is easy to see that $\sigma :Q\to G_Q$ is a $\delta$-morphism. The section~$\sigma$ is a key advantage of $Q$ over $\Sigma$.

Since $G_Q$ is dual to $H_Q$, the section $\sigma :Q\to G_Q$ defines a homomorphism 
\begin{equation}   \label{e:sigma^*}
\sigma^*:H_Q\to (\hat\BG_m)_Q \,  .
\end{equation}

On the other hand, base-changing the pieces of structure on $H_\Sigma$ described in \S\ref{sss:Pieces of structure on H_Sigma}, we get similar pieces of structure on $H_Q$.
Namely, we get a group homomorphism 
\begin{equation} \label{e:varphi_Q}
\varphi_Q :F^*H_Q\to H_Q
\end{equation}
whose restriction to $Q\otimes\BF_p$ is the Verschiebung and a section 
\begin{equation}  \label{e:s_Q}
s_Q:Q\to H_Q \, .
\end{equation}
such that
\begin{equation} \label{e:s_Q & varphi_Q}
\varphi_Q (F^*s_Q)=ps_Q
\end{equation}
(when writing $ps_Q$ we are using the additive notation for the group operation in $H_Q$).

\eqref{e:sigma^*} interacts with \eqref{e:varphi_Q}-\eqref{e:s_Q} as follows.

\begin{lem}   \label{l:sigma^* &other strictures}
(i) The following diagram commutes:
\[
\xymatrix{
F^*H_Q\ar[r]^{\varphi_Q} \ar[d]_{F^*(\sigma_Q^*)} & H_Q\ar[d]^{\sigma_Q^*}\\
(\hat\BG_m)_Q\ar[r]^{\id}& (\hat\BG_m)_Q\
}
\]

(ii) $\sigma^*\circ s_Q=q^p\in\hat\BG_m (Q)$.
\end{lem}

\begin{proof}
Statement (i) follows from $\sigma$ being a $\delta$-morphism.

Composing $\sigma_Q:Q\to G_Q$ with the  homomorphism $G_Q\to (\BG_m)_Q$ that comes from \eqref{e: G to G_m}, we get
$1+(q-1)\cdot\Phi_p(q)=q^p\in\BG_m (Q)$. Statement (ii) follows.
\end{proof}

\begin{thm}  \label{t:H_Q & rescaled hat G_m}   
The homomorphism $\sigma^*:H_Q\to (\hat\BG_m)_Q$ induces an isomorphism
\[
H_Q\iso (\hat\BG_m)_Q(-D)\, , 
\]
where $D\subset Q$ is the divisor $q=1$. 
\end{thm}

The proof is given in \S\ref{sss: H_Q & hat G_m proof}.

\subsubsection{$H_Q$ as a formal scheme}  \label{sss:H_Q as formal scheme}
By Theorem~\ref{t:H_Q & rescaled hat G_m}, $H_Q$ identifies with $(\hat\BG_m)_Q(-D)$. So the formal scheme $H_Q$ can be obtained as follows: first, blow up
the formal scheme 
$$(\hat\BG_m)_Q=Q\times\hat\BG_m=\Spf\BZ_p[[q-1,q'-1]]$$ 
along the subscheme $q=q'=1$, then $H_Q$ is the formal completion of the blow-up along the strict preimage of the unit section of $(\hat\BG_m)_Q$. In other words, 
\begin{equation}    \label{e:coordinate on H_Q}
H_Q=\Spf\BZ_p[[q-1,z]], \quad \mbox{where } z=\frac{q'-1}{q-1}.
\end{equation}

\subsubsection{The formal group $H_Q$ in explicit terms}   \label{sss:H_Q explicitly}
In terms of the coordinate $z$ from \eqref{e:coordinate on H_Q}, $H_Q$~corresponds to the following formal group law over $\BZ_p[[q-1]]$:
\begin{equation}  \label{e:group law for H_Q}
z_1*z_2=\frac{(1+(q-1)z_1)(1+(q-1)z_2)-1}{q-1}=z_1+z_2+(q-1)z_1z_2\, .
\end{equation}

Let us describe in these terms the pieces of structure on $H_Q$ defined in \S\ref{sss:structures on H_Q}.
The homomorphism $\sigma^*:H_Q\to (\hat\BG_m)_Q$ is just the map $(q,z)\mapsto (q, 1+(q-1)z)$.
By Lemma~\ref{l:sigma^* &other strictures}(ii), the section $s_Q: Q\to H_Q$ is given by $z=\frac{q^p-1}{q-1}=\Phi_p(q)$.
It remains to describe the homomorphism $\varphi_Q :F^*H_Q\to H_Q$. The formal group $F^*H_Q$ corresponds to the group law
\begin{equation}  \label{e:group law for F^*H_Q}
y_1*y_2=y_1+y_2+(q^p-1)y_1y_2,
\end{equation}
which is the $F$-pullback of \eqref{e:group law for H_Q}. By Lemma~\ref{l:sigma^* &other strictures}(i), the homomorphism $\varphi_Q$ is the homomorphism from
\eqref{e:group law for F^*H_Q} to \eqref{e:group law for H_Q} given by $z=\Phi_p(q)\cdot y$.

\subsubsection{The group scheme $H_Q^{\alg}$}   \label{sss:H_Q^alg}
Let $H_Q^{\alg}:=\Spf A$, where $A$ is the completion of $\BZ_p[q,z]$ for the $(p,q-1)$-adic topology. The morphism $H_Q^{\alg}\to\Spf\BZ_p[[q-1]]=Q$ is affine (in particular, schematic). The r.h.s. of \eqref{e:group law for H_Q} is a polynomial, so it gives a morphism $$H_Q^{\alg}\times_Q H_Q^{\alg}\to H_Q^{\alg}.$$ This morphism makes  $H_Q^{\alg}$ into a smooth affine group scheme over $Q$. The formal completion of $H_Q^{\alg}$ along its unit identifies with $H_Q$ (if you wish, $H_Q^{\alg}$ is an \emph{algebraization} of the formal group $H_Q$ in the sense of \S\ref{sss:Algebraizations}). The homomorphism $\varphi_Q :F^*H_Q\to H_Q$ comes from a homomorphism $F^*H_Q^{\alg}\to H_Q^{\alg}$. The group $H_Q^{\alg}$ has a remarkable section $\tilde s_Q:Q\to H_Q^{\alg}$ given by $z=1$; one has a commutative diagram
\[
\xymatrix{
Q\ar[r]^{\tilde s_Q} \ar[d]_{s_Q} & H_Q^{\alg}\ar[d]^p\\
H_Q\ar@{^{(}->}[r] & H_Q^{\alg}
}
\]
($s_Q$ was defined in \S\ref{e:s_Q} and described in \S\ref{sss:H_Q explicitly}).

\subsubsection{Restriction of $H_Q^{\alg}$ to $(\Delta_0)_Q$}   \label{sss:restriction of H_Q^alg to Delta_0_Q}
To get a feel of $H_Q^{\alg}$ let us discuss its restriction to~$(\Delta_0)_Q$. 

Recall that $(\Delta_0)_Q:=\Delta_0\times_\Sigma Q$, where $\Delta_0\subset\Sigma$ is the Hodge-Tate divisor; explicitly, $(\Delta_0)_Q=\Spf\BZ_p[[q-1]]/(\Phi_p(q))\subset\Spf\BZ_p[[q-1]]=Q$. Let $\zeta\in\BZ_p[[q-1]]/(\Phi_p(q))$ be the image of $q$; then $\zeta$ is a primitive $p$-th root of $1$.

Let $H_{(\Delta_0)_Q}^{\alg}$  be the restriction of $H_Q^{\alg}$ to $(\Delta_0)_Q$. It is easy to check that one has an exact sequence
\begin{equation}  \label{e:restriction of H_Q^alg to Delta_0_Q}
0\to (\BZ/p\BZ)_{(\Delta_0)_Q}\overset{i}\longrightarrow H_{(\Delta_0)_Q}^{\alg}\overset{\lambda}\longrightarrow (\BG_a)_{(\Delta_0)_Q}\to 0;
\end{equation}
here $i$ takes $1\in\BZ/p\BZ$ to the section $\tilde s_{(\Delta_0)_Q}:(\Delta_0)_Q\to H_{(\Delta_0)_Q}$ given by $z=1$ (then $1+(\zeta -1)z=\zeta$ is a $p$-th root of unity), and $\lambda$ is given by $(\zeta -1)^{-1}\cdot\log (1+(\zeta-1)z)$ (which is a power series in $z$ whose coefficients are in $\BZ_p[\zeta ]=\BZ_p[[q-1]]/(\Phi_p(q))$ and converge to $0$).

The exact sequence \eqref{e:restriction of H_Q^alg to Delta_0_Q} shows that $H_{(\Delta_0)_Q}$ is isomorphic to $(\hat\BG_a)_{(\Delta_0)_Q}$, but
$H_{(\Delta_0)_Q}^{\alg}$ is not isomorphic to $(\BG_a)_{(\Delta_0)_Q}$ (because $\Hom ((\BZ/p\BZ)_{(\Delta_0)_Q}, (\BG_a)_{(\Delta_0)_Q})=0$).

\subsection{The action of $\BZ_p^\times$ on $H_Q$}   \label{ss:Z_p^times-action on H_Q}
We keep the notation of \S\ref{ss:H_Q}.

\subsubsection{The action of $\BZ_p^\times$ on $Q$}   \label{sss:Z_p^times-action on Q}
The pro-finite (and therefore pro-algebraic) group $\BZ_p^\times$ acts on the formal $\delta$-scheme $Q=\Spf\BZ_p[[q-1]]$: the automorphism of $Q$ corresponding to 
$n\in\BZ_p^\times$ is
$$q\mapsto q^n=\sum\limits_{i=0}^{\infty}\frac{n(n-1)\ldots (n-i+1)}{i!}(q-1)^i .$$
In terms of the identification $Q=(\hat\BG_m)_{\Spf\BZ_p}$, this action comes from the isomorphism $$\BZ_p\iso\End ((\hat\BG_m)_{\Spf\BZ_p}).$$

\subsubsection{The action of $\BZ_p^\times$ on $H_Q$}   \label{sss:Z_p^times-action on H_Q}
It is easy to show that the morphism $\pi: Q\to\Sigma$ from \S\ref{sss:Recollections on Q} factors through the quotient stack $Q/\BZ_p^\times$ (see \cite{BL} or 
\cite[Appendix B]{Prismatization}). Therefore the formal group scheme $H_Q$ is $\BZ_p^\times$-equivariant. 

The morphisms $\varphi_Q :F^*H_Q\to H_Q$ and 
$s_Q:Q\to H_Q $ are $\BZ_p^\times$-equivariant because they are $\pi$-pullbacks of $\varphi:F^*H_\Sigma\to H_\Sigma$ and 
$s:\Sigma\to H_\Sigma\,$.

\begin{prop}   \label{p:sigma^* is equivariant}
The morphism $\sigma^*:H_Q\to (\hat\BG_m)_Q:=(\hat\BG_m)_{\Spf\BZ_p}\times Q$ is $\BZ_p^\times$-equivariant assuming that $(\hat\BG_m)_{\Spf\BZ_p}$ is equipped with the following $\BZ_p^\times$-action\footnote{This $\BZ_p^\times$-action is the same as the one from \S\ref{sss:Z_p^times-action on Q} (recall that $Q$ is just the formal scheme underlying the formal group $(\hat\BG_m)_{\Spf\BZ_p}$).}:  
$n\in\BZ_p^\times$ acts as raising to the power of $n$.  
\end{prop}

The proof is given in \S\ref{ss:equivariance of sigma^*}. Proposition~\ref{p:sigma^* is equivariant} means that if we think of  $H_Q$ as an affine blow-up of $(\hat\BG_m)_{\Spf\BZ_p}\times Q$ (see~\S\ref{sss:H_Q as formal scheme}) then the action of $\BZ_p^\times$ on 
$H_Q$ is the most natural one.

\begin{cor}   \label{c:Z_p^times-action on H_Q}
In terms of \S\ref{sss:H_Q explicitly}, the action of $n\in\BZ_p^\times$ on $H_Q$ is given by
\[
(q,z)\mapsto (q^n,\frac{h_n(z,q)}{h_n(1,q)}),
\]
\[
\mbox{ where }\quad h_n(z,q)=\frac{(1+(q-1)z)^n-1}{q-1}=\sum\limits_{i=1}^\infty \frac{n(n-1)\ldots (n-i+1)}{i!}z^i (q-1)^{i-1}
\]
(so $h_n(1,q)=\frac{q^n-1}{q-1}$). \qed
\end{cor}

\begin{rem}
Corollary~\ref{c:Z_p^times-action on H_Q} combined with \S\ref{sss:H_Q explicitly} gives a complete  description of the image of the formal group $H_\Sigma$ under the pullback functor
\begin{equation} \label{e:pullback functor}
\{\mbox{Formal groups over }\Sigma\}\to\{\BZ_p^\times\mbox{-equivariant formal groups over }Q\}.
\end{equation}
If $p>2$ this functor is fully faithful by \cite[Thm.~3.8.3]{BL}, so our description of the image of $H_\Sigma$ under \eqref{e:pullback functor} could be considered as a (not very good) description of $H_\Sigma$ itself.
\end{rem}

\begin{rem}
Let $H_Q^{\alg}$ be as in \S\ref{sss:H_Q^alg}.
The action of $\BZ_p^\times$ on $H_Q$ comes from an action of $\BZ_p^\times$ on $H_Q^{\alg}$; the latter is given by the formula from 
Corollary~\ref{c:Z_p^times-action on H_Q} (this formula makes sense in the context of $H_Q^{\alg}$ because the reduction of $h_n(z,q)$ modulo any power of $q-1$ is a polynomial in $z$).
\end{rem}

\subsection{A conjectural algebraization of $H_\Sigma$}   \label{ss:H_Sigma^alg}
\subsubsection{Algebraizations of formal groups}   \label{sss:Algebraizations}
Let $H$ be a formal group over a stack $\sX$. By an \emph{algebraization} of $H$ we mean an isomorphism class of pairs consisting of a smooth affine group scheme $G$ over $\sX$ with connected fibers and an isomorphism $H\iso\hat G$, where $\hat G$ is the formal completion of $G$ along its unit. 
Let $\Alg (H)$ denote the set of algebraizations of~$H$.

\subsubsection{The sheaf property of $\Alg$}  \label{sss:sheaf property of Alg}
Suppose that in addition to $\sX$ and $H$, we are given a morphism of stacks $\sX'\to\sX$ such that the corresponding morphism of fpqc-sheaves of sets is surjective (in other words, for every scheme $S$ and every morphism $S\to\sX$, the morphism $\sX'\times_{\sX}S\to S$ has a section fpqc-locally on $S$). Then we have an exact sequence of sets
\[
\Alg (H)\to\Alg(H')\rightrightarrows\Alg(H''),
\]
where $H'$ and $H''$ are the pullbacks of $H$ to $\sX'$ and $\sX'\times_{\sX}\sX'$, respectively. In particular, the map $\Alg (H)\to\Alg(H')$ is injective.

\subsubsection{Good news}   \label{sss:Good news} 
(i) By Theorem~\ref{t:H_Sigma}(ii), $F^*H_\Sigma =H_\Sigma (-\Delta_0)$. On the other hand, $H_\Sigma (-\Delta_0)$ has a canonical algebraization constructed in \S\ref{sss:refined rescaling} ``by pure thought''. Thus we get a canonical element $\alpha\in\Alg (F^*H_\Sigma)$.

(ii) The morphism $F:\Sigma\to\Sigma$ satisfies the condition of \S\ref{sss:sheaf property of Alg} (because $F:W\to W$ is faithfully flat).
So the canonical map $\Alg (H_\Sigma )\to\Alg (F^*H_\Sigma )$ is injective.
Thus $\alpha$ comes from at most one algebraization of $H_\Sigma$.

\begin{conj}    \label{c:H_Sigma^alg}
Such an algebraization of $H_\Sigma$ exists.
\end{conj}

The conjectural algebraization of $H_\Sigma$ will be denoted by $H_\Sigma^{\alg}$. 

\subsubsection{Evidence in favor of Conjecture~\ref{c:H_Sigma^alg}}   \label{sss:evidence}
Even though $H_\Sigma^{\alg}$ is conjectural, the corresponding algebraizations of $H_Q$ and $H_{\Delta_0}$ are \emph{unconditional,} as explained below.

(i) Let $\alpha$ be as in \S\ref{sss:Good news}(i). Then the image of $\alpha$ in $\Alg (F^*H_Q)$ comes from a (unique) element $\beta\in\Alg (H_Q)$, namely the one described in \S\ref{sss:H_Q^alg}.

(ii) Let $\beta_0\in \Alg (H_{(\Delta_0)_Q})$ be the image of $\beta$. Using the explicit description of $\beta_0$ from  \S\ref{sss:restriction of H_Q^alg to Delta_0_Q}, one can check that $\beta_0$ comes from a (unique) algebraization $H_{\Delta_0}^{\alg}$ of $H_{\Delta_0}$. Namely, while $H_{\Delta_0}$ is the formal completion of a certain line bundle $\sM^*$ over $\Delta_0$ (see Proposition~\ref{p:G_Delta_0}), $H_{\Delta_0}^{\alg}$ is a $(\BZ/p\BZ)$-covering of $\sM^*$.

\section{Generalities on formal groups and their Cartier duals}  \label{s:formal groups}
\subsection{The notion of based formal $S$-polydisk}  \label{ss:def of Polyd}
\subsubsection{Notation}
If $S$ is a scheme then the formal completion of $\BA^n_S:=\BA^n\times S$ along its zero section will be denoted by $\hat\BA^n_S$.

\subsubsection{Definition}
Let $S$ be a scheme. Let $X$ be a formal $S$-scheme and $\sigma :S\to X$ a section. We say that $(X,\sigma)$ is a \emph{based formal $S$-polydisk} if Zariski-locally on $S$ there exists an $S$-isomorphism $(X,\sigma)\iso (\hat\BA^n_S ,0)$ for some $n\in\BZ_+$; here $0:S\to \hat\BA^n_S$ is the zero section.

\subsubsection{Notation}
The category of based formal $S$-polydisks will be denoted by $\Polyd (S)$. For fixed  $n\in\BZ_+$, let $\Polyd_n (S)\subset \Polyd (S)$ be the full subcategory of based formal $S$-polydisks of dimension~$n$ (i.e., locally isomorphic to  $(\hat\BA^n_S ,0)$).

\subsubsection{Automorphisms of $(\hat\BA^n_S ,0)$}
The functor that to a scheme $S$ associates the group of $S$-automorphisms of $(\hat\BA^n_S ,0)$ is representable by an affine group scheme $\sD_n$ over $\BZ$.

\begin{lem}
The underlying groupoid of $\Polyd_n (S)$ is canonically equivalent to that of $\sD_n$-torsors on $S$.
\end{lem}

\begin{proof}
It suffices to show that any $\sD_n$-torsor on $S$ is Zariski-locally trivial. Indeed, $\sD_n$ can be represented as a projective limit of a diagram of group schemes
\[
\ldots \to G_2\to G_1\to G_0=GL(n)
\]
in which all morphisms are faithfully flat and for each $n$ the group scheme $\Ker (G_{n+1}\to G_n)$ is isomorphic to a power of $\BG_a$.
\end{proof}

\begin{cor}  \label{c:fpqc-stack}
The assignment $S\mapsto \Polyd (S)$ is a stack for the fpqc topology (not merely the Zariski topology). \qed
\end{cor}

\subsection{Formal groups}   \label{ss:def of 1-dim formal group}
Let $\fF (S)$ (resp.~$\fF_n (S)$) be the category of group objects in $\Polyd (S)$ (resp.~in $\Polyd_n (S)$).
Objects of $\fF (S)$ will be called \emph{formal groups} over $S$; in other words,  by a formal group over $S$ we mean a group object $H$ in the category of formal $S$-schemes 
such that the pair $(H,e:S\to H)$ is a based formal $S$-polydisk. Objects of $\fF_n (S)$ are called $n$-dimensional formal groups.

\subsection{Cartier duals of commutative formal groups}
\begin{lem}
Let $S$ be a scheme and $H\in\fF^{\com} (S)$. Then the Cartier dual $H^*$ exists as a flat affine group scheme over $S$. Moreover, $H^*=\Spec\cA$, where the quasi-coherent $\cO_S$-algebra $\cA$ is locally free as an $\cO_S$-module.
\end{lem}

\begin{proof}
We can assume that $S$ is affine and that the based formal polydisk $(H,e:S\to H)$ is isomorphic to $(\hat\BA^n_S,0)$. Let $A:=H^0(S,\cO_S)$. 
 Then the coordinate ring of $H$ (viewed as a topological $A$-module)  is the dual of a free $A$-module. The lemma follows.
\end{proof}

\subsubsection{Notation}
Let $\fF^*(S)$ be the full subcategory of the category of group $S$-schemes formed by Cartier duals of commutative formal groups over $S$.

\subsubsection{Remarks} \label{sss:fpqc-locality}
(i) The assignments $S\mapsto \fF (S)$ and  $S\mapsto \fF^*(S)$ are stacks for the fpqc topology (not merely the Zariski topology). This follows from Corollary~\ref{c:fpqc-stack}.

(ii) By Corollary~\ref{c:fpqc-stack} and the previous remark, if $S$ is a fpqc-stack rather than a scheme one can still talk about $\Polyd (S )$, $\fF (S )$, and  $\fF^*(S )$.

\begin{prop}   \label{p:deforming formal groups}
Let $S$ be a scheme and $S_0\subset S$ a closed subscheme whose ideal is nilpotent. Let $G$ be a flat commutative group scheme over $S$ such that $G\times_SS_0\in\fF^*(S_0)$. Then
$G\in \fF^*(S)$.
\end{prop}

As pointed out by the reviewer, the above proposition appears as Lemma 1.1.21 in J.~Lurie's work \cite{Lu}. Moreover,  the Cartier duals of  commutative formal groups play an important role in \cite{Lu}

\begin{proof}
We can assume that $S=\Spec A$ and $S_0=\Spec A_0$, where $A_0=A/I$ and $I^2=0$. We can also assume the existence of an isomorphism of based formal $S_0$-polydisks
\[
(G_0^*,e)\iso (\hat\BA^n_S,0),
\]
where $G_0^*$ is the Cartier dual of $G_0$. To simplify notation, we will assume that $n=1$.

$G_0$ is affine because $G_0\in\fF^*(S_0)$. So $G$ is affine. Let $B$ be the coordinate ring of $G$ and $B_0:=B\otimes_AA_0$. Then $G=\Spec B$ and $G_0=\Spec B_0$.

Let $B^*:=\Hom_A(B,A)$ and $B_0^*:=\Hom_{A_0}(B_0,A_0)$ be the dual modules. We equip them with the weak topology. The coproduct in $B$ and $B_0$ yields a topological  algebra structure on $B^*$ and~$B_0^*$.

By assumption, we have an isomorphism of based formal $S_0$-disks $(G_0^*,e)\iso (\hat\BA^1_S,0)$. It induces an isomorphism of topological algebras $f_0:A_0[[x]]\iso B_0^*$ such that $l_0(1)=0$, where $l_0:=f_0(x)\in B_0^*$ and $1\in B_0$ is the unit. We will lift it to an isomorphism $f:A[[x]]\iso B$ such that $l(1)=0$, where $l=f(x)\in B^*$. This will show that $\Spf B^*$ is a formal group over $S=\Spec A$, whose Cartier dual is $G$.

The $A_0$-module $B_0$ is free because $f_0^*$ identifies $B_0$  with the topological dual $(A_0[[x]])^*$, which is a free $A_0$-module. By assumption, $B$ is flat over $\Spec A$. So $B$ is a free $A$-module. Therefore we can lift $l_0$ to an element $l\in B^*$. Moreover, adding to $l$ a multiple of the counit of $B$, we can achieve the equality $l(1)=0$.

Let us prove that $l^n\to 0$. The problem is to show that for every $b\in B$ we have $l^n(b)=0$ for big enough $n$. Let $F\subset B$ be a finitely generated $A$-submodule such that the coproduct $\Delta :B\to B\otimes_AB$ takes $b$ to $\im (F\otimes_AF\to B\otimes_AB)$. Since $l_0^n\to 0$, there exists $m\in\BN$ such that for $n\ge m$ one has $l^n(F)\subset I:=\Ker (B\to B_0)$. Then for $n\ge 2m$ one has $l^n(b)=(l^m\otimes l^{n-m})(\Delta (b))\in I^2=0$.

Since $l^n\to 0$, there is a homomorphism of topological $A$-algebras $f:A[[x]]\to B$ such that $f(x)=l$. The dual map $f^*:B^*\to (A[[x]])^*$ is a homomorphism of free $A$-modules inducing an isomorphism modulo $I$. Therefore $f^*$ is an isomorphism, and so is $f$.
\end{proof}

\subsection{Rescaling formal groups}  \label{ss:Rescaling}
\subsubsection{The monoidal category $\sM (S)$}   \label{sss:sM (S)}
Given  a scheme $S$, let $\sM (S)$ be the category of pairs $(\sL, \alpha:\sL\to\cO_S)$, where $\sL$ is an invertible $\cO_S$-module; this is a monoidal category with respect to tensor product.

Let $\sM_{\inj} (S)\subset\sM (S)$ be the full monoidal subcategory of pairs $(\sL,\alpha)$ such that $\Ker\alpha=0$. In fact, the category $\sM_{\inj} (S)$ is an ordered set, which identifies  with the set $\Div_+(S)$ of effective Cartier divisors on $S$ equipped with the ordering opposite to the usual one: the invertible subsheaf of $\sL\subset\cO_S$ corresponding to $\Div_+(S)$ is  $\cO_S(-D)$. Moreover, the tensor product in $\sM_{\inj}(S)$ corresponds to addition in $\Div_+(S)$. For this reason, objects of $\sM (S)$ are called \emph{generalized Cartier divisors} in \cite{BL}.

Let $\sM_{\nilp} (S)\subset\sM (S)$ be the full subcategory of pairs $(\sL,\alpha)$ such that $\alpha$ vanishes on~$S_{\red}\,$. One has 
$\sM_{\nilp} (S)\cap\sM_{\inj} (S)=\emptyset$.

The assignment $S\mapsto\sM (S)$ is an fpqc-stack of monoidal categories\footnote{The same stack is introduced in \cite{Prismatization}, where it is denoted by 
$(\BA^1/\BG_m)_-\,$.}. For any $f:S'\to S$ one has $f^*(\sM_{\nilp} (S))\subset\sM_{\nilp} (S')$; if $f$ is flat then $f^*(\sM_{\inj} (S))\subset\sM_{\inj} (S')$.

\subsubsection{Remark}   \label{sss:1 is final}
The unit object of $\sM (S)$ is a final object.

\subsubsection{Goal}
We have the stack of monoidal categories $\sM$ from \S\ref{sss:sM (S)}. In \S\ref{sss:Action of sM on Polyd and sF} we will define an action of $\sM$ on $\Polyd$ and on $\fF$, where $\Polyd$ is the stack of based formal polydisks (see~\S\ref{ss:def of Polyd}) and $\fF$ is the stack of formal groups (see ~\S\ref{ss:def of 1-dim formal group}). 

\subsubsection{The prestacks $\Polyd_{\pre}$ and $\sM_{\pre}$} \label{sss:The prestacks}
Let $\Polyd_{\pre}(S)\subset \Polyd (S)$ be the full subcategory formed by formal schemes $\hat\BA^n_S\,$. Then $\Polyd_{\pre}$ is a prestack of categories such that the associated fpqc-stack is $\Polyd$.

Let $\sM_{\pre}(S)\subset\sM (S)$ be the full subcategory of pairs $(\sL ,\alpha )$ with $\sL=\cO_S$.  Then $\sM_{\pre}$ is a prestack of monoidal categories such that the associated fpqc-stack is $\sM$. Explicitly, $\Ob\sM_{\pre}(S)=H^0(S,\cO_S)$, a morphism from $\alpha\in H^0(S,\cO_S)$ to $\alpha'\in H^0(S,\cO_S)$ is a presentation of $\alpha$ as $\alpha'\alpha''$, one has $\alpha_1\otimes\alpha_2=\alpha_1\alpha_2$, and so on. In other words, $\sM_{\pre}(S)$ is obtained as follows: start with the multiplicative monoid $H^0(S,\cO_S)$ viewed as a discrete monoidal category, then add morphisms $\psi_\alpha :\alpha\to 1$, subject to the relations $\psi_{\alpha_1\alpha_2}=\psi_{\alpha_1}\otimes\psi_{\alpha_2}\,$.

\subsubsection{Action of $\sM_{\pre}$ on $\Polyd_{\pre}$}  \label{sss:Action of H^0(S,cO_S)}
(i) First, let us define a strict action of the multiplicative monoid $H^0(S,\cO_S)$ on the category $\Polyd_{\pre}(S)$, which is trivial at the level of objects of $\Polyd_{\pre}(S)$. To this end, note that a morphism $\hat\BA^m_S\to\hat\BA^n_S$ is just a  collection
\[
(f_1,\ldots , f_n), \quad f_i\in H^0(S,\cO_S)[[x_1,\ldots x_m]], \quad f_i(0)=0.
\]
Definition: $\alpha\in H^0(S,\cO_S)$ takes $(f_1,\ldots , f_n)$ to $(\tilde f_1,\ldots , \tilde f_n)$, where
\begin{equation}  \label{e:action of alpha on morphisms}
\tilde f_i(x_1,\ldots x_m):=\alpha^{-1} f_i(\alpha x_1,\ldots \alpha x_m).
\end{equation}
The r.h.s. of \eqref{e:action of alpha on morphisms} makes sense (even though $\alpha^{-1}$ is not assumed to exist) because $f_i(0)=0$.

(ii) Let $\Phi_\alpha :\Polyd_{\pre} (S)\to\Polyd_{\pre} (S)$ be the functor corresponding to $\alpha\in H^0(S,\cO_S)$. 
The explicit description of $\sM_{\pre}(S)$ (see \S\ref{sss:The prestacks}) shows that extending the above action of $H^0(S,\cO_S)$ on $\Polyd_{\pre}(S)$ to an action of
$\sM_{\pre} (S)$ amounts to specifying natural transformations $\psi_\alpha :\Phi_\alpha\to\Id$ so that 
$\psi_{\alpha_1\alpha_2}=\psi_{\alpha_1}\circ\Phi_{\alpha_1}(\psi_{\alpha_2})$. We define the morphism 
$\hat\BA^n_S=\Phi_\alpha (\hat\BA^n_S)\overset{\psi_\alpha}\longrightarrow\hat\BA^n_S$ to be multiplication by $\alpha$.

\subsubsection{Action of $\sM$ on $\Polyd$ and $\fF$}  \label{sss:Action of sM on Polyd and sF}
(i) In \S\ref{sss:Action of H^0(S,cO_S)} we defined an action of $\sM_{\pre}$ on $\Polyd_{\pre}\,$. It
induces an action of $\sM$ on $\Polyd$.

(ii) The endofunctor of $\Polyd (S)$ corresponding to each object of $\sM (S)$ preserves finite products. So $\Polyd (S)$ acts on the category of group objects in $\Polyd (S)$, i.e., on
$\fF (S)$.

\begin{lem}    \label{l:refined rescaling}
Let $\sM_{\nilp} (S)$ be as in \S\ref{sss:sM (S)} and $(\sL,\alpha)\in\sM_{\nilp} (S)$. 
Then the rescaling functor $\Phi_{\sL,\alpha}:\fF (S)\to\fF (S)$ canonically factors as
\begin{equation}  \label{e:refined rescaling}
\fF (S)\to \Aff (S)\to\fF (S),
\end{equation}
where $\Aff (S)$ is the category of smooth affine group $S$-schemes with connected fibers and the second arrow in \eqref{e:refined rescaling} is the functor of formal completion along the unit. Moreover, if $G$ is in the essential image of the functor  $\fF (S)\to \Aff (S)$ then Zariski-locally on $S$, the pointed $S$-scheme $(G,0)$ is isomorphic to $(\BA^m_S,0)$ for some $m$.
\end{lem}

\begin{proof}
If in the situation of \S\ref{sss:Action of H^0(S,cO_S)} the function $\alpha\in H^0(S,\cO_S)$ is nilpotent then the formal series \eqref{e:action of alpha on morphisms} is a polynomial.
\end{proof}

\subsubsection{Notation}
Recall that $\sM (S)\supset \sM_{\inj}(S)=\Div_+(S)$ (see \S\ref{sss:sM (S)}).
If $D\in\Div_+(S)$ then the  action of $D$ on $\Polyd (S)$ or $\fF (S)$ will be denoted by $\sX\mapsto \sX (-D)$.
By \S\ref{sss:1 is final}, we have a canonical morphism $\sX (-D)\to\sX$.

\begin{lem}  \label{l:univ property of sX(-D)}
Let $S$ be a scheme and $D\overset{i}\mono S$ an effective Cartier divisor.

(i) For any $\sX,\sX'\in\Polyd (S)$, the map $\Mor (\sX' ,\sX (-D))\to \Mor (\sX' ,\sX)$ is injective. Its image is equal to the preimage of the distinguished element\,\footnote{This element is due to the fact that we are dealing with \emph{based} formal $S$-polydisks.} of 
 $\Mor (i^*\sX' ,i^*\sX)$.
 
(ii) The same is true if $\sX,\sX'$ are formal groups over $S$. \qed
 \end{lem}
 
 \subsubsection{Remark} \label{sss:generalization of the lemma}
 Lemma~\ref{l:univ property of sX(-D)}(i) can be generalized as follows. Let $(\sL ,\alpha)\in\sM (S)$. 
 Let $D\overset{i}\mono S$ be the closed subscheme corresponding to the ideal $\im\alpha\subset\cO_S$ ; let $S'\overset{\nu}\mono S$ be the closed subscheme corresponding to
 the ideal $\Ker (\alpha^*:\cO_S\to\sL^*)$. Let $\sX,\sX'\in\Polyd (S)$, and
 let $\tilde\sX\in\sM (S)$ be obtained by acting on $\sX$ by $(\sL ,\alpha)$. By \S\ref{sss:1 is final}, we have a canonical morphism $\tilde\sX\to\sX$ and therefore a morphism
$f:\MMor (\sX', \tilde\sX )\to \MMor (\sX' ,\sX)$,
 where $\MMor$ denotes the sheaf on $S$ formed by morphisms. Then the sequence
 \[
 \MMor (\sX', \tilde\sX )\overset{f}\longrightarrow \MMor (\sX' ,\sX)\to i_*\MMor (i^*\sX' ,i^*\sX)
 \]
 is exact in the following sense: the sections of $\im f$ are precisely those sections of $\MMor (\sX' ,\sX)$ which map to the distinguished section of $ i_*\MMor (i^*\sX' ,i^*\sX)$.
 Moreover, 
  $$\im f=\nu_*\MMor (\nu^*\sX', \nu^*\tilde\sX )$$ 
 (the two sheaves are equal as quotients of $\MMor (\sX', \tilde\sX )$).

\subsubsection{What if $S$ is a stack?}  \label{sss:What if S is a stack?}
It is straightforward to generalize the material of \S\ref{sss:sM (S)}-\ref{sss:generalization of the lemma} to the situation where $S$ is an algebraic stack\footnote{Let us note that the definition of algebraic stack from \cite[\S 2.4]{Prismatization} involves no finiteness conditions.} of groupoids in 
the sense of \cite[\S 2.4]{Prismatization}. But algebraic stacks are not enough for us: the stack $\Sigma$ and the $q$-de Rham prism $Q$ are \emph{formal} stacks rather than agebraic ones.

If $S$ is any fpqc-stack we still have the monoidal category $\sM (S)$ and its action on $\Polyd (S)$ and $\fF (S)$. 
For a reasonable class of stacks $S$ (which includes all formal stacks, e.g., $\Sigma$, $Q$, and $Q\times_\Sigma Q$) one also has a good notion of effective Cartier divisor on $S$ and an analog of Lemma~\ref{l:univ property of sX(-D)}, see \S\ref{ss:What if S is a stack} below.

\subsection{Deformation of a formal group to the formal completion of its Lie algebra} \label{ss:Deformation to Lie algebra}

In this subsection we briefly discuss a formal version of a particular case of the Fulton-MacPherson construction of \emph{deformation to the normal cone}, see 
 \cite[Ch.~5]{F}, \cite[\S 2]{V}, and also \S 10 of the article \cite{R} (where some generalizations of the original construction are discussed).

\subsubsection{}  \label{sss:deformation to normal cone}
Let $\sX\in\Polyd (S)$, $\sX= (X,\sigma :S\to X)$. Let $\sN$ be the $\sigma$-pullback of the tangent bundle of $X$ relative to $S$ (or equivalently, 
the normal bundle of $\sigma (S)\subset X$). Let $\pi:\BA^1_S\to S$ be the projection and $i_0:S\to\BA^1_S$ the zero section. Let $D:=i_0(S)\subset\BA^1_S\,$. Let 
$$\tilde\sX:=(\pi^*\sX)(-D)\in\Polyd (\BA^1_S).$$
One checks that $i_0^*\tilde\sX$ canonically identifies with the formal completion of the vector bundle $\sN$ along its zero section.

\subsubsection{}
Now let $\sX\in\sF (S)$.  Just as in \S\ref{sss:deformation to normal cone}, let $\tilde\sX:=(\pi^*\sX)(-D)\in\sF (\BA^1_S)$.
Then the formal group $i_0^*\tilde\sX$ canonically identifies with the formal completion of the vector bundle $\Lie (\sX )$ along its zero section.

\subsection{An analog of Lemma~\ref{l:univ property of sX(-D)} if $S$ is a stack}  \label{ss:What if S is a stack}
\subsubsection{A class of stacks}  \label{sss:A class of stacks}

Let $S$ be an  fpqc-stack of groupoids which can be represented as
\begin{equation}  \label{e:S as colimit}
S= \underset{\longrightarrow}{\lim}(S_1\mono S_2\mono\ldots ),
\end{equation}
where each $S_i$ is an algebraic stack in the sense of \cite[\S 2.4]{Prismatization} and the morphisms $S_i\to S_{i+1}$ are closed immersions.
 Such $S$ is \emph{pre-algebraic} in the sense of \cite[\S 2.3]{Prismatization}.
 
\subsubsection{The notion of effective Cartier divisor}
We will use the notion of effective Cartier divisor on  a pre-algebraic stack introduced in  \cite[\S 2.10-2,11]{Prismatization}.
If $S$ admits a presentation~\eqref{e:S as colimit} the definition from  \cite{Prismatization} is equivalent to the following one: an \emph{effective Cartier divisor} on $S$ is a closed substack $D\subset S$ such that

(i) the ideal $\cI_n$ of the closed substack $D\cap S_n\subset S_n$ is an invertible sheaf on some closed substack $S'_n\subset S_n$;

(ii) the inductive limit\footnote{It is easy to check that $S'_n\subset S'_{n+1}$, so the stacks $S'_n$ form an inductive system.} of the stacks $S'_n$ equal $S$; 
equivalently, for every quasi-compact scheme $\tilde S$ every morphism $f:\tilde S\to S$ factors through some $S'_n$.

In this situation one can define the  line bundle $\cO_S (-D )$:  its pullback to $S'_n$ equals $\cI_n\,$. Therefore we have $\sX (-D)$ for $\sX\in\Polyd (S)$ or for $\sX\in\fF (S)$.

\begin{prop}   \label{p:univ property of sX(-D)} 
Lemma~\ref{l:univ property of sX(-D)} remains valid for any stack $S$ which admits a presentation~\eqref{e:S as colimit}.
\end{prop}

\begin{proof}
It suffices to prove the analog of Lemma~\ref{l:univ property of sX(-D)}(i) for the stack $S$. To this end,
for each~$n$ apply the analog of \S\ref{sss:generalization of the lemma} for algebraic stacks to the pullback of $\cO_S (-D)$ to $S_n\,$.
\end{proof}

The author expects that using \S\ref{sss:generalization of the lemma}
one can prove  that Lemma~\ref{l:univ property of sX(-D)} remains valid for any \emph{pre-algebraic} stack in the sense of \cite[\S 2.3]{Prismatization}.

\subsubsection{A corollary of Lemma~\ref{l:refined rescaling}}   \label{sss:refined rescaling}
Let $S$ be as in \S\ref{sss:A class of stacks} and $H\in\fF (S)$. Let $D\subset S$ be an effective Cartier divisor such that for every scheme $T$ and every morphism
$T\to S$ one has $T\times_SD\supset T_{\red}$. Lemma~\ref{l:refined rescaling} implies that in this situation the formal group $H(-D)$ can be canonically represented as a formal completion of a smooth affine group $S$-scheme with connected fibers. We denote this group scheme by $H(-D)^{\alg}$.

In particular, we have the group scheme $H_{\Sigma}(-\Delta_0)^{\alg}$ over $\Sigma$.

\section{Proofs of the statements from \S\ref{s:main results}} \label{s:Proofs}
\subsection{Recollections on the Hodge-Tate divisor $\Delta_0\subset\Sigma$}  \label{ss:Delta_0 as classifying stack}
By definition, 
$\Delta_0\subset\Sigma:=W_{\prim}/W^\times$ is the preimage of  $\{ 0\}/\BG_m\subset\BA^1/\BG_m$ under the morphism
$W_{\prim}/W^\times\to\BA^1/\BG_m$.

The element $V(1)\in W(\BZ_p)$ defines a morphism 
\begin{equation} \label{e:eta}
\eta :\Spf\BZ_p\to\Delta_0.
\end{equation}
$\eta$ is faithfully flat, and it identifies $\Delta_0$ with the classifying stack $(\Spf\BZ_p)/(W^\times)^{(F)}$, where $(W^\times)^{(F)}:=\Ker (F:W^\times\to W^\times)$; the proof of this fact is straightforward (see \cite{BL} or Lemma~4.5.2 of \cite{Prismatization}).

\subsection{Proof of Proposition~\ref{p:G' to W^times/G_m faithfully flat}}  \label{ss:G' to W^times/G_m faithfully flat}
We have to show that the composite morphism
\begin{equation}  \label{e: 2 G' to W^times/G_m}
G'_{\Sigma}\to W^\times_\Sigma\to(W^\times/\BG_m)_\Sigma
\end{equation}
is faithfully flat.

\subsubsection{Reductions}   \label{sss:Reductions}
Both $G'_{\Sigma}$ and $(W^\times/\BG_m)_\Sigma$ are flat over $\Sigma$. For any morphism from a quasi-compact scheme $S$ to $\Sigma$, the ideal of
the closed subscheme $S\times_\Sigma\Delta_0\subset S$ is nilpotent. So it suffices to check faithful flatness of \eqref{e: 2 G' to W^times/G_m} after base change to $\Delta_0$ and even after further pullback via the faithfully flat morphism \eqref{e:eta}.

\subsubsection{Pullback via $\eta: \Spf\BZ_p\to\Sigma$}    \label{sss:eta-pullback}
Let $G'_\eta$ be the pullback of $G'_{\Sigma}$ via $\eta: \Spf\BZ_p\to\Delta_0\subset\Sigma$. By \S\ref{sss:G'_W_prim},
$G'_\eta=W_{\Spf\BZ_p}$ (disregarding the group operation), and the $\eta$-pullback of \eqref{e: 2 G' to W^times/G_m} is the map
\begin{equation}  \label{e:G'_eta to G_m}
W_{\Spf\BZ_p}\to (W^\times/\BG_m)_{\Spf\BZ_p}=\Ker (W^\times\epi\BG_m)_{\Spf\BZ_p}, \quad x\mapsto 1+V(1)\cdot x=1+V(Fx).
\end{equation}
This map is faithfully flat because $F:W\to W$ is a Frobenius lift. \qed

\subsection{The group schemes $G_\eta,G_{\Delta_0}$ and the proof of Theorem~\ref{t:dual to formal group}}   \label{ss:G_eta}
Let $G_{\Delta_0}$ be the pullback of $G_{\Sigma}$ to $\Delta_0$. 
Let $G_\eta$ be the pullback of $G_{\Sigma}$ via $\eta: \Spf\BZ_p\to\Delta_0\subset\Sigma$; this is a group scheme over $\Spf\BZ_p$. 
\begin{prop}    \label{p:G_eta}
(i) There is a canonical isomorphism of group schemes
\begin{equation}    \label{e:G_eta=W^{F}}
G_\eta\iso (W^{(F)})_{\Spf\BZ_p},
\end{equation}
where $W^{(F)}:=\Ker (F:W\to W)$.

(ii) The homomorphism $G_\eta\to (\BG_m)_{\Spf\BZ_p}$ induced by \eqref{e: G to G_m} is trivial.

(iii) The $\eta$-pullback of the morphism $G_\Sigma\to F^*G_\Sigma$ from \S\ref{sss:Pieces of structure on H_Sigma} is trivial.
\end{prop}

\begin{proof}
Let us prove (i). By \S\ref{sss:G'_W_prim}, $G'_\eta$ is $W_{\Spf\BZ_p}$ equipped with the group 
operation 
$$(x_1,x_2)\mapsto x_1+x_2+V(1)\cdot x_1x_2\, .$$
By \eqref{e:G'_eta to G_m}, the subgroup $G_\eta\subset G'_\eta$ is defined by the equation $V(1)\cdot x=0$ or equivalently, $Fx=0$.

Statement (ii) is clear because the  homomorphism $G_\eta\to (\BG_m)_{\Spf\BZ_p}$ is the restriction of the map $G'_\eta =W_{\Spf\BZ_p}\to W^{\times}_{\Spf\BZ_p}$ given by $x\mapsto 1+V(1)\cdot x$.

To prove (iii), note that the morphism in question is $x\mapsto Fx$, but we already know that $G_\eta\subset G'_\eta$ is defined by the equation $Fx=0$.
\end{proof}

\begin{lem}   \label{l:W^(F)=G_a^sharp}
The canonical homomorphism $W\epi W_1=\BG_a$ induces an isomorphism 
$$W^{(F)}_{\Spec\BZ_{(p)}}\iso (\BG_a^\sharp)_{\Spec\BZ_{(p)}}\, ,$$
where $\BG_m^\sharp$ is the divided powers additive group and $\BZ_{(p)}$ is the localization of $\BZ$ at $p$.
\end{lem}

For a proof of the lemma, see \cite{BL} or \cite[Lemma~3.2.6]{Prismatization}.

\begin{cor}   \label{c:dual to formal group}
$G_\eta=(\BG_a^\sharp)_{\Spf\BZ_p}$, so $G_\eta$ is Cartier dual to the formal group~$(\hat\BG_a)_{\Spf\BZ_p}\,$.
\end{cor}

\begin{proof}
Follows from Proposition~\ref{p:G_eta} and Lemma~\ref{l:W^(F)=G_a^sharp}.
\end{proof}

\subsubsection{Proof of Theorem~\ref{t:dual to formal group}}  \label{sss:dual to formal group proof}
Corollary~\ref{c:dual to formal group} and Proposition~\ref{p:deforming formal groups} imply (similarly to \S\ref{sss:Reductions}) that $G_\Sigma$ is the Cartier dual of some 1-dimensional formal group over~$\Sigma$ (which is denoted by $H_\Sigma$). \qed

\subsubsection{Proof of Proposition~\ref{p:G_Delta_0}}  \label{sss:G_Delta_0 proof}
We have to construct an isomorphism $G_{\Delta_0}\iso\sM^\sharp$, where $\sM$ is the conormal bundle of $\Delta_0\subset\Sigma$.
Corollary~\ref{c:dual to formal group} provides an isomorphism $f:G_\eta\iso\eta^*\sM^\sharp$. By \S\ref{ss:Delta_0 as classifying stack}, $(W^\times)^{(F)}$ acts on $G_\eta$ and $\eta^*\sM^\sharp$, and the  problem is to check that $f$ is $(W^\times)^{(F)}$-equivariant. Indeed, $u\in (W^\times)^{(F)}$ acts on $G_\eta=W^{(F)}$ as multiplication by $u$, and it acts on $\eta^*\sM^\sharp=\BG_a^\sharp$ as multiplication by the $0$-th component of the Witt vector~$u$. \qed

\medskip

We can now prove the following weaker version of Theorem~\ref{t:H_Sigma}(i).

\begin{cor}   \label{c:s vanishes on Delta_0}
The section $s:\Sigma\to H_\Sigma$ vanishes on $\Delta_0$.
\end{cor}

\begin{proof}
As already mentioned in \S\ref{ss:Delta_0 as classifying stack}, $\eta :\Spf\BZ_p\to\Delta_0$ is faithfully flat. So Proposition~\ref{p:G_eta}(ii) implies that the canonical homomorphism 
$G_\Sigma\to (\BG_m)_\Sigma$ vanishes on $\Delta_0$. By the definition of $s$ (see \S\ref{sss:structures on H_Q}), this means that $s:\Sigma\to H_\Sigma$ vanishes on $\Delta_0$.
\end{proof}

\subsection{The ``de Rham pullback'' of $G_\Sigma$ and the proof of Proposition~\ref{p:G_dR}}   \label{ss:G_dR}
\subsubsection{Recollections}  \label{sss:G_dR}
Recall that $G_{\dR}:=\rho_{\dR}^*G_\Sigma$, where $\rho_{\dR}:\Spf\BZ_p\to\Sigma$ comes from the element $p\in W(\BZ_p )$.  
For any $p$-nilpotent ring $A$ we have
\begin{equation}  \label{e:def of G_dR}
G_{\dR}(A):=\{ x\in W(A)\,|\, 1+px\in A^\times\subset W(A)^\times\},
\end{equation}
where $A^\times\subset W(A)^\times$ is the image of the Teichm\"uller embedding, and the group operation on $G_{\dR}(A)$ is given by
\begin{equation}  \label{e:p-rescaled G_m}
(x_1,x_2)\mapsto x_1+x_2+px_1x_2.
\end{equation}

\subsubsection{The homomorphisms $f:G_{\dR}\to W_{\Spf\BZ_p}$ and $f_0:G_{\dR}\to (\BG_a)_{\Spf\BZ_p}\,$} \label{sss:f&f_0}
The coefficients of the formal series
\begin{equation}   \label{e:the series f}
f(x):=p^{-1}\log (1+px)=\sum_{n=1}^\infty \frac{(-p)^{n-1}}{n}x^n \, .
\end{equation}
belong to $\BZ_p$ and converge to $0$. One has
\[
f(x_1+x_2+px_1x_2)=f(x_1)+f(x_2).
\]
So the series \eqref{e:the series f} defines a group homomorphism $f:G_{\dR}\to W_{\Spf\BZ_p}\,$. Composing it with the canonical homomorphism 
$W_{\Spf\BZ_p}\epi (W_1)_{\Spf\BZ_p}=(\BG_a)_{\Spf\BZ_p}\,$, we get a homomorphism 
\begin{equation}    \label{e:f_0}
f_0:G_{\dR}\to (\BG_a)_{\Spf\BZ_p}\,; 
\end{equation}
equivalently, $f_0(x)=p^{-1}\log (1+px_0)$, where $x_0$ is the $0$-th component of the Witt vector $x$.

By Lemma~\ref{l:W^(F)=G_a^sharp}, $(\BG_a^\sharp)_{\Spf\BZ_p}=W^{(F)}_{\Spf\BZ_p}$, where $W^{(F)}:=\Ker (F:W\to W)$. So Proposition~\ref{p:G_dR} is equivalent to the following one.

\begin{prop}  \label{p:2G_dR}
There exists an isomorphism $G_{\dR}\iso W^{(F)}_{\Spf\BZ_p}$ whose composition with the canonical homomorphism 
$W^{(F)}_{\Spf\BZ_p}\to (\BG_a)_{\Spf\BZ_p}$ equals \eqref{e:f_0}.
\end{prop}

Note that the isomorphism in question is unique (to see this, identify $W^{(F)}_{\Spf\BZ_p}$ with $(\BG_a^\sharp)_{\Spf\BZ_p}$).
In \S\ref{sss:G_dR=W^(F)} we will deduce Proposition~\ref{p:2G_dR} from the following lemma, which will be proved in \S\ref{ss:G_dR=W^{F=p}}.
 
\begin{lem}   \label{l:G_dR=W^{F=p}}
The homomorphism $f:G_{\dR}\to W_{\Spf\BZ_p}$ from  \S\ref{sss:f&f_0} induces an isomorphism
\begin{equation}   \label{e:G_dR & Fy=py}
G_{\dR}\iso W_{\Spf\BZ_p}^{F=p}, \quad \mbox{where \;}W_{\Spf\BZ_p}^{F=p}:=\{y\in W_{\Spf\BZ_p}\,|\, Fy=py\}.
\end{equation}
The inverse isomorphism is given by $y\mapsto g(y)$, where $g$ is the formal power series
\begin{equation}  \label{e: the power series g}
g(y):=\frac{\exp (py)-1}{p}=\sum_{n=1}^\infty \frac{p^{n-1}}{n!}y^n \, .
\end{equation}
\end{lem}

Note that if $A$ is a $p$-nilpotent ring and $y\in W(A)$ satisfies $Fy=py$ then $y$ is topologically nilpotent, so $h(y)$ makes sense for \emph{any} formal power series $h$ over $\BZ_p$. In particular, this is true for the power series \eqref{e: the power series g} (even though in the case $p=2$ its coefficients do not converge to $0$).

\subsubsection{Deducing Proposition~\ref{p:2G_dR} from Lemma~\ref{l:G_dR=W^{F=p}}}\label{sss:G_dR=W^(F)}
The equation $Fy=py$ from \eqref{e:G_dR & Fy=py} can be rewritten as $F(y-Vy)=0$. The operator $\id -V$ is invertible because $V$ is topologically nilpotent. So we get an isomorphism 
\begin{equation}   \label{e:1-V}
\id -V:W_{\Spf\BZ_p}^{F=p}\iso W^{(F)}_{\Spf\BZ_p}\, .
\end{equation}
Composing it with \eqref{e:G_dR & Fy=py}, we get an isomorphism
\begin{equation}   \label{e:G_dR=W^(F)}
G_{\dR}\iso W^{(F)}_{\Spf\BZ_p}\, ,
\end{equation}
which has the required property. \qed

\subsubsection{Warning about the base change of  \eqref{e:G_dR=W^(F)} to $\Spec\BF_p$}   \label{sss:somewhat unexpected}
In $W(\BF_p )$ we have 
$p=V(1)$. So by Proposition~\ref{p:G_eta}(i), the base change of $G_{\dR}$ to $\Spec\BF_p$ identifies with~$W^{(F)}_{\Spec\BF_p}$. The group scheme $W_{\Spec\BF_p}^{F=p}:=\{y\in W_{\Spec\BF_p}\,|\, Fy=py\}$ also equals $W^{(F)}_{\Spec\BF_p}$: indeed, if $A$ is an $\BF_p$-algebra and $y\in W(A)$ then
\[
Fy=py\Leftrightarrow (\id-V)Fy=0\Leftrightarrow Fy=0.
\]
We claim that the base change to $\Spec\BF_p$ of the isomorphism \eqref{e:G_dR & Fy=py} equals the identity (so the base change to $\Spec\BF_p$ of \eqref{e:G_dR=W^(F)} is $\id-V\ne\id$ !!). This follows from the next

\begin{lem}   \label{l:W^F in characteristic p}
Let $A$ be an $\BF_p$-algebra and $x\in W^{(F)}(A)$. Then $px=x^p=0$.
\end{lem}

\begin{proof}
We have $px=FVx=VFx=0$. Write $x=[x_0]+Vy$, where $x_0$ is the $0$-th coordinate of the Witt vector $x$. Then $x_0^p=0$ and $Fy=0$ (because in characteristic $p$ the Witt vector Frobenius equals the usual one). So $py=VFy=0$ and $(Vy)^2=V(py^2)=0$. Therefore $x^p=0$.
\end{proof}

\subsection{Proof of Lemma~\ref{l:G_dR=W^{F=p}}}   \label{ss:G_dR=W^{F=p}}
\subsubsection{} \label{sss:both are flat}
By Corollary~\ref{c:G_Sigma is flat}, $G_{\dR}$ is flat over $\Spf\BZ_p$. By \eqref{e:1-V} and Lemma~\ref{l:W^(F)=G_a^sharp}, $W_{\Spf\BZ_p}^{F=p}$ is also flat over 
$\Spf\BZ_p$.

\subsubsection{} \label{sss:f lands where needed}
Let us prove that the homomorphism $f:G_{\dR}\to W_{\Spf\BZ_p}$ from Lemma~\ref{l:G_dR=W^{F=p}} factors through $W_{\Spf\BZ_p}^{F=p}$. Since $f\circ F=F\circ f$, it suffices to show that for $x\in G_{\dR}(A)$ one has
\begin{equation}  \label{e:Fx=h(x)}
Fx=h(x),
\end{equation}
where $h:G_{\dR}\to G_{\dR}$ is raising to the power of $p$ in the sense of the operation \eqref{e:p-rescaled G_m}; explicitly,
\[
h(x)=\frac{(1+px)^p-1}{p}:=\sum_{i=1}^p \binom{p}{i}p^{i-1}x^i.
\]
Since $1+px\in A^\times\subset W(A)^\times$ we have $F(1+px)=(1+px)^p$, so
\begin{equation} \label{e:p(Fx-h(x))=0}
p(Fx-h(x))=0.
\end{equation}
But the coordinate ring $B$ of $G_{\dR}$ is flat over $\BZ_p$ (see \S\ref{sss:both are flat}), so $W(B)$ is also flat over $\BZ_p$. The elements $Fx-h(x)$ for all $p$-nilpotent rings $A$ and all $x\in G_{\dR}(A)$ define an element $u\in W(B)$, and by \eqref{e:p(Fx-h(x))=0} we have $pu=0$. So $u=0$, which proves \eqref{e:Fx=h(x)}.

\subsubsection{} \label{sss:g lands where needed}
The formal series \eqref{e: the power series g} defines a homomorphism $g:W_{\Spf\BZ_p}^{F=p}\to G'_{\dR}:=\rho_{\dR}^*G'_\Sigma$, where $G'_\Sigma$ is as in \S\ref{ss:G'_Sigma}. Let us prove that this homomorphism factors through $G_{\dR}\subset G'_{\dR}$.
The problem is to show that for any $p$-nilpotent ring $A$ and any $x\in W_{\Spf\BZ_p}^{F=p}(A)$ the Witt vector $1+pg(x)$ is Teichm\"uller. It is clear that
$$F(1+pg(x))=(1+pg(x))^p.$$
But  the coordinate ring $C$ of $W_{\Spf\BZ_p}^{F=p}$ is flat over $\BZ_p$  (see \S\ref{sss:both are flat}), so an element $u\in W(C)$ such that $Fu=u^p$ has to be Teichm\"uller. 

\subsubsection{} \label{sss:coda}
By \S\ref{sss:both are flat}, $G_{\dR}$ and $W_{\Spf\BZ_p}^{F=p}$ are flat over $\Spf\BZ_p$. 
The morphism $f:G_{\dR}\to W_{\Spf\BZ_p}^{F=p}$ becomes an isomorphism after base change to $\Spec\BF_p$, see \S\ref{sss:somewhat unexpected}. So $f$ itself is an isomorphism. Finally, it is easy to see that $f\circ g=\id$. \qed

\subsubsection{Remark}  \label{sss:using delta-ring structure}
In \S\ref{sss:f lands where needed}-\ref{sss:g lands where needed} we used a flatness argument. Instead, one could use the canonical $\delta$-ring structure on $W(A)$.

\subsection{Remarks related to \S\ref{ss:G_dR}}   \label{ss:discrepancies}
Let $(W^\times)^{(F)}:=\Ker (F:W^\times\to W^\times )$. In \S\ref{sss:A way to think about G_dR} we identify the group scheme $G_{\dR}$ from \S\ref{ss:G_dR} with 
$((W^\times)^{(F)}/\mu_p)_{\Spf\BZ_p}$. This allows us to think of \eqref{e:G_dR=W^(F)} as an isomorphism
\begin{equation}   \label{e:W^times^(F) and W^(F)}
((W^\times)^{(F)}/\mu_p)_{\Spf\BZ_p}\iso W^{(F)}_{\Spf\BZ_p}\, .
\end{equation}
In \S\ref{sss:unexpected reduction mod p} we show that the base change of \eqref{e:W^times^(F) and W^(F)} to $\Spec\BF_p$ is somewhat unexpected.
Related to this is Lemma~\ref{l:H_{T'}}, which says that the restriction of the formal group $H_Q$ to the subscheme $\Spf\BZ_p[[q-1]]/(q^p-1)\subset Q$ is somewhat unusual.

\subsubsection{}               \label{sss:A way to think about G_dR}
If $w\in (W^\times)^{(F)}(A)$ and $w_0\in A^\times$ is the $0$-th component of the Witt vector $w$ then $[w_0]/w=1+Vx$ for a unique $x\in W(A)$; moreover, $x\in G_{\dR}(A)$ because $$1+px=F([w_0]/w)=[w_0^p].$$ It is easy to check that one thus gets an isomorphism 
\begin{equation}    \label{e:G_dR=W^times/mu_p}
((W^\times)^{(F)}/\mu_p)_{\Spf\BZ_p}\iso G_{\dR}\, .
\end{equation}
Composing \eqref{e:G_dR=W^times/mu_p} and  \eqref{e:G_dR=W^(F)}, one gets an isomorphism \eqref{e:W^times^(F) and W^(F)}.

On the other hand, one has canonical isomorphisms
\[
(W^\times)^{(F)}_{\Spf\BZ_p}\iso (\BG_m^\sharp)_{\Spf\BZ_p},   \quad W^{(F)}_{\Spf\BZ_p}\iso (\BG_a^\sharp)_{\Spf\BZ_p},
\] 
where $\BG_m^\sharp$ and $\BG_a^\sharp$ are the divided powers versions of $\BG_m$ and $\BG_a$ (see \cite{BL} or Lemma~3.2.6 and \S 3.3.3 of 
\cite{Prismatization}). So one can think of \eqref{e:G_dR=W^times/mu_p} as an isomorphism $(\BG_m^\sharp/\mu_p)_{\Spf\BZ_p}\iso G_{\dR}\,$, and one can think of \eqref{e:W^times^(F) and W^(F)} as an isomorphism 
\begin{equation}  \label{e:G_m^sharp/mu_p=G_a^sharp}
(\BG_m^\sharp/\mu_p)_{\Spf\BZ_p}\iso (\BG_a^\sharp)_{\Spf\BZ_p}.
\end{equation}
It is easy to check that the isomorphism \eqref{e:G_m^sharp/mu_p=G_a^sharp} is equal to the isomorphism
$$\log: (\BG_m^\sharp/\mu_p)_{\Spf\BZ_p}\iso (\BG_a^\sharp)_{\Spf\BZ_p}$$ 
from Proposition~\ref{p:G_m^sharp sequence}(i) of Appendix~\ref{s:Cartier dual of G_m^sharp}.

\subsubsection{Warning}  \label{sss:unexpected reduction mod p}
In $W_{\BF_p}$ one has $VF=FV$. Using this, it is easy to check that the map $W_{\BF_p}\to W_{\BF_p}^\times$ defined by $x\mapsto 1+Vx$ induces an isomorphism
\[
f_{naive}:W^{(F)}_{\BF_p}\iso ((W^\times)^{(F)}/\mu_p)_{\BF_p}\, .
\]
On the other hand, let $f:W^{(F)}_{\BF_p}\iso ((W^\times)^{(F)}/\mu_p)_{\BF_p}$ be the base change of the inverse of \eqref{e:W^times^(F) and W^(F)} to $\Spec\BF_p$.
It turns out that $f\ne f_{naive}\,$; more precisely, using \S\ref{sss:somewhat unexpected} one gets
\begin{equation} \label{e:f_naive & f}
f(x)=f_{naive}(Vx-x).
\end{equation}

The remaining part of \S\ref{ss:discrepancies} is closely related to formula~\eqref{e:f_naive & f}.

\subsubsection{Notation}   \label{sss:T:=Spf B}
Let
\[
T:=\Spf B, \quad \mbox{where }B:=\{ (x,y)\in\BZ_p\times\BZ_p\,|\, x\equiv y\!\! \!  \mod p\}.
\]
The element $(p,V(1))\in W(\BZ_p)\times W(\BZ_p)=W(\BZ_p\times\BZ_p)$ belongs to $W(B)$. It defines a morphism $T\to W_{\prim}$ and therefore a morphism $T\to\Sigma$.
Let $G_T$ and $H_T$ be the pullbacks of $G_\Sigma$ and $H_\Sigma$ to $T$ (so $H_T$ is a formal group over $T$, and $G_T$ is the Cartier dual affine group scheme over $T$).

\begin{lem}  \label{l:G_T & H_T}
(i) The pullback of $G_T$ (resp.~$H_T$) via each of the two closed immersions $i_1,i_2:\Spf\BZ_p\mono T$ is isomorphic to~$W^{(F)}_{\Spf\BZ_p}$ (resp.~ to $(\hat\BG_a)_{\Spf\BZ_p}$).

(ii) $G_T$ is not isomorphic to $W^{(F)}_T\,$, and $H_T$ is not isomorphic to $(\hat\BG_a)_T\,$.
\end{lem}

\begin{proof}
We have the isomorphisms $i_1^*G_T\iso W^{(F)}_{\Spf\BZ_p}$ and $i_2^*G_T\iso W^{(F)}_{\Spf\BZ_p}$ given by \eqref{e:G_dR=W^(F)}  and~\eqref{e:G_eta=W^{F}}. Their pullbacks to $\Spec\BF_p=i_1(\Spf\BZ_p )\cap i_2(\Spf\BZ_p )$ are different: by \S\ref{sss:somewhat unexpected}, they differ by $\id -V\in\Aut W^{(F)}_{\BF_p}$.

It remains to show that the automorphism $\id -V\in\Aut W^{(F)}_{\BF_p}$ is \emph{not} in the image of $\Aut W^{(F)}_{\Spf\BZ_p}$. 
The Cartier duals of $W^{(F)}_{\BF_p}$ and $\id -V$ are $(\hat\BG_a)_{\BF_p}$ and $\id-\Fr\in\Aut (\hat\BG_a)_{\BF_p}\,$. It is clear that $\id-\Fr$ is not in the image of
$\Aut (\hat\BG_a)_{\Spf\BZ_p}=\BZ_p^\times$.
\end{proof}

\subsubsection{A subscheme $T'\subset Q$}
Let us formulate a variant of Lemma~\ref{l:G_T & H_T}. As usual, let $Q$ be the $q$-de~Rham prism, i.e., $Q:=\Spf\BZ_p [[q-1]]$ . Let $T'\subset Q$ be defined by the equation $q^p=1$.
Let $T$ be as in \S\ref{sss:T:=Spf B}. We have a commutative diagram
\[
\xymatrix{
T'\ar[r]\ar[d] & Q\ar[d]\\
T\ar[r] & \Sigma
}
\]
in which the morphism $Q\to\Sigma$ is as in \S\ref{sss:Recollections on Q} and the morphism $T'\to T$ comes from the ring homomorphism
\[
B\to\BZ_p [q]/(q^p-1), \quad (x,y)\mapsto y+\frac{x-y}{p}\cdot (1+q+\ldots +q^{p-1}),
\]
where $B$ is as in \S\ref{sss:T:=Spf B}.

\begin{lem}  \label{l:H_{T'}}
As before, let $T'\subset Q$ be defined by the equation $q^p=1$. 
Let $T'_1\subset T'$ (resp.~$T'_2\subset T'$) be defined by the equation $q=1$ (resp.~by $1+q+\ldots q^{p-1}=0$). Let $H_{T'}$, $H_{T'_1}$, $H_{T'_2}$ be the pullbacks of $H_Q$ to $T',T'_1, T'_2$. Then

(i) $H_{T'_1}\simeq (\hat\BG_a)_{T'_1}$ and $H_{T'_2}\simeq (\hat\BG_a)_{T'_2}$;

(ii) $H_{T'}$ is not isomorphic to $(\hat\BG_a)_{T'}$.
\end{lem}

\begin{proof}
Statement (i) follows from Lemma~\ref{l:G_T & H_T}(i). Statement (ii) is proved similarly to Lemma~\ref{l:G_T & H_T}(ii).
\end{proof}

\subsubsection{Remark} In connection with Lemma~\ref{l:H_{T'}}, let us note that $H_Q$ has a very explicit description, see \eqref{e:group law for H_Q}. This description was deduced from Theorem~\ref{t:H_Q & rescaled hat G_m}, which will be proved in the next subsection.

\subsection{Proof of Theorem~\ref{t:H_Q & rescaled hat G_m}}   \label{ss: H_Q & hat G_m proof}
\subsubsection{Recollections}
By \eqref{e:preimage of Delta_0 in Q}, the effective divisor $(\Delta_0)_Q:=\Delta_0\times_\Sigma Q\subset Q$ is defined by the equation $\Phi_p(q)=0$. 
Recall that $D\subset Q$ denotes the divisor $q=1$. Since $q^p-1=(q-1)\cdot\Phi_p(q)$, we get
\begin{equation}   \label{e:F^{-1}(D)}
F^{-1}(D)=D+(\Delta_0)_Q \, .
\end{equation}

We have a section $s_Q:Q\to H_Q$ and a homomorphism $\sigma^*:H_Q\to (\hat\BG_m)_Q$. By Lemma~\ref{l:sigma^* &other strictures}(ii),  $\sigma^*\circ s_Q$ is given by $q^p\in\hat\BG_m (Q)$, so $s_Q^{-1} (\Ker\sigma^* )$ is the divisor $q^p=1$. By \eqref{e:F^{-1}(D)}, we get
\begin{equation}    \label{e:s_Q-preimage of Ker sigma^*}
s_Q^{-1} (\Ker\sigma^* )=D+(\Delta_0)_Q.
\end{equation}

\begin{lem}     \label{l:Ker sigma^*}
(i) The closed subscheme $\Ker\sigma^*\subset H_Q$ is equal to the divisor $H_D+0_Q$, where $H_D\subset H_Q$ is the preimage of $D$ and $0_Q\subset H_Q$ is the zero section.

(ii) $s_Q^{-1}(0_Q)=(\Delta_0)_Q\,$.
\end{lem}

\begin{proof}
By \eqref{e:s_Q-preimage of Ker sigma^*}, $\Ker\sigma^*\ne H_Q$. Since $\Ker\sigma^*=(\sigma^*)^{-1}(0_Q)$ and $0_Q$ is a divisor in $H_Q$, we see that $\Ker\sigma^*$ is a divisor in $H_Q\,$.

From the definition of $\sigma$ (see \S\ref{sss:structures on H_Q}) it is clear that $\sigma :Q\to G_Q$ vanishes on $D$. So $\sigma^*:H_Q\to (\hat\BG_m)_Q$ vanishes over $D$. Therefore $\Ker\sigma^*\ge H_D+0_Q\,$. In other words, 
\[
\Ker\sigma^*=H_D+0_Q+\fD , \quad \mbox{ where } \fD\ge 0.
\]
Combining this with~\eqref{e:s_Q-preimage of Ker sigma^*}, we see that 
\[
s_Q^{-1}(0_Q)+s_Q^{-1}(\fD )=(\Delta_0)_Q\, .
\]
But $s_Q^{-1}(0_Q)\ge (\Delta_0)_Q$ by Corollary~\ref{c:s vanishes on Delta_0}. So $s_Q^{-1}(\fD )=0$. Therefore $\fD=0$.
\end{proof}

\subsubsection{Proof of Theorem~\ref{t:H_Q & rescaled hat G_m}}   \label{sss: H_Q & hat G_m proof}
We have to show that $\sigma^*:H_Q\to (\hat\BG_m)_Q$ induces an isomorphism $H_Q\iso (\hat\BG_m)_Q(-D)$.  
Choose an isomorphism  $H_Q\iso\Spf\BZ_p[[q-1,x]]$ of formal schemes over $Q$. Then $\sigma^*$ is given by a formal series $f\in\BZ_p[[q-1,x]]$ such that $f(x_1\star x_2)=f(x_1)f(x_2)$, where $\star$ is the group operation in $H_Q$. By Lemma~\ref{l:Ker sigma^*}(i), $f=1+(q-1)g$, where
\begin{equation}  \label{e:property of g}
g\in x\cdot \BZ_p[[q-1,x]]^\times .
\end{equation}
Then $g(x_1\star x_2)=g(x_1)+g(x_2)+(q-1)g(x_1)g(x_2)=g(x_1)*g(x_2)$, where $*$ is the group operation in $(\hat\BG_m)_Q(-D)$.  . Combining this with \eqref{e:property of g}, we see that $g$ defines an isomorphism of formal groups $H_Q\iso (\hat\BG_m)_Q(-D)$.  \qed

\subsection{Proof of Theorem~\ref{t:H_Sigma}}   \label{ss:H_Sigma proof}
As already mentioned in \S\ref{sss:Recollections on Q}, the morphism $\pi:Q\to\Sigma$ is faithfully flat.
So to prove Theorem~\ref{t:H_Sigma}, it suffices to check analogous statements about~$H_Q$. The analog of Theorem~\ref{t:H_Sigma}(i) has already been proved, see 
Lemma~\ref{l:Ker sigma^*}(ii). It remains to show that the morphism 
 $\varphi_Q :F^*H_Q\to H_Q$ factors as $F^*H_Q\iso H_Q ((-\Delta_0)_Q)\to H_\Sigma$.
This follows from Lemma~\ref{l:sigma^* &other strictures}(i), Theorem~\ref{t:H_Q & rescaled hat G_m} and formula~\eqref{e:F^{-1}(D)}. \qed

\subsubsection{Remark}
The interested reader can prove Theorem~\ref{t:H_Sigma} without using the $q$-de Rham prism. One can deduce it from Proposition~\ref{p:H_dR} and the description of $G_\eta$ given in the proof of Proposition~\ref{p:G_eta}(i).
(Proposition~\ref{p:H_dR} was deduced in \S\ref{s:main results} from Proposition~\ref{p:G_dR}, and the latter was proved in \S\ref{ss:G_dR}.)

\subsection{Proof of Proposition~\ref{p:sigma^* is equivariant} }\label{ss:equivariance of sigma^*}
\subsubsection{}   \label{sss:recollections on Q}
Recall that $Q=\Spf\BZ_p[[q-1]]$, $H_Q=\Spf\BZ_p[[q-1,z]]$, and the group operation on $H_Q$ is given by $z_1*z_2=z_1+z_2+(q-1)z_1z_2$.
We have a canonical section $s_Q:Q\to H_Q$; as explained in \S\ref{sss:H_Q explicitly}, it is given by $z=\frac{q^p-1}{q-1}$.

 \begin{lem}   \label{l:s_Q generates H_Q}
 Let $\cK\subset H_Q$ be a closed group subscheme such that $s_Q:Q\to H_Q$ factors through $\cK$. Then $\cK=H_Q$.
 \end{lem}
 
\begin{proof}
Assume the contrary. Then there exists a nonzero regular function $f$ on $H_Q$ which vanishes on the image of each composite morphism
\begin{equation}   \label{e:s_Q^n}
Q\overset{s_Q}\longrightarrow H_Q\overset{n}\longrightarrow H_Q\, , \quad n\in\BZ.
\end{equation}
Without loss of generality, we can assume that $f\in\BZ_p[[q-1,z]]$ has the form 
\[
f=\sum_{i=0}^\infty a_i(z)(q-1)^i, \quad \mbox{where } a_i\in\BZ_p[[z]], \; a_0\ne 0.
\]
By \S\ref{sss:recollections on Q}, the morphism \eqref{e:s_Q^n} is given by $z=\frac{q^{pn}-1}{q-1}$. The value of $\frac{q^{pn}-1}{q-1}$ at $q=1$ equals $pn$. So 
$a_0(pn)=0$ for all $n\in\BZ$, which contradicts the assumption $a_0\ne 0$.
\end{proof}

\subsubsection{Proof of Proposition~\ref{p:sigma^* is equivariant} }\label{sss:equivariance of sigma^*}
Recall that the formal groups $H_Q$ and $(\hat\BG_m)_Q$ are $\BZ_p^\times$-equivariant (the action of $\BZ_p^\times$ on $(\hat\BG_m)_Q$ is as in the formulation of 
Proposition~\ref{p:sigma^* is equivariant}). So $\BZ_p^\times$ acts on $\Hom (H_Q, (\hat\BG_m)_Q)$. 
We have an element $\sigma^*\in\Hom (H_Q, (\hat\BG_m)_Q)$, and the problem is to show that  $\alpha (\sigma^*)=\sigma^*$ for all $\alpha\in\BZ_p^\times$.
By Lemma~\ref{l:s_Q generates H_Q}, it suffices to check that for every $\alpha\in\BZ_p^\times$ one has
\begin{equation}    \label{e:inclusion to check}
s_Q (Q)\subset\cK_{\alpha}, \quad \mbox {where } \cK_{\alpha}:=\Ker (\alpha (\sigma^*)-\sigma^*)\subset H_Q.
\end{equation}

The section $s_Q:Q\to H_Q$ is $\BZ_p^\times$-equivariant because it comes from $s:\Sigma\to H_\Sigma$. 
By Lemma~\ref{l:sigma^* &other strictures}(ii), $\sigma^*\circ s_Q :Q\to (\hat\BG_m)_Q$ is also $\BZ_p^\times$-equivariant. So \eqref{e:inclusion to check} holds. \qed

\section{Several realizations of the group scheme $G_Q$}   \label{s:realizations of G_Q}
By definition, $G_Q$ is the pullback of $G_\Sigma$ to the $q$-de Rham prism $Q$. This immediately leads to the first realization of $G_Q$ and its coordinate ring, see 
\S\ref{ss:Recollections on G_Q}-\ref{ss:coordinate ring of G_Q}. 
In \S\ref{ss:tilde G_Q} we note that the coordinate ring of a certain extension of $G_Q$ by $(\mu_p)_Q$ appears in the theory of $q$-logarithm from \cite{ALB,BL}.

On the other hand, Theorem~\ref{t:H_Q & rescaled hat G_m} identifies $G_Q$ with the Cartier dual of a very explicit formal group $(\hat\BG_m)_Q(-D)$. This Cartier dual is denoted by $G_Q^!\,$. 
We explicitly describe $G_Q^!$ (see \S\ref{ss:G^!_Q}-\ref{ss:G_Q^!}) and the isomorphism $G_Q^!\iso G_Q$ (see \S\ref{ss:G_Q^!=G_Q}).

In  \S\ref{ss:G_Q^!?} we define group schemes $G_Q^{!?}, G_Q^{!!}$ and isomorphisms between them and $G_Q^!$.
Unlike $G_Q^!$ and similarly to $G_Q$, both $G_Q^{!?}$ and $G_Q$ are defined in terms of Witt vectors.

Let us note that in \S\ref{ss:G_Q^!}-\ref{ss:G_Q^!=G_Q} a key role is played by the expressions $(1+(q-1)z)^{\frac{t}{q-1}}$ and~$q^{\frac{pt}{q-1}}$. 
The closely related $q$-logarithm (in the sense of \cite{ALB,BL}) appears in formula \eqref{e:t=log_q(u)}.

\subsection{Recollections}   \label{ss:Recollections on G_Q}
\subsubsection{The formal $\delta$-scheme $Q$}
Let $Q:=\Spf\BZ_p[[q-1]]$, where $\BZ_p[[q-1]]$ is equipped with the $(p,q-1)$-adic topology.  Define $F:Q\to Q$ by $q\mapsto q^p$. Then $(Q,F)$ is
a formal $\delta$-scheme. 

\subsubsection{Pieces of structure on $G_Q$}    \label{sss:Pieces of structure on G_Q}
According to the definition from \S\ref{sss:structures on H_Q}, 
\[
G_Q:=G_\Sigma\times_\Sigma Q,
\]
where $G_\Sigma$ is as in \S\ref{ss:G_Sigma}.
$G_Q$ is a formal scheme over $Q$. The morphism $G_Q\to Q$ is schematic and affine; by Corollary~\ref{c:G_Sigma is flat}, it is flat.

Let us recall the pieces of structure on $G_Q\,$.
Most of them come from similar pieces of structure on $G_\Sigma$ (the only exception is (iii). 

(i$'$) $G_Q$ is a formal $\delta$-scheme over the formal $\delta$-scheme $Q$; in other words, $G_Q$ is equipped with a Frobenius lift $F:G_Q\to G_Q$, which is compatible with 
$F:Q\to Q$.

(i$''$)  $G_Q$ is a group scheme over $Q$. The group structure is compatible with $F:G_Q\to G_Q$, so $G_Q$ is a group $\delta$-scheme over $Q$.

(ii) One has a canonical map $G_Q\to (\BG_m)_Q$, which is a homomorphism of group $\delta$-schemes over $Q$. As usual, $(\BG_m)_Q:=\BG_m\times Q$, and the $\delta$-scheme structure on $\BG_m$ is given by raising to the power of $p$.

(iii) In \S\ref{sss:structures on H_Q} we defined a canonical section $\sigma :Q\to G_Q$, which is a $\delta$-morphism.

\subsubsection{Who is who}   \label{sss:Who is who}
For any $p$-nilpotent ring $A$ one has
\[
Q(A)=\{q\in A^\times\,|\, q-1 \mbox{ is nilpotent}\},
\]
\begin{equation}   \label{e:G_Q(A)}
G_Q(A)=\{ (q,x)\in Q(A)\times W(A)\,|\, 1+\Phi_p([q])x\in\BG_m (A)\},
\end{equation}
where $\Phi_p$ is the cyclotomic polynomial and $\BG_m$ is identified with a subgroup of $W^\times$ via the Teichm\" uller embedding. The morphism $F:G_Q\to G_Q$ is given by
\[
F(q,x)=(q^p,Fx).
\]
The group operation on $G_Q$ and the homomorphism $G_Q\to(\BG_m)_Q$ are given by the maps
\[
G_Q\times_Q G_Q\to G_Q\,  ,\quad (q,x_1,x_2)\mapsto (q,x_1+x_2+\Phi_p([q])x_1x_2),
\]
\[
G_Q\to\BG_m\, , \quad (q,x)\mapsto 1+\Phi_p([q])x.
\]
The section $\sigma :Q\to G_Q$ is given by
\begin{equation}  \label{e:sigma(q)}
\sigma (q):=(q,[q]-1).
\end{equation}

\subsection{The coordinate ring of $G_Q\,$}   \label{ss:coordinate ring of G_Q}
The coordinate ring $H^0(G_Q, \cO_{G_Q})$ is  a $(p,q-1)$-adically complete $\BZ_p[[q-1]]$-algebra. 
Since $G_Q$ is flat over $Q$, for any open ideal $I\subset\BZ_p[[q-1]]$ the tensor product $H^0(G_Q, \cO_{G_Q})\otimes_{\BZ_p[[q-1]]}(\BZ_p[[q-1]]/I)$ is flat over
$\BZ_p[[q-1]]/I$. In particular, $H^0(G_Q, \cO_{G_Q})$ is $p$-torsion-free, so $F:G_Q\to G_Q$ induces on $H^0(G_Q, \cO_{G_Q})$ a $\delta$-ring structure in the sense of \cite{JoyalDelta} and \cite[\S 2]{BS}. Let us describe $H^0(G_Q, \cO_{G_Q})$ as a $\delta$-algebra over $\BZ_p[[q-1]]$, where $\BZ_p[[q-1]]$ is considered as a $\delta$-ring with $\delta (q)=0$.

\begin{prop}   \label{p:coordinate ring of G_Q}
Let $R_0$ be the $\delta$-algebra over $\BZ [q]$ with a single generator $x_0$ and a single defining relation $\delta (1+\Phi_p(q)x_0)=0$. Let $R$ be the $(p,q-1)$-adic completion of $R_0$. Then there is a unique isomorphism $R_0\iso H^0(G_Q, \cO_{G_Q})$ of $\delta$-algebras over $\BZ_p [[q-1]]$ such that $x_0\in R$
goes to the following function on $G_Q$: the value of the function on a pair $(q,x)$ as in~\eqref{e:G_Q(A)} is the $0$-th component of the Witt vector $x$.
\end{prop}

\begin{proof}   
Let $Y$ be the affine scheme over $\BZ [q]$ such that for any $\BZ [q]$-algebra $A$ one has
\[
Y(A)=\{ x\in W(A)\,|\, 1+\Phi_p ([q])x\in\tau (A)\},
\]
where $\tau :A\to W(A)$ is the Teichm\"uller embedding. Then $G_Q$ is the $(p,q-1)$-adic formal completion of $Y$.

Let us construct an isomorphism $R_0\iso H^0(Y,\cO_Y)$. By definition, $Y$ is a closed subscheme of $W_{\BZ [q]}:=W\times\Spec\BZ [q]$.
By \S\ref{sss:Joyal} of Appendix~\ref{s:dual of hat G_m}, the coordinate ring of  $W_{\BZ [q]}$ is a free $\delta$-algebra over $\BZ [q]$ on a single generator $x_0$, where $x_0$ is the function that takes a Witt vector to its $0$-th component. Since the Teichm\"uller embedding $\BA^1\to W$ is a $\delta$-morphism, we see that
the ideal of $Y$ in $W_{\BZ [q]}$ is generated by $\delta^n(1+\Phi_p(q)x_0)$, $n>0$. So $H^0(Y,\cO_Y)=R_0$.
\end{proof}   

\subsection{$G_Q$ and the $q$-logarithm in the sense of  \cite{ALB,BL}}  \label{ss:tilde G_Q}
This subsection is a commentary on the notion of $q$-logarithm\footnote{As explained in \cite[Prop.~2.6.10]{BL}, the $q$-logarithm is closely related to the prismatic logarithm (i.e., to the homomorphism $G_\Sigma \to\cO_\Sigma \{ 1\}$ from our Corollary~\ref{c:prismatic logarithm}). We do not discuss this relation here.} from \cite[\S 4]{ALB} and \cite[\S 2.6]{BL}; the main point is that the $q$-logarithm is the unique group homomorphism 
$G_Q\to (\BG_a)_Q$ with a certain property (see the last sentence of \S\ref{sss:q-logarithm}). This material will be used in formula~\eqref{e:t=log_q(u)} and nowhere else.

\subsubsection{An extension of $G_Q$ by $(\mu_p)_Q$}
The definition of $q$-logarithm given in \cite[\S 2.6]{BL} following \cite[\S 4]{ALB} secretly uses the coordinate ring of a slight modification of~$G_Q$. Namely, for any $p$-nilpotent ring $A$ let
\begin{equation}  \label{e:tilde G_Q}
\tilde G_Q(A):=\{ (q,x,u)\in Q(A)\times W(A)\times A^\times\,|\, 1+\Phi_p([q])x=[u^p]\}
\end{equation}
(so $\tilde G_Q$ is an extension of  $G_Q$ by $(\mu_p)_Q$). The $\delta$-ring constructed in \cite[Prop.~2.6.5]{BL} is just the coordinate ring of $\tilde G_Q$; this easily follows from Proposition~\ref{p:coordinate ring of G_Q}. 

We have a section
\begin{equation}   \label{e:tilde sigma}
\tilde\sigma :Q\to \tilde G_Q, \quad \tilde\sigma (q):=(q, [q]-1, q),
\end{equation}
which lifts the section $\sigma :Q\to \tilde G_Q$ defined by \eqref{e:sigma(q)}.

\subsubsection{The $q$-logarithm}    \label{sss:q-logarithm}
On $\tilde G_Q$ we have an invertible regular function $u$, see formula \eqref{e:tilde G_Q}; note that $u^p$ is a regular function on $G_Q$ (unlike $u$).
The authors of \cite{ALB,BL} define another regular function on 
$\tilde G_Q$ denoted by $\log_q (u)$ and called the \emph{$q$-logarithm}\footnote{Warning: in the literature the word ``$q$-logarithm'' is used for many quite different functions, see the article~\cite{KVA}, especially its last section.}  of $u$. As explained below, $\log_q (u)$ is, in fact, a regular function on $G_Q$ itself.

Very informally, $\log_q (u)=\frac{q-1}{\log q}\cdot\log u$ (so $\log_q (u)$  is $q-1$ times the logarithm of $u$ with base $q$). From this informal description we see that $\log_q (u_1u_2)=\log_q (u_1)+\log_q(u_2)$ and $\log_q (q)=q-1$.

The precise definition of 
$\log_q (u)$ from \cite{ALB,BL} can be paraphrased as follows: $\log_q (u)$ is the unique group homomorphism $\tilde G_Q\to (\BG_a)_Q$ that takes the section \eqref{e:tilde sigma} to the section $q-1:Q\to  (\BG_a)_Q\,$ (the existence and uniqueness of such a homomorphism is proved in \cite[\S 4]{ALB}; see also \cite[Prop.~2.6.9]{BL}).

Note that the group $\Ker (\tilde G_Q\epi G_Q)=(\mu_p)_Q$ is killed by $\log_q (u)$ because 
$$\Hom ((\mu_p)_Q,  (\BG_a)_Q)=0.$$
So \emph{$\log_q (u)$ is a group homomorphism $G_Q\to (\BG_a)_Q\,$; it is the unique homomorphism that takes the section 
$\sigma :Q\to \tilde G_Q$ from \eqref{e:sigma(q)} to  the section $q-1:Q\to  (\BG_a)_Q\,$.}

\subsection{The group scheme $G_Q^!$}    \label{ss:G^!_Q}
\subsubsection{Definition of $H_Q^!$ and $G_Q^!$}
Let $D\subset Q$ be the effective divisor $q=1$.
Let $$H_Q^!:=(\hat\BG_m)_Q (-D),$$
i.e., $H_Q^!$ is the formal group over $Q$ obtained from $(\hat\BG_m)_Q$  by rescaling via the invertible subsheaf $\cO_Q(-D)\subset\cO_Q\,$, see \S\ref{ss:Rescaling}. 

Now define $G_Q^!$ to be the Cartier dual of $H_Q^!$. Then $G_Q^!$ is a flat affine group scheme over~$Q$.

Theorem~\ref{t:H_Q & rescaled hat G_m}  yields canonical isomorphisms $H_Q\iso H_Q^!$, $G_Q^!\iso G_Q$. But we will disregard these isomorphisms until \S\ref{ss:G_Q^!=G_Q}.

\subsubsection{$H_Q^!$ in explicit terms}   \label{sss:H_Q^! explicitly}
As a formal scheme, $H_Q^!=\Spf\BZ_p[[q-1,z]]$, and the group operation is
\begin{equation}   \label{e:z_1*z_2}
z_1*z_2=\frac{(1+(q-1)z_1)(1+(q-1)z_2)-1}{q-1}=z_1+z_2+(q-1)z_1z_2\, .
\end{equation}

Let $H^!$ be the formal group over $\BA^1=\Spec\BZ [q]$ defined by the group law \eqref{e:z_1*z_2}; then $H_Q^!=H^!\times_{\BA^1}Q$.

\subsubsection{Pieces of structure on $H_Q^!\,$}   \label{sss:Pieces of structure on H_Q^!}
$H_Q^!$ is a formal group over $Q$ equipped with a homomorphism 
\begin{equation}  \label{e:H_Q^! to G_m}
H_Q^!\to (\hat\BG_m)_Q.
\end{equation}
 In terms of the coordinate $z$ from \S\ref{sss:H_Q^! explicitly}, it is given by the function $1+(q-1)z$.
 
Since $F^{-1}(D)\supset D$, there is a unique homomorphism
$\varphi_Q :F^*H_Q^!\to H_Q^!$ such that the diagram
\begin{equation}   \label{e:diagram defining varphi_Q}
\xymatrix{
F^*H_Q^!\ar[r]^{\varphi_Q} \ar[d] & H_Q^!\ar[d]\\
F^*(\hat\BG_m)_Q\ar[r]^-\sim & (\hat\BG_m)_Q
}
\end{equation}
commutes; here the lower horizontal arrow comes from the fact that $(\hat\BG_m)_Q:=\BG_m\times Q$. Over $Q\otimes\BF_p$ the upper horizontal arrow of \eqref{e:diagram defining varphi_Q} becomes the Verschiebung (because the lower horizontal arrow does). 

Moreover, the map $Q\to\hat\BG_m$, $q\mapsto q^p$, defines a section $Q\to (\hat\BG_m)_Q\,$, which comes from a section 
\begin{equation}    \label{e:Q to  H_Q^!}
s_Q:Q\to  H_Q^!\,. 
\end{equation}
In terms of \S\ref{sss:H_Q^! explicitly}, $s_Q$ is given by $z=\Phi_p(q)$.

The following diagram commutes: 
\[
\xymatrix{
  F^*H_Q^!\ar[r]^-{\varphi_Q} &H_Q^!    \\
Q  \ar[u]^{F^*(s_Q)}\ar[r]^{s_Q}&H_Q^!  \ar[u]_{p}
}
\]

Note that \eqref{e:H_Q^! to G_m} and $\varphi_Q :F^*H_Q^!\to H_Q^!$ come from similar pieces of structure on the formal group $H^!$ from \S\ref{sss:H_Q^! explicitly}; on the other hand, \eqref{e:Q to  H_Q^!} does not have an analog for $H^!$.

\subsubsection{Pieces of structure on $G_Q^!$}  \label{sss:Pieces of structure on G_Q^!}
Dualizing \S\ref{sss:Pieces of structure on H_Q^!}, we get the following 
pieces of structure on $G_Q^!$, which are parallel to those from \S\ref{sss:Pieces of structure on G_Q}.

(i) The homomorphism $\varphi_Q :F^*H_Q^!\to H_Q^!$ yields a map $F:G_Q^!\to G_Q^!$, which makes $G_Q^!$ into a group $\delta$-scheme over $Q$.

(ii) The section \eqref{e:Q to  H_Q^!} yields a canonical map $G_Q^!\to (\BG_m)_Q$, which is a homomorphism of group $\delta$-schemes over $Q$. 

(iii) The homomorphism \eqref{e:H_Q^! to G_m}  yields a canonical section $\sigma^!: Q\to G_Q^!\,$, which is a $\delta$-mor\-phism.

An explicit description of $G_Q^!$ (together with the above pieces of structure on it) will be given in Proposition~\ref{p:G_Q^!=SpfB}.

\subsubsection{The group scheme $G^!$}  
Let $G^!$ be the Cartier dual of the formal group $H^!$ from \S\ref{sss:H_Q^! explicitly}. Then $G^!_Q=G^!\times_{\BA^1}Q$.

The pieces of structure from \S\ref{sss:Pieces of structure on G_Q^!}(i,iii) are pullbacks of similar pieces of structure on $G^!$. On the other hand, the piece of structure from 
\S\ref{sss:Pieces of structure on G_Q^!}(ii) does not have an analog for $G^!$.

The affine group scheme $G^!$ and its coordinate ring are described in Appendix~\ref{s:rescaled G_m and its Cartier dual}. 
We will use these results below.

\subsection{Explicit description of $G_Q^!$}    \label{ss:G_Q^!}
\subsubsection{The ring $B$}    \label{sss:B}

Let $B_0$ be the Hopf algebra over $\BZ [h]$ from Proposition~\ref{p:G=Spec B_0} (see also \S\ref{sss:B_0 as delta-ring} for a description of $B_0\otimes\BZ_{(p)}$).
Let $B$ be the $(p,h)$-adic completion of $B_0$. Then $B$ is a topological Hopf algebra over $\BZ_p[[q-1]]$, where $q=1+h$.
Elements of $B$ are infinite sums
\begin{equation}  \label{e:elements of B}
\sum_{n=0}^\infty a_n\cdot \frac{t(t-q+1)\ldots (t-(n-1)(q-1))}{n!}, \quad \mbox{ where } a_n\in\BZ_p[[q-1]], \; a_n\to 0.
\end{equation}
An element \eqref{e:elements of B} is in $B_0$ if and only if $a_n\in\BZ [q]$ for all $n$ and $a_n=0$ for $n\gg 0$.
Note that $B$ is torsion-free as a $\BZ_p[[q-1]]$-module.

\begin{prop}   \label{p:G_Q^!=SpfB}
(a) The group scheme $G_Q^!$ identifies with $\Spf B$
so that in terms of this identification the pairing $G_Q^!\times H_Q^!\to (\BG_m)_Q$ is given by the formal series
\begin{equation}   \label{e:almost exp(tx)}
(1+(q-1)z)^{\frac{t}{q-1}}:=\sum_{n=0}^\infty \frac{t(t-q+1)\ldots (t-(n-1)(q-1))}{n!}\cdot z^n\in B[[z]]^\times , 
\end{equation}
where $z$ is the coordinate on $H_Q^!$ from \S\ref{sss:H_Q^! explicitly}.

(a$'$) The regular function on $G_Q^!$ corresponding to $t\in B$ defines a group homomorphism 
\begin{equation}   \label{e:t is additive}
G_Q^!\to (\BG_a)_Q\,.
\end{equation}

(b) The homomorphism $\phi :B\to B$ corresponding to  the morphism $F: G_Q^!\to G_Q^!$ from \S\ref{sss:Pieces of structure on G_Q^!}(i) is given by 
$$\phi (q)=q^p,\; \phi (t)=\Phi_p(q)t.$$
Moreover, $\phi$ makes $B$ into a $\delta$-ring.

(c) The homomorphism $G_Q^!\to (\BG_m)_Q$ from \S\ref{sss:Pieces of structure on G_Q^!}(ii) is given by the element
\begin{equation}   \label{e:q'}
q^{\frac{pt}{q-1}}:=\sum_{n=0}^\infty \frac{t(t-q+1)\ldots (t-(n-1)(q-1))}{n!}\cdot \Phi_p(q)^n\in B^\times ,
\end{equation}
which is obtained from \eqref{e:almost exp(tx)} by setting $z=\Phi_p(q)$ (the sum converges because $\Phi_p(1)=p$).

(c$'$) One has
\[
q^{\frac{pt}{q-1}}=\sum_{n=0}^\infty \alpha_n\, ,\quad\mbox{ where } \alpha_n:= \frac{pt(pt-q+1)\ldots (pt-(n-1)(q-1))}{n!};
\]
more precisely, $\alpha_n\in B_0$ and the series $\sum_{n=0}\limits^\infty \alpha_n$ converges in $B$ to the element $q^{\frac{pt}{q-1}}$ defined by~\eqref{e:q'}.

(d) The section $\sigma^!: Q\to G_Q^!$ from \S\ref{sss:Pieces of structure on G_Q^!}(iii) corresponds to the algebra homomorphism $B\to\BZ_p[[q-1]]$ such that $t\mapsto q-1$.
\end{prop}

The ``true meaning'' of the homomorphism \eqref{e:t is additive} will be explained later, see formula \eqref{e:t=log_q(u)}.

\begin{proof}
Statement (a) follows from Proposition~\ref{p:G=Spec B_0} and formula \eqref{e:Euler series}. 

Statement (a$'$) is clear from \S\ref{ss:G in terms of Newton&Euler} or Proposition~\ref{p:G=Spec B_0}(iii). It also follows from \eqref{e:almost exp(tx)} combined with the formula $(1+(q-1)z)^{\frac{t_1+t_2}{q-1}}=(1+(q-1)z)^{\frac{t_1}{q-1}}\cdot (1+(q-1)z)^{\frac{t_2}{q-1}}$.

By Lemma~\ref{l:motivation of lambda-structure}(ii), $F: G_Q^!\to G_Q^!$ is the base change of the morphism $\Psi_p:G^!\to G^!$ from \S\ref{sss:Psi_n}, so $\phi :B\to B$ is the base change of the homomorphism $\psi^p:B_0\to B_0$ from Lemma~\ref{l:B_0 as lambda-ring}. Since $B$ is $p$-torsion-free, $\phi$ makes $B$ into a $\delta$-ring. This proves (b).

Statement (c) is clear because the homomorphism $G_Q^!\to (\BG_m)_Q$ comes from the section~\eqref{e:Q to  H_Q^!}, which is given by $z=\Phi_p(q)$.

To prove (d), recall that $\sigma^!$ comes from the homomorphism \eqref{e:H_Q^! to G_m}, which is given by the function $1+(q-1)z$; this function is
the result of substituting $t=q-1$ into \eqref{e:almost exp(tx)}.

Let us prove (c$'$). By Lemma~\ref{l:simple lemma}, $\alpha_n\in B_0$ and in the ring $B_0[[z]]$ one has
\begin{equation} \label{e:sum of alpha_n z^n}
\sum_{n=0}^\infty \alpha_nz^n=(1+(q-1)vz)^{\frac{t}{q-1}}:=\sum_{n=0}^\infty\beta_nv^nz^n,
\end{equation}
where $\beta_n:=\frac{t(t-q+1)\ldots (t-(n-1)(q-1))}{n!}$ and
\[
v=v(q,z):=\frac{(1+(q-1)z)^p-1}{(q-1)z}=\sum_{i=1}^p\binom{p}{i}(q-1)^{i-1}z^{i-1}\in\BZ [q,z].
\]
Note that $v\in I[z]$, where $I\subset\BZ[q]$ is the ideal $(p,q-1)$. So the r.h.s of \eqref{e:sum of alpha_n z^n} belongs to the subring 
$\underset{m}{\underset{\longleftarrow}{\lim}}\, (B/I^mB)[z]\subset B[[z]]$. Therefore we can set $z=1$ and get
\[
\sum_n\alpha_n=\sum_n\beta_n \cdot v(q,1)^n=\sum_n\beta_n\cdot\Phi_p(q)^n;
\]
in other words, $\sum_n\limits\alpha_n$ equals the r.h.s of \eqref{e:q'}.
\end{proof}

\subsection{The isomorphism $G_Q^!\iso G_Q$ in explicit terms}   \label{ss:G_Q^!=G_Q}
\subsubsection{The isomorphisms $H_Q\iso H_Q^!$ and $G_Q^!\iso G_Q$}  \label{sss:preserve pieces of structure}
Theorem~\ref{t:H_Q & rescaled hat G_m}  yields a canonical isomorphism $H_Q\iso H_Q^!$. It is compatible with the pieces of structure on $H_Q$ and $H_Q^!$
introduced in \S\ref{sss:structures on H_Q} and  \S\ref{sss:Pieces of structure on H_Q^!}. So the Cartier dual isomorphism $G_Q^!\iso G_Q$ transforms the pieces of structure on $G_Q^!$ from \S\ref{sss:Pieces of structure on G_Q^!} into the corresponding pieces of structure on $G_Q$ (see \S\ref{sss:Pieces of structure on G_Q}-\ref{sss:Who is who}).

\subsubsection{The isomorphism between the coordinate rings of $G_Q^!$ and $G_Q$}   \label{sss:R=B}
Recall that $G_Q=\Spf R$, $G_Q^!=\Spf B$, where $R:=\hat R_0$ and $B:=\hat B_0$ are the $(p,q-1)$-adic completions of the $\BZ [q]$-algebras $R_0$ and $B_0$ from Propositions~\ref{p:coordinate ring of G_Q} and \ref{p:G=Spec B_0}. So the canonical isomorphism $G_Q^!\iso G_Q$ induces an isomorphism $R\iso B$; using it, we identify $R$ and $B$. Then the element $x_0\in R_0$ from Proposition~\ref{p:coordinate ring of G_Q} and the element $t\in B_0$ from \S\ref{sss:B} live in the same ring $R=B$. Let us discuss the relation between them.

By Proposition~\ref{p:G_Q^!=SpfB}(c), we have
\begin{equation}   \label{e:1+Phi_p(q)x_0}
1+\Phi_p(q)x_0=q^{\frac{pt}{q-1}},
\end{equation}
\begin{equation}   \label{e:x_0}
x_0=\frac{q^{\frac{pt}{q-1}}-1}{\Phi_p(q)}=\sum\limits_{n=1}^\infty \frac{t(t-q+1)\ldots (t-(n-1)(q-1))}{n!}\cdot \Phi_p(q)^{n-1}.
\end{equation}

We claim that in terms of the $q$-logarithm (see \S\ref{sss:q-logarithm}) one has
\begin{equation} \label{e:t=log_q(u)}
t=\log_q(u), \quad\mbox{ where } u^p=1+\Phi_p(q)x_0\,, 
\end{equation}
which implies that $pt=\log_q(1+\Phi_p(q)x_0)$. This follows from Proposition~\ref{p:G_Q^!=SpfB}(a$'$,d) and the definition of $\log_q(u)$ at the end of  \S\ref{sss:q-logarithm}.

\subsubsection{Remark}    \label{sss:R_0 different from B_0}
Using \eqref{e:x_0}, it is easy to show that $R_0$ and $B_0$ are \emph{different} as subrings of~$R=B$; moreover, $R_0/(q-1)R_0$ and $B_0/(q-1)B_0$ are different as subrings of the ring $R/(q-1)R=B/(q-1)B$.

\subsubsection{Plan of what follows}
By Proposition~\ref{p:G_Q^!=SpfB}, $G_Q^!=\Spf B$. By \eqref{e:G_Q(A)}, the isomorphism
\begin{equation}  \label{e:Spf B=G_Q}
\Spf B=G_Q^!\iso G_Q
\end{equation}
 is given by an element 
$x\in W(B)$ such that $1+\Phi_p([q])x\in B^\times$, where $B^\times\subset W(B)^\times$ is the subgroup of Teichm\"uller elements.
In Proposition~\ref{p:formula for tilde x} we will write a formula for $x$. 

\subsubsection{The homomorphism $\psi:B\to W(B)$}  \label{e:psi}
According to A.~Joyal \cite{JoyalDelta}, the forgetful functor from the category of $\delta$-rings to that of rings has a right adjoint, which is nothing but the functor $W$. Our $B$ is a $\delta$-ring, so 
the unit of Joyal's adjunction yields a homomorphism of $\delta$-rings $\psi:B\to W(B)$. It is the unique homomorphism of $\delta$-rings $B\to W(B)$ whose composition with the
canonical epimorphism $W(B)\epi W_1(B)=B$ equals $\id_B\,$. For any $b\in B$ the $n$-th Buium-Joyal component (see \S\ref{sss:Joyal}) of the Witt vector $\psi (b)$ equals $\delta^n (b)$.

\begin{prop}    \label{p:formula for tilde x}
One has
\begin{equation}  \label{e:formula for tilde x}
x=\psi (\frac{q^{\frac{pt}{q-1}}-1}{\Phi_p(q)}), 
\end{equation}
\begin{equation}   \label{e:formula for 1+tilde x xi}
1+\Phi_p([q])x=[q^{\frac{pt}{q-1}}],
\end{equation}
where $\psi:B\to W(B)$ is as in \S\ref{e:psi} and $q^{\frac{pt}{q-1}}\in B$ is defined by \eqref{e:q'} (so $q^{\frac{pt}{q-1}}-1$ is divisible by $\Phi_p(q)$).
\end{prop}

\begin{proof}  
By \S\ref{sss:preserve pieces of structure}, the morphism $\Spf B=G_Q^!\iso G_Q$ is a $\delta$-morphism. So $x:\Spf B\to W$ is a $\delta$-mor\-phism. Therefore the corresponding map $H^0(W,\cO_W )\to B$ is a $\delta$-homomorphism. So the description of $H^0(W,\cO_W )$ from \S\ref{sss:Joyal} shows 
that $x=\psi (x_0)$, where $x_0$ is the $0$-th component of the Witt vector $x$. Combining this with \eqref{e:x_0}, we get \eqref{e:formula for tilde x}.

Formula \eqref{e:formula for 1+tilde x xi} follows from \eqref{e:1+Phi_p(q)x_0} because $1+\Phi_p([q])x$ is a Teichm\'uller element. 
\end{proof}  

\subsection{The group schemes $G_Q^{!?}$ and $G_Q^{!!}$}    \label{ss:G_Q^!?}
Using Witt vectors, we will define group $\delta$-schemes $G_Q^{!?}$ and $G_Q^{!!}$ over $Q$; each of them is canonically isomorphic to $G_Q^!$ and therefore to $G_Q$. The author is not sure that $G_Q^{!?}$ is really useful; this explains the question mark in the notation.

\subsubsection{Definition of $G_Q^{!?}$}
For any $p$-nilpotent ring $A$ let
\begin{equation}   \label{e:G_Q^{!?}}
G_Q^{!?}(A)=\{ (q,y)\in Q(A)\times W(A)\,|\, Fy=[q-1]^{p-1}\cdot y\}.
\end{equation}
Then $G_Q^{!?}\subset W_Q$ is a group subscheme. The map
\[
G_Q^{!?}\to G_Q^{!?}, \quad (q,y)\mapsto (q^p, [\Phi_p(q)]\cdot y)
\]
makes $G_Q^{!?}$ into a group $\delta$-scheme over the formal $\delta$-scheme $Q$.

\begin{prop}    \label{p:G_Q^! inside W_Q}
One has a canonical isomorphism 
\begin{equation}   \label{e:! to !?}
G_Q^!\iso G_Q^{!?}\, ;
\end{equation}
of group $\delta$-schemes over $Q$; it is induced by the map \eqref{e:G to W_big}.
\end{prop}

\begin{proof}
Follows from Proposition~\ref{p:G in terms of Witt} and \eqref{e:G^!? p-typically}. 
\end{proof}

\subsubsection{Definition of $G_Q^{!!}$}
For any $p$-nilpotent ring $A$ let
\begin{equation}   
G_Q^{!!}(A)=\{ (q,y)\in Q(A)\times W(A)\,|\, Fy=\Phi_p([q])\cdot y\}.
\end{equation}
Then $G_Q^{!!}\subset W_Q$ is a group subscheme. Moreover, $G_Q^{!!}$ is a $\delta$-subscheme if $W_Q=W\times Q$ is equipped with the product of the standard $\delta$-structures on $W$ an $Q$.

\subsubsection{The isomorphism $G_Q^!\iso G_Q^{!!}$}  \label{sss:G_Q^!=G_Q^!!}
Let $t\in H^0(G_Q^{!!}, \cO_{G_Q^{!!}})$ be the function that takes $(q,y)\in G_Q^{!!}(A)$ to the $0$-th component of the Witt vector $y$.
Similarly to the proof of Proposition~\ref{p:coordinate ring of G_Q}, one shows that $H^0(G_Q^{!!}, \cO_{G_Q^{!!}})$ is the $(p,q-1)$-adic completion of the 
$\delta$-algebra over $\BZ [q]$ with a single generator $t$ and a single relation 
\[
t^p+p\delta (t)=\Phi_p(q)\cdot t.
\]
Combining this with \S\ref{sss:B_0 as delta-ring} and Proposition~\ref{p:G_Q^!=SpfB}(a,b), we get an isomorphism of $\delta$-rings 
$H^0(G_Q^{!!}, \cO_{G_Q^!})\iso H^0(G_Q, \cO_{G_Q^!})$. The corresponding isomorphism $G_Q^!\iso G_Q^{!!}$ is a group isomorphism by 
Proposition~\ref{p:G_Q^!=SpfB}(a$'$).

\subsubsection{Remark}   \label{sss:the q=1 case}
Combining Proposition~\ref{p:G_Q^! inside W_Q} and \S\ref{sss:G_Q^!=G_Q^!!} with the isomorphism $G_Q^!\iso G_Q$ from \S\ref{ss:G_Q^!=G_Q},  we get canonical isomorphisms between the group $\delta$-schemes $G_Q$, $G_Q^{!?}$, and $G_Q^{!!}$. These group $\delta$-schemes are defined in terms of $W$, but I do not know an explicit description of these isomorphisms in terms of the standard Witt vector formalism. However, \emph{after} the ``de Rham'' specialization $q=1$ the isomorphisms in question specialize to the explicit isomorphisms from~\S\ref{ss:G_dR} (note that $G_Q$, $G_Q^{!?}$, and  $G_Q^{!!}$ specialize to $G_{\dR}$,  
$W^{(F)}_{\Spf\BZ_p}$, and $W^{F=p}_{\Spf\BZ_p}$, respectively).

\appendix

\section{On the prismatic cohomology of $(\BA^1\setminus\{0\})_{\Spf\BZ_p}$}   \label{s:prismatic cohomology of G_m}
\subsection{The result}
Let $\BG_m^\prism$ be the prismatization of  $(\BA^1\setminus\{0\})_{\Spf\BZ_p}=(\BG_m)_{\Spf\BZ_p}$. The projection $(\BG_m)_{\Spf\BZ_p}\to\Spf\BZ_p$ induces a morphism 
$\pi:\BG_m^\prism\to (\Spf\BZ_p)^\prism=\Sigma$. The goal of this Appendix is to compute the higher derived images
$R^i\pi_*\cO_{\BG_m^\prism}$ using some results of \S\ref{s:main results}. Here is the answer; it is almost contained\footnote{In \cite[\S 16]{BS} the pullback of $R^i\pi_*\cO_{\BG_m^\prism}$ to the $q$-de Rham prism $Q$ was computed. Theorem~\ref{t:prismatic cohomology of G_m} can be easily deduced from this computation.} in \cite[\S 16]{BS}.
\begin{thm}  \label{t:prismatic cohomology of G_m}
(i) $R^i\pi_*\cO_{\BG_m^\prism}=0$ if $i\ne 0,1$.

(ii) $R^0\pi_*\cO_{\BG_m^\prism}=\cO_\Sigma$.

(iii) $R^1\pi_*\cO_{\BG_m^\prism}=\bigoplus\limits_{n\in\BZ}\cM_n\,$, where $\cM_0:=\cO_{\Sigma}\{ -1\}$ and if $n\ne 0$ and $m$ is the biggest number such that $p^m | n$ then
$\cM_n:=\cO_{\Sigma}(\Delta_0+\ldots +\Delta_m)/\cO_\Sigma\,$. Here $\Delta_0\subset\Sigma$ is the Hodge-Tate divisor and $\Delta_i:=(F^i)^{-1}(\Delta_0)$.
\end{thm}

\subsection{Proof of Theorem~\ref{t:prismatic cohomology of G_m}}
By Corollary~\ref{c:2 BG_m^prism as Cone},
$\BG_m^\prism=\Cone (G_{\Sigma}\to (\BG_m)_\Sigma  )$. Thus
\[
R^i\pi_*\cO_{\BG_m^\prism}=H^i(G_{\Sigma}, \cO_\Sigma\otimes A),
\]
where $A$ is the regular representation of $\BG_m$ (so $\cO_\Sigma\otimes A$ is a $(\BG_m)_\Sigma$-module and therefore a $G_\Sigma$-module). Equivalently, 
\begin{equation}   \label{e:decomposition w.r.t. characters}
R^i\pi_*\cO_{\BG_m^\prism}=\bigoplus_{n\in\BZ}H^i(G_{\Sigma}, \cO_\Sigma\otimes \chi_n),
\end{equation}
where $\chi_n$ is the 1-dimensional $\BG_m$-module corresponding to the character $z\mapsto z^n$.

By Theorem~\ref{t:dual to formal group}, $G_\Sigma$ is the Cartier dual of a $1$-dimensional formal group $H_\Sigma$. Let $s:\Sigma\to H_\Sigma$ be the section corresponding to the canonical homomorphism $G_\Sigma\to (\BG_m)_\Sigma$. 
For $n\in\BZ$, let $D_n\subset H_\Sigma$ be the image of the composite morphism $\Sigma\overset{s}\longrightarrow H_\Sigma \overset{n}\longrightarrow H_\Sigma\,$; in particular, $D_0\subset H_\Sigma$ is the image of the zero section. Then
\begin{equation}  \label{e:after Fourier transform}
H^i(G_{\Sigma}, \cO_\Sigma\otimes \chi_n)=R^i {\bf 0}^! \cO_{D_n}\, ,
\end{equation}
where ${\bf 0}:\Sigma\to H_\Sigma$ is the zero section and $\cO_{D_n}$ is viewed as an $\cO_{H_\Sigma}$-module.

\begin{lem} \label{l:n=0}
$R^i {\bf 0}^! \cO_{D_0}=0$ for $i\ne 0,1$,  $R^0 {\bf 0}^! \cO_{D_0}=\cO_\Sigma$, and $R^1 {\bf 0}^! \cO_{D_0}=\cO_\Sigma \{ -1\}$.
\end{lem}

\begin{proof}
By Theorem~\ref{t:2Hom (G_Sigma ,G_a)}, $\Lie (H_\Sigma)=\cO_\Sigma \{ -1\}$.
\end{proof}

The following lemma is a reformulation of Corollary~\ref{c:zeros of p^ns}.

\begin{lem}
Let $n=p^m n'$, where $(n',p)=1$. Then the projection $H_\Sigma\to\Sigma$ induces an isomorphism $D_0\cap D_n\iso\Delta_0+\ldots\Delta_m$,
where $\Delta_0\subset\Sigma$ is the Hodge-Tate divisor and $\Delta_i:=(F^i)^{-1}(\Delta_0)$. \qed
\end{lem}

\begin{cor} \label{c:n nonzero}
If $n\ne 0$ then $R^i {\bf 0}^! \cO_{D_n}=0$ for $i\ne 1$ and $R^1 {\bf 0}^! \cO_{D_n}=\cM_n$, where $\cM_n$ is as in Theorem~\ref{t:prismatic cohomology of G_m}(iii). \qed
\end{cor}

Combining Lemma~\ref{l:n=0}, Corollary~\ref{c:n nonzero}, and \eqref{e:decomposition w.r.t. characters}-\eqref{e:after Fourier transform}, we get Theorem~\ref{t:prismatic cohomology of G_m}.

\section{The Cartier dual of the divided powers version of $\BG_m$}  \label{s:Cartier dual of G_m^sharp}
\subsection{Plan}   \label{ss:plan of Cartier dual of G_m^sharp}
As usual, let $\BG_m=\Spec\BZ [x,x^{-1}]$ be the multiplicative group over $\BZ$. Let $\BM_m$ denote the scheme $\BA^1=\Spec\BZ [x]$ viewed as a multiplicative monoid over $\Spec\BZ$. Let $\BG_m^\sharp$ (resp.~$\BM_m^\sharp$) be the PD hull of the unit in $\BG_m$ (resp.~in $\BM_m$);
explicitly,
\begin{equation}  \label{e:coordinate rings of PD hulls}
\BM_m^\sharp=\Spec A,\; \;\BG_m^\sharp=\Spec A[1/x], \quad \mbox{ where } A:=\BZ [x, \frac{(x-1)^2}{2!},  \frac{(x-1)^3}{3!}, \ldots]\, .
\end{equation}
The monoid structure on $\BM_m$ and $\BG_m$ extends to a monoid structure on $\BM_m^\sharp$ and $\BG_m^\sharp\,$. Moreover, the monoid $\BG_m^\sharp$ is a group.

 Theorem~\ref{t:the dual of G_m^sharp} below describes the Cartier duals\footnote{By the Cartier dual of $\BM_m^\sharp$ we mean $\HHom (\BM_m^\sharp ,\BM_m)$; equivalently, the bialgebras corresponding to $\BM_m^\sharp$ and its Cartier dual are dual to each other.} of $\BG_m^\sharp$ and $\BM_m^\sharp$ (this description is likely to be known, but I was unable to find a reference). The description becomes even simpler after base change to $\Spf\BZ_p\,$, see \S\ref{ss:over Spf Z_p}.
 
 In \S\ref{ss:dualizing the exact sequence} we construct an exact sequence \eqref{e:G_m^sharp sequence} of group schemes over $\Spf\BZ_p\,$, which plays an important role in \cite{BL}.
 In \S\ref{ss:variant over Z_(p)} we discuss a variant of \eqref{e:G_m^sharp sequence} over $\Spec\BZ_{(p)}\,$, where $\BZ_{(p)}$ is the localization of $\BZ$ at $p$.

\subsection{Formulation of the theorem}
We will define ind-schemes $\Gamma$ and $\Gamma_+$ equipped with a monoid structure; moreover, $\Gamma$ is a group.
Then we will identify the Cartier duals of $\BG_m^\sharp$ and $\BM_m^\sharp$ with $\Gamma$ and $\Gamma_+\,$, respectively.

\subsubsection{Definition of $\Gamma$ and $\Gamma_+$}  \label{sss: def of Gamma}
Given integers $a\le b$, define a polynomial $f_{a,b}\in\BZ [u]$ by
\begin{equation}   \label{e:f_{a,b}}
f_{a,b}(u):=\prod_{i=a}^{b} (u-i).
\end{equation}
Define a closed subscheme $\Gamma^{[a,b]}\subset\BA^1=\Spec\BZ [u]$ by  
\[
\Gamma^{[a,b]}:=\Spec\BZ [u]/(f_{a,b}).
\]

The addition map $\BA^1\times\BA^1\to\BA^1$ induces a morphism $\Gamma^{[a,b]}\times\Gamma^{[c,d]}\to \Gamma^{[a+c,b+d]}$.
So the ind-schemes
\[
\Gamma:=\Gamma^{[-\infty,\infty]}:=\underset{N}{\underset{\longrightarrow}{\lim}}\, \Gamma^{[-N,N]}, \quad \Gamma^+:=\Gamma^{[0,\infty]}:= \underset{N}{\underset{\longrightarrow}{\lim}}\, \Gamma^{[0,N]}
\]
are monoids; moreover, $\Gamma$ is a group.

\subsubsection{The pairings}
We have a pairing
\begin{equation}  \label{e:BM_m times Gamma^+ to BM_m}
\BM_m^\sharp\times\Gamma^+\to\BM_m\, , \quad (x,u)\mapsto x^u:=\sum_{n=0}^\infty f_{0,n}(u)\cdot \frac{(x-1)^n}{n!},
\end{equation}
where $ f_{0,n}$ is defined by formula \eqref{e:f_{a,b}}. Since $\BG_m^\sharp$ is a group, the morphism \eqref{e:BM_m times Gamma^+ to BM_m} maps
$\BG_m^\sharp\times\Gamma^+$ to $\BG_m$. Define a pairing
\begin{equation}  \label{e:BG_m times Gamma to BG_m}
\BG_m^\sharp\times\Gamma\to\BG_m 
\end{equation}
as follows: for each integer $a\ge 0$ its restriction to $\BG_m^\sharp\times\Gamma^{[-a,\infty]}$ is given by
\[
(x,u)\mapsto x^{-a}\cdot x^{u+a}=x^{-a}\cdot\sum_{n=0}^\infty f_{0,n}(u+a)\cdot \frac{(x-1)^n}{n!}.
\]
\begin{thm}  \label{t:the dual of G_m^sharp}
(i) The pairings \eqref{e:BG_m times Gamma to BG_m} and \eqref{e:BM_m times Gamma^+ to BM_m} induce isomorphisms
\[
\BG_m^\sharp\iso\HHom (\Gamma ,\BG_m), \quad \BM_m^\sharp\iso\HHom (\Gamma^+ ,\BM_m).
\]

(ii) The coordinate ring of $\BG_m^\sharp$ is a free $\BZ$-module.\footnote{A similar statement for $\BM_m^\sharp$ is obvious, see formula \eqref{e:coordinate rings of PD hulls}.}
\end{thm}

\subsection{Proof of Theorem~\ref{t:the dual of G_m^sharp}}
\subsubsection{Distributions on $\Gamma$ and $\Gamma^+$}
Let $\Distr (\Gamma^{[a,b]})$ be the $\BZ$-module dual to the coordinate ring of $\Gamma^{[a,b]}$; equivalently, $\Distr (\Gamma^{[a,b]})$ is the $\BZ$-module of those linear functionals $\BZ [u]\to\BZ$ that are trivial on the ideal $(f_{a,b})\subset\BZ [u]$.
We think of elements of $\Distr (\Gamma^{[a,b]})$ as \emph{distributions} on $\Gamma^{[a,b]}$. Let 
\[
\Distr (\Gamma ):=\underset{N}{\underset{\longrightarrow}{\lim}}\, \Distr (\Gamma^{[-N,N]}), \quad 
\Distr (\Gamma^+):=\underset{N}{\underset{\longrightarrow}{\lim}}\, \Distr (\Gamma^{[0,N]}).
\]
Then $\Distr (\Gamma )$ and $\Distr (\Gamma^+ )$ are rings with repspect to convolution; moreover, they are bialgebras over $\BZ$.

For each $n\in\BZ$ we have the functional $\BZ [u]\to\BZ$ given by evaluation at $n$; it defines an element $\delta_n\in\Distr (\Gamma )$. If $n\ge 0$ then  $\delta_n\in\Distr^+(\Gamma )$. It is clear that $\delta_m\delta_n=\delta_{m+n}$ and $\delta_0$ is the unit of $\Distr (\Gamma )$. So $\delta_{-n}=\delta_n^{-1}$.

\begin{lem}   \label{l:the dual of G_m^sharp}
(i) For every $n\ge 0$ one has $(\delta_1-\delta_0)^n\in n!\cdot \Distr (\Gamma^{[0,n]})$;

(ii) the distributions $\frac{(\delta_1-\delta_0)^n}{n!}=\frac{(\delta_1-1)^n}{n!}$, $n\ge 0$, form a basis in $\Distr (\Gamma^+)$;

(iii) $\Distr (\Gamma )$ is equal to the localization $\Distr (\Gamma )[\delta_1^{-1}]=\Distr (\Gamma )[\delta_{-1}]$.
\end{lem}

\begin{proof}  
$(\delta_1-\delta_0)^n$ is the unique element of $\Distr (\Gamma^{[0,n]})$ such that the corresponding functional  $\BZ [u]\to\BZ$ takes $u^n$ to $n!$ and $u^{n-1},\ldots, u,1$ to $0$. The value of this functional on the polynomial $f_{0,n-1}$ equals $n!$. This implies (i)-(ii). Statement (iii) follows from (ii).
 \end{proof}

\subsubsection{End of the proof}
The pairings \eqref{e:BG_m times Gamma to BG_m} and \eqref{e:BM_m times Gamma^+ to BM_m} induce bialgebra homomorphisms
\begin{equation}  \label{e:our bialgebra homomorphisms}
\Distr (\Gamma )\to \Fun (\BG_m^\sharp ) \quad \mbox{ and } \quad   \Distr (\Gamma^+ )\to \Fun (\BM_m^\sharp ),
\end{equation}
where $\Fun$ stands for the coordinate ring. The homomorphisms \eqref{e:our bialgebra homomorphisms} take $\delta_n$ to~$x^n$, where $x$ is the coordinate on $\BG_m$ or $\BM_m\,$. Lemma~\ref{l:the dual of G_m^sharp} implies that the maps \eqref{e:our bialgebra homomorphisms} are isomorphisms. Theorem~\ref{t:the dual of G_m^sharp}(i) follows.

It is easy to see that the $\BZ$-module $\Distr (\Gamma^{[-N-1,N+1]})/\Distr (\Gamma^{[-N,N]})$ is free. Therefore the $\BZ$-module $\Distr (\Gamma )$ is free. Theorem~\ref{t:the dual of G_m^sharp}(ii) follows.
\qed

\subsection{Base change to $\Spf\BZ_p$} \label{ss:over Spf Z_p}
Fix a prime $p$. Let $\Gamma$ be as in \S\ref{sss: def of Gamma}. Let
\[
\Gamma_{\BZ/p^n\BZ}:=\Gamma\times\Spec \BZ/p^n\BZ , \quad \Gamma_{\Spf\BZ_p}:=\Gamma\times\Spf\BZ_p\, .
\]
$\Gamma_{\BZ/p^n\BZ}$ is a group ind-scheme over $\BZ/p^n\BZ$, and $\Gamma_{\Spf\BZ_p}$ is a group ind-scheme over $\Spf\BZ_p\,$. The next lemma shows that these ind-schemes are \emph{formal schemes.}

\begin{lem}
$\Gamma_{\BZ/p^n\BZ}$ is the formal completion of $\BA^1_{\BZ/p^n\BZ}=\Spec (\BZ/p^n\BZ ) [u]$ along the closed subscheme of $\BA^1_{\BF_p}$ defined by the equation
$u(u-1)\ldots (u-p+1)=0$. \qed
\end{lem}

The lemma yields canonical exact sequences
\begin{equation}  \label{e:Gamma_Z/p^n}
0\to (\hat\BG_a)_{\BZ/p^n\BZ}\to \Gamma_{\BZ/p^n\BZ}\to \BZ/p\BZ\to 0,
\end{equation}
\begin{equation}   \label{e:Gamma_Z_p}
0\to (\hat\BG_a)_{\Spf\BZ_p}\to \Gamma_{\Spf\BZ_p}\to \BZ/p\BZ\to 0.
\end{equation}

\begin{rem}
If $n=1$ the exact sequence \eqref{e:Gamma_Z/p^n} has a unique splitting. If $n>1$ then \eqref{e:Gamma_Z/p^n} has no splittings.
\end{rem}

\subsection{Dualizing the exact sequence \eqref{e:Gamma_Z_p}}    \label{ss:dualizing the exact sequence}
\subsubsection{The homomorphism $\log :(\BG_m^\sharp)_{\Spf\BZ_p}\to (\BG_a)_{\Spf\BZ_p}\,$}
We have the homomorphism 
\[
\log :(\BG_m^\sharp)_{\Spf\BZ_p}\to (\BG_a)_{\Spf\BZ_p}, \quad \log x:=\sum_{n=1}^\infty(-1)^{n-1}\cdot \frac{(x-1)^n}{n}=\sum_{n=1}^\infty(-1)^{n-1}(n-1)!\cdot \frac{(x-1)^n}{n!}.
\]

Let $\BG_a^\sharp$ be the divided powers additive group, i.e., the PD hull of $0$ in $\BG_a$; as a scheme, $$\BG_a^\sharp=\Spec \BZ [y,\frac{y^2}{2!}, \frac{y^3}{3!},\ldots ].$$

\begin{lem}  \label{l:factorization of log}
The homomorphism $\log :(\BG_m^\sharp)_{\Spf\BZ_p}\to (\BG_a)_{\Spf\BZ_p}$ factors through $(\BG_a^\sharp)_{\Spf\BZ_p}\,$, so we get a homomorphism
\begin{equation} \label{e:log sharp}
\log :(\BG_m^\sharp)_{\Spf\BZ_p}\to (\BG_a^\sharp)_{\Spf\BZ_p}\, .
\end{equation}
\end{lem}

\begin{rem}
In the lemma the factorization is unique since the map $\Fun (\BG_a)\to \Fun (\BG_a^\sharp)$ becomes an isomorphism after tensoring by $\BQ$ (here $\Fun$ stands for the ring of regular functions).
\end{rem}

\begin{proof}[Proof of Lemma~\ref{l:factorization of log}]
We have to show that  $(\log x)^k$ is divisible by $k!$ in the ring of regular functions on $(\BG_m^\sharp)_{\Spf\BZ_p}$ for any $k>0$. Since $\frac{d}{dx} (\log x)^k=kx^{-1}(\log x)^{k-1}$, this follows by induction on~$k$.
\end{proof}

\begin{lem}   \label{l:mu_p in G_m^sharp}
The embedding $(\mu_p)_{\Spf\BZ_p}\mono (\BG_m)_{\Spf\BZ_p}$ comes from a unique homomorphism
\begin{equation}  \label{e:mu_p in G_m^sharp}
(\mu_p)_{\Spf\BZ_p}\mono (\BG_m^\sharp)_{\Spf\BZ_p}\, .
\end{equation}
\end{lem}

\begin{proof}
It suffices to show that $(\mu_p)_{\Spec\BZ_p}$ is a PD-thickening of the unit of $(\mu_p)_{\Spec\BZ_p}$. 
We have $(\mu_p)_{\Spec\BZ_p}=\Spec A$, where $A=\BZ_p[x]/(x^p-1)$, and the unit corresponds to the ideal $I:=(x-1)\subset A$, so the problem is to show that $f^p\in pI$ for $f\in I$.
Indeed, the image of $(x-1)^p$ in $A/pA$ is zero, so for $f\in I$ one has $f^p\in pA\cap I=pI$.
\end{proof}

\begin{rem}  \label{r:zero composition}
The composition of \eqref{e:mu_p in G_m^sharp} and \eqref{e:log sharp} is zero because $$\Hom ((\mu_p)_{\Spf\BZ_p}, (\BG_a^\sharp)_{\Spf\BZ_p})=0.$$
\end{rem}

\begin{prop}   \label{p:G_m^sharp sequence}
(i) The sequence
\begin{equation} \label{e:G_m^sharp sequence}
0\to (\mu_p)_{\Spf\BZ_p}\to (\BG_m^\sharp)_{\Spf\BZ_p}\overset{\log}\longrightarrow (\BG_a^\sharp)_{\Spf\BZ_p}\to 0,
\end{equation}
whose morphisms are \eqref{e:mu_p in G_m^sharp} and \eqref{e:log sharp}, is exact.

(ii) The exact sequence \eqref{e:G_m^sharp sequence} is Cartier dual to \eqref{e:Gamma_Z_p}; the pairing between $(\BG_m^\sharp)_{\Spf\BZ_p}$ and $\Gamma_{\Spf\BZ_p}$ is given by \eqref{e:BG_m times Gamma to BG_m}, and the pairing $(\BG_a^\sharp)_{\Spf\BZ_p}\times (\hat\BG_a)_{\Spf\BZ_p}\to\BG_m$ is the exponent of the product.
\end{prop}

\begin{proof}
It suffices to prove that the morphisms \eqref{e:mu_p in G_m^sharp} and \eqref{e:log sharp} are dual to the corresponding morphisms in the exact sequence \eqref{e:Gamma_Z_p}.
This follows from the equality $x^u=\exp (u\cdot \log x)$.
\end{proof}

In the next subsection we describe another approach to the exact sequence \eqref{e:G_m^sharp sequence}.

\subsection{A variant of  \eqref{e:G_m^sharp sequence} over $\BZ_{(p)}$}   \label{ss:variant over Z_(p)}
Let $\BZ_{(p)}$ be the localization of $\BZ$ at $p$. Base-changing $\BG_m$ and $\mu_p$ to
$\BZ_{(p)}$, one gets group schemes $(\BG_m)_{\BZ_{(p)}}$ and $(\mu_p)_{\BZ_{(p)}}$ over $\BZ_{(p)}$.
Similarly to Lemma~\ref{l:mu_p in G_m^sharp}, one sees that the embedding $(\mu_p)_{\BZ_{(p)}}\mono (\BG_m)_{\BZ_{(p)}}$ comes from a unique homomorphism
\begin{equation}  \label{e:again mu_p in G_m^sharp}
(\mu_p)_{\BZ_{(p)}}\mono (\BG_m^\sharp)_{\BZ_{(p)}}\, .
\end{equation}
We are going to describe the cokernel of \eqref{e:again mu_p in G_m^sharp}, see Proposition~\ref{p:G_m^sharp/mu_p over Z_{(p)}}.
Then we will deduce exactness of \eqref{e:G_m^sharp sequence} from this description, see \S\ref{sss:nother proof of exactness}.

\subsubsection{The group schemes $G$ and $G^\sharp$}
Let $G$ be the group scheme over $\BZ$ whose group of $A$-points is the set $\{z\in A\,|\, 1+pz\in A^\times\}$ equipped with the operation 
$z_1*z_2:=z_1+z_2+pz_1z_2$; in other words, $G$ is the $p$-rescaled version of $\BG_m$. We have a canonical homomorphism 
\begin{equation} \label{e:G to G_m}
G\to\BG_m, \quad z\mapsto 1+pz. 
\end{equation}
As usual, let $G^\sharp$ be the divided powers version of $G$ (i.e., the PD hull of the unit in $G$).

\begin{lem}   \label{l:G_m^sharp to G^sharp}
There is a unique homomorphism
\begin{equation}  \label{e:G_m^sharp to G^sharp}
\BG_m^\sharp\to G^\sharp
\end{equation}
such that the diagram
\[
\xymatrix{
\BG_m^\sharp\ar[r] \ar[rd]_p &G^\sharp\ar[d] \\
& \BG_m^\sharp 
}
\]
commutes; here the vertical arrow comes from \eqref{e:G to G_m}.
\end{lem}

\begin{proof}
As above, let $z$ be the coordinate on $G$. Let $x$ be the usual coordinate on $\BG_m$ and $t:=x-1$. The homomorphism
$(\BG_m^\sharp)_\BQ\overset{p}\longrightarrow(\BG_m^\sharp)_\BQ=G^\sharp_\BQ$ is given by $z=\frac{(1+t)^p-1}{p}$. The problem is
to check that
 $\frac{(1+t)^p-1}{p}=\sum\limits_{i=1}^p m_i\gamma_i(t)$ for some $m_i\in\BZ$ (here $\gamma_i$ is the $i$-th divided power).
 This is clear.
\end{proof}

\begin{prop}   \label{p:G_m^sharp/mu_p over Z_{(p)}}
The homomorphism $(\BG_m^\sharp)_{\BZ_{(p)}}\to (G^\sharp)_{\BZ_{(p)}}$ corresponding to \eqref{e:G_m^sharp to G^sharp} induces an isomorphism
$(\BG_m^\sharp)_{\BZ_{(p)}}/(\mu_p)_{\BZ_{(p)}}\iso (G^\sharp)_{\BZ_{(p)}}\,$.
\end{prop}

For a proof, see \S\ref{ss:G_m^sharp/mu_p over Z_{(p)}}.

\subsubsection{Passing to formal completions}
(i) Let $\hat G,\hat\BG_a$ be the formal completions of the group schemes $G,\BG_a$ along their units; these are formal groups over $\BZ$. Let 
$\hat G_{\BZ_{(p)}},(\hat\BG_a)_{\BZ_{(p)}}$ be the corresponding formal groups over $\BZ_{(p)}$. One has an isomorphism
\begin{equation}       \label{e:isomorphism of formal groups}
\hat G_{\BZ_{(p)}}\iso (\hat\BG_a)_{\BZ_{(p)}}, \quad z\mapsto\frac{\log (1+pz)}{p}:=\sum_{n=1}^\infty \frac{(-p)^{n-1}}{n}\cdot z^n.
\end{equation}

(ii) The isomorphism \eqref{e:isomorphism of formal groups} induces an isomorphism
\begin{equation}       \label{e:isomorphism of group schemes over Spf Z_p}
G^\sharp_{\Spf\BZ_p}\iso (\BG_a^\sharp)_{\Spf\BZ_p}\,  
\end{equation}
because one can think of $G^\sharp_{\Spf\BZ_p}$ (resp.~$(\BG_a^\sharp)_{\Spf\BZ_p}$) as the $p$-adically completed PD version of 
$\hat G_{\BZ_{(p)}}$ (resp.~ $(\hat\BG_a)_{\BZ_{(p)}}$).

Note that \eqref{e:isomorphism of formal groups} is an isomorphism of \emph{formal} groups over the \emph{scheme} $\Spec\BZ_{(p)}\,$, while \eqref{e:isomorphism of group schemes over Spf Z_p} is an isomorphism of group \emph{schemes} over the \emph{formal scheme} $\Spf\BZ_p\,$.

\subsubsection{A proof of exactness of \eqref{e:G_m^sharp sequence}}   \label{sss:nother proof of exactness}
Using that $\log (x^p)=p\cdot\log x$, one checks that the homomorphism $\log :(\BG_m^\sharp)_{\Spf\BZ_p}\to (\BG_a^\sharp)_{\Spf\BZ_p}$ from \eqref{e:G_m^sharp sequence} equals the composite map
\[
(\BG_m^\sharp)_{\Spf\BZ_p}\to G^\sharp_{\Spf\BZ_p}\iso (\BG_a^\sharp)_{\Spf\BZ_p},
\]
where the first arrow comes from \eqref{e:G_m^sharp to G^sharp} and the second one is \eqref{e:isomorphism of group schemes over Spf Z_p}. 
So exactness of \eqref{e:G_m^sharp sequence} follows from Proposition~\ref{p:G_m^sharp/mu_p over Z_{(p)}}.

\subsection{Proof of Proposition~\ref{p:G_m^sharp/mu_p over Z_{(p)}}}  \label{ss:G_m^sharp/mu_p over Z_{(p)}}

\subsubsection{Straightforward proof}
The kernel of the homomorphism $(\BG_m^\sharp)_{\BZ_{(p)}}\to (G^\sharp)_{\BZ_{(p)}}$ equals $(\mu_p)_{\BZ_{(p)}}$. The problem is to show that the homomorphism is faithfully flat.

We will use the coordinates $z$ and $t$ from the proof of Lemma~\ref{l:G_m^sharp to G^sharp}. We have
\begin{equation}  \label{e:z in terms fo t}
z=\frac{(1+t)^p-1}{p}=\gamma (t)+\sum_{i=1}^{p-1} n_it^i \quad\mbox{ for some } n_i\in\BZ,
\end{equation}
where $\gamma (t):=\frac{t^p}{p}$.

The coordinate ring of $(G^\sharp)_{\BZ_{(p)}}$ is $A[\frac{1}{1+pz}]$, where $A:=\BZ_{(p)}[z,\gamma (z),\gamma^2(z),\ldots ]\subset\BQ [z]$.
The coordinate ring of $(\BG_m^\sharp)_{\BZ_{(p)}}$ is $B[\frac{1}{1+t}]=B[\frac{1}{1+pz}]$, where $B:=\BZ_{(p)}[t,\gamma (t),\gamma^2(t),\ldots ]\subset\BQ [t]$.
It suffices to show that the homomorphism $A\to B$ given by \eqref{e:z in terms fo t} makes $B$ into a free $A$-module with basis $1, t,\ldots,t^{p-1}$. These elements form a basis of $B\otimes\BQ=\BQ [t]$ over $A\otimes\BQ=\BQ [z]$, so we only have to check that $1, t,\ldots,t^{p-1}$ generate $B$ as an $A$-module. Note that as a $\BZ_{(p)}$-module, $B$ is generated by elements
\[
\prod_{i=0}^\infty (\gamma^i(t))^{m_i}, \quad\mbox{ where } 0\le m_i<p \mbox{ and } m_i=0 \mbox{ for } i\gg 0.
\]
By \eqref{e:z in terms fo t}, $\prod\limits_{i=0}^\infty (\gamma^i(t))^{m_i}=t^{m_0}\cdot \prod\limits_{i>0} (\gamma^{i-1}(z))^{m_i}+\mbox{\{lower terms}\}$, so we can proceed by induc\-tion. \qed

\subsubsection{Proof via Cartier duality (sketch)}
One can also prove Proposition~\ref{p:G_m^sharp/mu_p over Z_{(p)}} by passing to the Cartier duals. 
Similarly to Theorem~\ref{t:the dual of G_m^sharp}, the Cartier dual of $G^\sharp$ identifies with the group ind-scheme $\Gamma_p$ whose definition is parallel to that of $\Gamma$ (see \S\ref{sss: def of Gamma}) but with the polynomial $\prod\limits_{i=a}^{b} (u-i)$ from formula \eqref{e:f_{a,b}} being replaced by $\prod\limits_{i=a}^{b} (u-pi)$. Details are left to the reader.

\section{The Cartier dual of $\hat\BG_m$}  \label{s:dual of hat G_m}
Let $\hat\BG_m$ denote the formal multiplicative group over $\BZ$. For any ring $A$ one has
\[
\hat\BG_m (A)=\{y\in A^\times\,|\, y-1 \mbox{ is nilpotent}\}.
\] 
In this section we give two descriptions of the Cartier dual of $\hat\BG_m$, see \S\ref{ss:R in terms of Int} and \S\ref{ss:R via W_big}.
They are probably well known: the description from  \S\ref{ss:R via W_big} is contained in \cite{MRT}, and the one from \S\ref{ss:R in terms of Int} was known to T.~Ekedahl (see Remark 4 on p.~197 of \cite{Ek}).

\subsection{The Cartier dual in terms of the ring of integer-valued polynomials}   \label{ss:R in terms of Int}
\subsubsection{The ring scheme $\sR$}
Let $\sR:=\HHom (\hat\BG_m ,\hat\BG_m)$. This is a unital ring scheme over~$\Spec\BZ$. The action of $\sR$ on $\Lie (\hat\BG_m )$ defines a homomorphism of ring schemes
\begin{equation} \label{e:R to G_a}
\sR\to\BG_a
\end{equation}
(the multiplication operation in $\BG_a$ is the usual one). The coordinate ring of $\BG_a$ equals $\BZ [u]$, so \eqref{e:R to G_a} induces a ring homomorphism 
\begin{equation}   \label{e:Z[u] to our ring}
\BZ [u]\to H^0(\sR ,\cO_{\sR}).
\end{equation}

As a group scheme, $\sR$ equals $\HHom (\hat\BG_m ,\BG_m)$, i.e., the Cartier dual of~$\hat\BG_m$. So $\sR$ is a flat affine scheme over $\Spec\BZ$.

By Lie theory, the homomorphism \eqref{e:R to G_a} induces an isomorphism
\begin{equation} \label{e:iso after censoring by Q}
\sR\otimes\BQ\iso\BG_a\otimes\BQ.
\end{equation}
The action of $\Spec\BQ [u]=\BG_a\otimes\BQ=\sR\otimes\BQ$ on $\hat\BG_m\otimes\BQ$ is given by Newton's binomial formula
\begin{equation}  \label{e:2binomial formula}
y^u=\sum_{n=0}^\infty \binom{u}{n}(y-1)^u, \quad \binom{u}{n}:=\frac{u(u-1)\ldots (u-n+1)}{n!}\in\BQ [u].
\end{equation}

The ring scheme $\sR$ is commutative by virtue of \eqref{e:iso after censoring by Q} and flatness of $\sR$ over $\Spec\BZ$.

The homomorphism \eqref{e:Z[u] to our ring} becomes an isomorphism after tensoring by $\BQ$. So
\[
\BZ [u]\subset H^0(\sR ,\cO_{\sR})\subset\BQ [u].
\]
The homomorphism $H^0(\sR ,\cO_{\sR})\to H^0(\sR ,\cO_{\sR})\otimes H^0(\sR ,\cO_{\sR})$ corresponding to addition (resp.~multiplication) in $\sR$ takes $u$ to $u\otimes 1+1\otimes u$ (resp.~ to $u\otimes u$). To finish the explicit description of $\sR$, it remains to describe the subring $H^0(\sR ,\cO_{\sR})\subset\BQ [u]$.

\begin{prop}   \label{p:Newton's description of sR}
$H^0(\sR ,\cO_{\sR})=\Int$, where $\Int\subset\BQ [u]$ is the subring generated by the polynomials  $\binom{u}{n}$, $n\ge0$.
\end{prop}

\begin{proof}
$H^0(\sR ,\cO_{\sR})$ is the smallest subring $A\subset\BQ [u]$ such that the action of $\Spec\BQ [u]=\sR\otimes\BQ$ on $\hat\BG_m$ extends to an action of $\Spec A$ on $\hat\BG_m$. So $H^0(\sR ,\cO_{\sR})$ is generated by the coefficients of the formal series \eqref{e:2binomial formula}.
\end{proof}

\subsubsection{On the ring $\Int$}
It is well known that 
\[
\Int=\{ f\in\BQ [u]\,|\, f(m)\in\BZ \mbox{ for all }m\in\BZ\}; 
\]
for this reason, $\Int$ is
known as the \emph{ring of integer-valued polynomials}. It is also well known that

(i) the polynomials  $\binom{u}{n}$ form a \emph{basis} of the $\BZ$-module $\Int$;

(ii) one has
\begin{equation}      \label{e:difference equation}
\Int =\{f\in \Fun (\BZ ,\BZ )\,|\,\Delta^m (f)=0 \mbox{ for some }m\},
\end{equation}
where $\Fun (\BZ ,\BZ )$ is the ring of all functions $\BZ\to\BZ$ and $\Delta :\Fun (\BZ ,\BZ )\to \Fun (\BZ ,\BZ )$ is
the \emph{difference operator} $\Delta :\Fun (\BZ ,\BZ )\to \Fun (\BZ ,\BZ )$ defined by
\[
(\Delta f)(u)=f(u+1)-f(u).
\]

More details about the ring $\Int$ and some references can be found in \cite{CC,Ch,Ek,El}.

\subsubsection{Remark}
Here is an interpretation of \eqref{e:difference equation} via Cartier duality between $\sR$ and $\hat\BG_m$.

The Cartier dual of the embedding $\hat\BG_m\mono\BG_m$ is a morphism $\BZ\times\Spec\BZ\to\sR$, and the embedding
\begin{equation}  \label{e:Int to Fun(Z)}
H^0(\sR,\cO_{\sR})=\Int\mono\Fun (\BZ ,\BZ )
\end{equation}
is  the corresponding homomorphism of coordinate rings. As a $\BZ$-module, $H^0(\sR,\cO_{\sR})$ is the topological dual $(\BZ [[y-1]])^*$, and the map
\eqref{e:Int to Fun(Z)} is just the natural map $$\varphi :(\BZ [[y-1]])^*\to (\BZ [y,y^{-1}])^*.$$ 
So \eqref{e:difference equation} means that $\varphi$ is injective, and  $\im\varphi$ consists of those linear functionals on $\BZ [y,y^{-1}]$ that are trivial on $(y-1)^m\BZ [y,y^{-1}]$ for some $m$. This is, of course, true because $\BZ [[y-1]]$ is the $(y-1)$-adic completion of $\BZ [y,y^{-1}]$.

\subsection{The reduction of the scheme $\sR$ modulo $p^n$ and the $\lambda$-ring structure on $\Int$}
\subsubsection{The reduction of $\sR$ modulo $p^n$}
Let $p$ be a prime. If $A$ is a ring in which $p$ is nilpotent then $\hat\BG_m\otimes A$ is the inductive limit of $\mu_{p^n}\otimes A$. The Cartier dual of $\mu_{p^n}$ is 
$\BZ/p^n\BZ$. So 
\begin{equation} \label{e:2Mahler}
\sR\otimes A=(\BZ_p)_A:=\underset{n}{\underset{\longleftarrow}{\lim}} (\BZ/p^n\BZ)_A,
\end{equation}
where $(\BZ/p^n\BZ)_A$ is the constant ring scheme over $\Spec A$ with fiber $\BZ/p^n\BZ$. 

\subsubsection{Mahler's theorem}
Let $A=\BZ/p^n\BZ$. Combining \eqref{e:2Mahler} with the equality $H^0(\sR,\cO_{\sR})=\Int$, we get an isomorphism
\begin{equation} \label{e:3Mahler}
\Int/p^n\Int\iso \{\mbox{Locally constant functions }\BZ_p\to\BZ/p^n\BZ\};
\end{equation}
the map \eqref{e:3Mahler} is as follows: given a function $f\in\Int\subset\Fun (\BZ ,\BZ)$, we reduce it modulo $p^n$ and then extend from $\BZ$ to $\BZ_p$ by continuity.
The isomorphism \eqref{e:3Mahler} is due to K.~Mahler~\cite{Ma}. It is discussed, e.g., in \cite[Ch.~4]{La}. 

\begin{lem}   \label{l:Fr=id}
For every prime $p$, the Frobenius endomorphism of $\Int /p\Int$ equals the identity.
\end{lem}

This well known fact follows from \eqref{e:3Mahler} or from \eqref{e:difference equation}.

\subsubsection{Wilkerson's theorem on $\lambda$-rings}  \label{sss:Wilkerson}
Any $\lambda$-ring $R$ is equipped with an action of the multiplicative monoid $\BN$; the endomorphism of $R$ corresponding to $n\in\BN$ is denoted by $\psi^n$ and called the $n$-th \emph{Adams operation}. So we get a functor from the category of $\lambda$-rings to that of rings equipped with $\BN$-action. C.~Wilkerson \cite{W} proved that this functor identifies the category of torsion-free $\lambda$-rings with the category of torsion-free rings equipped with an action of $\BN$ satisfying the following condition: $\psi^p (x)$ is congruent to $x^p$ modulo $p$ for every prime $p$ and every $x\in R$.

\subsubsection{The $\lambda$-ring structure on $\Int$}   \label{sss:lambda-ring structure on Int}
By \S\ref{sss:Wilkerson}, a torsion free ring $R$ such that for every prime $p$ the Frobenius endomorphism of $R/pR$ equals the identity is the same as a torsion-free $\lambda$-ring such that $\psi^n=\id$ for all $n$. It is known (see \cite{W,El}) that for such $R$ one has
\begin{equation}  \label{e:lambda-operations}
\lambda_n(x)=\frac{x(x-1)\ldots (x-n+1)}{n!}\quad \mbox{for all } n\in\BN, x\in R.
\end{equation}
By Lemma~\ref{l:Fr=id}, this applies to the ring $\Int$. On the other hand, in the case $R=\Int$ the $\lambda$-ring structure comes from the embedding $\Int\mono\Fun (\BZ ,\BZ)=\BZ\times\BZ\times\ldots$ and the $\lambda$-ring structure on $\BZ$, so \eqref{e:lambda-operations} is clear.

\subsubsection{Generators of $\Int\otimes\BZ_{(p)}$}  \label{sss:generators of Int otimesZ_p}
Fix a prime $p$. Let $\BZ_{(p)}$ be the localization of $\BZ$ at $p$ and $\Int_{(p)}:=\Int\otimes\BZ_{(p)}\subset\BQ [u]$. For $x\in\Ind_{(p)}$ set $\delta (x):=(x-x^p)/p$; then $\delta (x)\in\Int_{(p)}$ by Lemma~\ref{l:Fr=id}. The pair $(\Int_{(p)}, \delta :\Int_{(p)}\to\Int_{(p)})$ is a $\delta$-ring in the sense of \cite{JoyalDelta} and \cite{BS}. The following lemma is well known (e.g., see \cite[\S 3]{El}).

\begin{lem}   \label{l:generators of Int otimesZ_p}
(i) The elements $\delta^n(u)$, $n\in\BZ_+$, generate $\Int_{(p)}$ as an $\BZ_{(p)}$-algebra.

(ii) Elements of the form
\[
\prod_i (\delta^i(u))^{d_i}, \quad \mbox{where } 0\le d_i<p  \mbox{ for all } i \mbox{ and } d_i=0 \mbox{ for }i\gg 0 
\]
form a basis of the $\BZ_{(p)}$-module $\Int_{(p)}$.
\end{lem}

\begin{proof}
It suffices to proof (ii). Let $n\ge 0$ be an integer. Write $n=\sum\limits_i d_ip^i$, where $0\le d_i<p$ for all $i$ and $d_i=0$ for $i\gg 0 $. There exists $c\in\BQ$ such that the polynomial 
$$\binom{u}{n}-c\prod\limits_i (\delta^i(u))^{d_i}$$
has degree $<n$. It remains to check that $c\in\BZ_{(p)}$. To do this, use that $n!\in p^m\cdot\BZ_{(p)}^\times$, where $m=\sum\limits_id_i(p^{i-1}+\ldots +p+1)$.
\end{proof}

\subsection{The ring scheme $\sR$ via Witt vectors}  \label{ss:R via W_big}
\subsubsection{The ring scheme $W_{\bbig}$}  \label{sss:W_big}
Let $W_{\bbig}$ be the ring scheme of ``big'' Witt vectors. Recall that for any ring $A$, the additive group of $W_{\bbig}(A)$ is the subgroup of $A[[z]]^\times$ that consists of all power series with constant them $1$. For each $n\in\BZ$ one has the \emph{Witt vector Frobenius} map $F_n:W_{\bbig}\to W_{\bbig}$, which is a ring scheme endomorphism; one has $F_mF_n=F_{mn}$ and $F_1=\id$. Recall  that the unit of $W_{\bbig}(A)$ corresponds to $1-z\in A[[z]]^\times$.

\subsubsection{The map $\sR\to W_{\bbig}$}  \label{sss:R to W}
By definition, an $A$-point of $\sR$ is an element $f\in A[[y-1]]$ satisfying the functional equation 
\begin{equation} \label{e:f(y_1y_2)}
f(y_1y_2)=f(y_1)f(y_2).
\end{equation}
 Associating to such $f$ the formal power series 
$f(1-z)\in A[[z]]^\times$, we get a group homomorphism $\sR (A)\to W_{\bbig}(A)$ functorial in $A$, i.e., a homomorphism of group schemes
\begin{equation}  \label{e:R to W}
\sR\to W_{\bbig}\, .
\end{equation}
This morphism is a closed immersion because \eqref{e:f(y_1y_2)} is a closed condition. Note that the map \eqref{e:R to W} takes $1\in\sR (\BZ)$ to $1\in W_{\bbig}(\BZ )$ (see the end of \S\ref{sss:W_big}).

\subsubsection{Remark}
Here is a slightly different way of thinking about \eqref{e:R to W}. Consider the unique homomorphism of unital rings $f:\BZ\to W_{\bbig} (\BZ )$. Then each component of the Witt vector $f(n)$ is an (integer-valued) polynomial in $n$, so we get an element of $W_{\bbig} (\Int )$, i.e., a morphism $\Spec\Int\to W_{\bbig}$. This is  \eqref{e:R to W}. 

\begin{prop}   \label{p:sR= W_{big}^F}
(i) The map \eqref{e:R to W} is a homomorphism of ring schemes.

(ii) It induces an isomorphism $\sR\iso W_{\bbig}^F$, where
\begin{equation}   \label{e:W^F_big}
W_{\bbig}^F:=\{w\in W_{\bbig} \,|\, F_n (w)=w \mbox{ for all } n\in\BN\}.
\end{equation}
\end{prop}

\begin{proof}
We know that $\sR$ is flat over $\Spec\BZ$ and the morphism $\sR\to W_{\bbig}$ is a closed immersion. It is straightforward to check (i) and (ii) after base change to
$\Spec\BQ$. It remains to show that $W_{\bbig}^F$ is flat over $\Spec\BZ$. This follows from Lemmas~\ref{l:W_big & W}-\ref{l:flatness of W^F} below.
\end{proof}

\begin{lem}   \label{l:W_big & W}
Let $p$ be a prime and $W$ the ring scheme of $p$-typical Witt vectors. Let $F:W\to W$ be the Witt vector Frobenius and 
\begin{equation}   \label{e:W^F}
W^F:=\{w\in W \,|\, F (w)=w \}.
\end{equation}
Then the natural ring scheme morphism $W_{\bbig}\to W$ induces an isomorphism 
\begin{equation}
W_{\bbig}^F\otimes\BZ_{(p)}\iso W^F\otimes\BZ_{(p)}
\end{equation}
\end{lem}

\begin{proof}
The proof is based on the identification of $W_{\bbig}\otimes\BZ_{(p)}$ with the product of infinitely many copies of $W\otimes\BZ_{(p)}$ (the copies are labeled by positive integers coprime to $p$) and the usual description of the morphisms $F_n:W_{\bbig}\otimes\BZ_{(p)}\to W_{\bbig}\otimes\BZ_{(p)}$ in terms of this identification.
\end{proof}

\begin{lem}   \label{l:flatness of W^F}
The scheme $W^F$ defined by \eqref{e:W^F} is flat over $\BZ$.
\end{lem}

Before proving the lemma, let us briefly recall the approach to $W$ developed by Joyal \cite{JoyalDelta}  (a detailed exposition of this approach can be found in \cite{BorgerCourse} and \cite[\S 1]{BorgerGurney}).

\subsubsection{Joyal's approach to $W$}  \label{sss:Joyal}
Let $C$ be the coordinate ring of $W$. 
Let $\phi:C\to C$ be the homomorphism corresponding to $F:W\to W$. 
The map $W\otimes\BF_p\to W\otimes\BF_p$ induced by $F$ is the usual Frobenius, so there is a map $\delta:C\to C$ such that $\phi(c)=c^p+p\delta (c)$ for all $c\in C$ (of course, the map $\delta$ is neither additive nor multiplicative). 

The pair $(C,\delta)$ is a $\delta$-ring in the sense of \cite{JoyalDelta} and \cite[\S 2]{BS}. The main theorem of \cite{JoyalDelta} says that $C$ is the \emph{free $\delta$-ring} on $x_0$, where $x_0\in C$ corresponds to the canonical homomorphism $W\to W/VW=\BG_a$. This means that as a ring, $C$ is freely generated by the elements $x_n:=\delta^n(x_0)$, $n\ge 0$. We have 
\begin{equation}   \label{e:phi(y_n)}
\phi(x_n)=x_n^p+px_{n+1}
\end{equation}

The elements $x_n$ (which are regular functions on $W$) are called \emph{Buium-Joyal coordinates} or \emph{Buium-Joyal components} (this terminology is introduced in \cite{BorgerGurney}).  For $n>1$ they are different from Witt components (i.e., the usual ones).

\subsubsection{Proof of Lemma~\ref{l:flatness of W^F}} \label{sss:proof of flatness of W^F}
Let $C$ be as in \S\ref{sss:Joyal}. Formula~\eqref{e:phi(y_n)} implies that the coordinate ring of $W^F$ is the quotient of $C$ by the ideal $I$ generated by the elements
\begin{equation}  \label{e:equations for W^F}
x_n^p+px_{n+1}-x_n, \quad n\in\BZ_+.
\end{equation}
This quotient is a free $\BZ$-module whose basis is formed by elements $\prod\limits_ix_i^{d_i}$, where $0\le d_i<p$ for all $i$ and $d_i=0$ for $i\gg 0 $. Indeed, these elements clearly generate $C/I$, and they are linearly independent in $(C/I)\otimes\BQ=\BQ [x_0]$. \qed

\section{The rescaled $\hat\BG_m$ and its Cartier dual}  \label{s:rescaled G_m and its Cartier dual}
As noted by the reviewer, a substantial part of this Appendix and the previous one is contained in \cite{MRT}.

\subsection{Rescaling $\hat\BG_m$}
\subsubsection{The formal group $H^!$}
Let $H^!$ be the formal group scheme over $\BA^1=\Spec\BZ [h]$ defined by the formal group law
\[
z_1*z_2=z_1+z_2+hz_1z_2.
\]
Note that $1+h\cdot (z_1*z_2)=(1+hz_1)(1+hz_2)$. So we have a homomorphism of formal groups over $\BA^1$
\begin{equation}  \label{e:H to hat G_m}
H^!\to\hat\BG_m\times\BA^1 , \quad z\mapsto 1+hz ,
\end{equation}
which induces an isomorphism over the locus $h\ne 0$.

After specializing $h$ to $1$ and $0$, the formal group $H^!$ becomes  $\hat\BG_m$ and $\hat\BG_a$, respectively. If you wish, $H^!$ is a deformation of $\hat\BG_m$ to $\hat\BG_a$.

\subsubsection{Remarks}
(i) The action of $\BG_m$ on $\BA^1$ by multiplication lifts to an action of $\BG_m$ on~$H^!$: namely, $\lambda\in\BG_m$ takes $(h,z)$ to $(\lambda h, \lambda^{-1}z)$.
So $H^!$ descends from $\BA^1$ to the quotient stack~$\BA^1/\BG_m$.

(ii) $H^!$ is obtained from $\hat\BG_m$ by rescaling depending on a parameter $h$. This is a particular case of the construction of \S\ref{ss:Deformation to Lie algebra}.

\subsubsection{Plan}
In \S\ref{ss:G in terms of Newton&Euler} (which is parallel to  \S\ref{ss:R in terms of Int}) we give a description of the Cartier dual $G^!$ of $H^!$.
In \S\ref{ss:G in terms of Witt}-\ref{ss:G^! in terms of Witt} we describe $G^!$ in terms of Witt vectors in two different ways; 
the description from  \S\ref{ss:G in terms of Witt} is quite parallel to  \S\ref{ss:R via W_big}.
In \S\ref{ss:lambda-ring structure on B_0} we discuss a certain $\lambda$-ring structure on the coordinate ring of $G^!$.

\subsection{The first description of the Cartier dual of $H^!$}  \label{ss:G in terms of Newton&Euler}
\subsubsection{The group scheme $G^!$}
Let $G^!$ be the Cartier dual of $H^!$; this is a flat affine scheme over $\BA^1=\Spec\BZ [h]$. The group scheme $\sR$ from \S\ref{ss:R in terms of Int} can be obtained from $G^!$ by specializing $h$ to~$1$. In this subsection we describe $G^!$ in the spirit of \S\ref{ss:R in terms of Int}. Later we will give two different descriptions of $G^!$
in terms of Witt vectors (see Propositions~\ref{p:G in terms of Witt} and \ref{p:G^!=G^!!}).

For any $\BZ [h]$-algebra $A$, an $A$-point of $G^!$ is a formal series $f\in 1+zA[[z]]\subset (A[[z]])^\times$ such that $f(z_1*z_2)=f(z_1)f(z_2)$. Associating $f'(0)$ to such $f$, we get a homomorphism
\begin{equation}   \label{e:G to G_a}
G^!\to\BG_a\times\BA^1
\end{equation}
of group schemes over $\BA^1=\Spec\BZ [h]$.
The coordinate ring of $\BG_a$ equals $\BZ [t]$, so \eqref{e:G to G_a} induces a homomorphism of $\BZ[h]$-algebras
\begin{equation}   \label{e:2 Z[u] to our ring}
\BZ [h,t]\to H^0(G^! ,\cO_{G^!}).
\end{equation}

By Lie theory, the homomorphism \eqref{e:G to G_a} becomes an isomorphism after base change to $\BA^1_\BQ=\Spec\BQ [h]$.
So we have a pairing $\BG_a\times H^!_{\BA^1_\BQ}\to\BG_m\times\BA^1_\BQ$, where $H^!_{\BA^1_\BQ}:=H^!\times_{\BA^1}\BA^1_\BQ$. The corresponding map
$\BG_a\times H^!_{\BA^1_\BQ}\to\BG_m$ is given by the formal series
\begin{equation}   \label{e:Euler series}
(1+hz)^{t/h}=\sum_{n=0}^\infty \frac{t(t-h)\ldots (t-h(n-1))}{n!}\cdot z^n\in (\BQ [h,t][[z]])^\times ,
\end{equation}
where $t$ is the coordinate on $\BG_a$ and $z$ is the coordinate on $H^!$. Note that after substituting $h=0$ the formal series \eqref{e:Euler series} becomes equal to $\exp (tz)$.

The homomorphism \eqref{e:2 Z[u] to our ring} becomes an isomorphism after tensoring by $\BQ$. Since $G^!$ is flat over $\BZ [h]$, we see that
$
\BZ [h,t]\subset H^0(G^! ,\cO_{G^!})\subset\BQ [h,t].
$

\begin{prop}   \label{p:G=Spec B_0}
(i) $H^0(G^! ,\cO_{G^!})=B_0$, where $B_0\subset\BQ [h,t]$ is the subring generated over $\BZ [h]$ by the polynomials  
\begin{equation}  \label{e:h-binomial coeffciients}
\frac{t(t-h)\ldots (t-h(n-1))}{n!}, \quad n\ge 0. 
\end{equation}

(ii) The polynomials \eqref{e:h-binomial coeffciients} form a basis of the $\BZ [h]$-module $B_0$.

(iii) The Hopf algebra structure on $B_0$ corresponding to the group structure on $G^!$ is given by $t\mapsto t\otimes 1+1\otimes t$.
\end{prop}

\begin{proof}   
The proof of (i) is parallel to that of Proposition~\ref{p:Newton's description of sR}. 

Let us prove (ii). The product of two polynomials of the form \eqref{e:h-binomial coeffciients} can be represented as an $\BZ [h]$-linear combination of such polynomials using the formula
\[
(1+hz_1)^{t/h}(1+hz_2)^{t/h}=(1+h(z_1*z_2))^{t/h}, \quad \mbox{ where } z_1*z_2=z_1+z_2+hz_1z_2.
\]
So polynomials of the form \eqref{e:h-binomial coeffciients} generate $B_0$ as a $\BZ [h]$-module. They are linearly independent over $\BZ [h]$ because the polynomial \eqref{e:h-binomial coeffciients} has degree $n$ with respect to $u$.

Finally, (iii) is clear because $t$ is the pullback via \eqref{e:G to G_a} of the natural coordinate on~$\BG_a$.
\end{proof}

The following simple lemma is used in the proof of Proposition~\ref{p:G_Q^!=SpfB}(c$'$).

\begin{lem}  \label{l:simple lemma}
Let $m\in\BN$. Then

(i) the homomorphism $g_m:B_0\to B_0$ induced by the morphism $G^!\overset{m}\longrightarrow G^!$ takes $t$ to $mt$;

(ii)  $\frac{mt(mt-h)\ldots (mt-(n-1)h)}{n!}\in B_0$ for all $n$;

(iii) in $B_0[[z]]$ one has the equality
\[
\sum_{n=0}^\infty\frac{mt(mt-h)\ldots (mt-(n-1)h)}{n!}\cdot z^n=(1+hvz)^{\frac{t}{h}},
\]
where $v:=\frac{(1+hz)^m-1}{hz}\in\BZ [h,z]$ and $(1+hvz)^{\frac{t}{h}}:=\sum\limits_{n=0}^\infty\frac{t(t-h)\ldots (t-(n-1)h)}{n!}\cdot v^nz^n$.
\end{lem}

\begin{proof}
Statement (i) follows from Proposition~\ref{p:G=Spec B_0}(iii). The expression from (ii) is just $g_m(t)$, where $g_m$ is as in (i); so (ii) is clear.
In statement (iii)  one can replace $B_0$ by $B_0\otimes\BQ=\BQ [h,t]$, so (iii) is classical.
\end{proof}

\subsubsection{The homomorphism $\sR\times\BA^1\to G^!$}
Recall that $\sR$ is the Cartier dual of $\hat\BG_m$. So the Cartier dual of \eqref{e:H to hat G_m} is a homomorphism
\begin{equation}  \label{e:R to G}
\sR\times\BA^1\to G^!
\end{equation}
of group schemes over $\BA^1=\Spec\BZ [h]$, which induces an isomorphism over the locus $h\ne 0$.

\subsubsection{Relation between $B_0$ and $\Int$} \label{sss:B_0 in terms of Int}
The map \eqref{e:R to G} induces a homomorphism of $\BZ [h]$-algebras 
\begin{equation} \label{e:B_0 to Int[h]}
B_0\to\Int [h], 
\end{equation}
which becomes an isomorphism after base change to $\BZ [h,h^{-1}]$. The homorphism \eqref{e:B_0 to Int[h]} takes $t$ to $hu$, so the polynomial \eqref{e:h-binomial coeffciients} goes to $h^n\binom{u}{n}$. 

Equip $\BQ [h,t]$ with the grading such that $\deg h=\deg t=1$, then $B_0$ is a graded subring of $\BQ [h,t]$. Equip $\Int [h]$ with the grading such that $\deg h=1$ and all elements of $\Int$ have degree $0$. Then the homomorphism \eqref{e:B_0 to Int[h]} is graded.

The subring $\Int\subset\BQ [u]$ from Proposition~\ref{p:Newton's description of sR} is filtered by degree of polynomials. Let $\Int_{\le n}$ be the $n$-th term of this filtration. It is easy to see that \eqref{e:B_0 to Int[h]} induces an isomorphism
\begin{equation}  \label{e:B_0 in terms of Int}
B_0\iso\bigoplus_n h^n\Int_{\le n}.
\end{equation}
Thus the graded $\BZ [h]$-algebra $B_0$ is obtained from the filtered ring $\Int$ by a very familiar procedure.

\subsubsection{Remarks}  \label{sss:remarks on B_0 in terms of Int}
(i) $B_0/hB_0=\gr\Int$ is the ring of divided powers polynomials in $u$.

(ii) One can rewrite \eqref{e:B_0 in terms of Int} as an isomorphism
\begin{equation}   \label{e:2 B_0 in terms of Int}
B_0\iso \Int [h]\cap\BQ [h,hu]\subset\BQ [h,u];  
\end{equation}
under this isomorphism $t\in B_0$ corresponds to $hu\in\BQ [h,u]$. 
Note that \eqref{e:2 B_0 in terms of Int} induces an isomorphism
$B_0\otimes\BQ\iso\BQ [h,hu]\subset\BQ [h,u]$.
\

\subsection{$\lambda$-ring structure on $B_0$}  \label{ss:lambda-ring structure on B_0}
\subsubsection{Notation}
Let $q:=h+1\in\BZ [h]$; then $\BZ [h]=\BZ [q]$. 

\subsubsection{A $\lambda$-ring structure on $\BZ [h]$}   \label{sss:lambda-ring structure on Z [h]}
By \S\ref{sss:Wilkerson}, there is a unique $\lambda$-ring structure on $\BZ [h]=\BZ [q]$ such that $\psi^n(q)=q^n$ for all $n\in\BN$.

Another way to get this $\lambda$-ring structure is to chose a field $k$ and to identify $\BZ [q, q^{-1}]$ (resp.~$\BZ [q]$) with the Grothendieck ring of the category of finite-dimensional representations of $(\BG_m)_k$ (resp.~of the multiplicative monoid over $k$) so that $q$ identifies with the class of the tautological $1$-dimensional representation.

Let us note that the $\lambda$-ring $\BZ [q]$ is studied in Pridham's article  \cite{Pri}.

\medskip

In the next lemma we define a $\lambda$-ring structure on $B_0$; the definition will be motivated by Lemma~\ref{l:motivation of lambda-structure}(ii).

\begin{lem} \label{l:B_0 as lambda-ring}
Consider $B_0$ as a graded ring (see \S\ref{sss:B_0 in terms of Int}). For $n\in\BN$ let $\psi^n$ be the endomorphism of $B_0$ whose restriction to the $m$-th graded piece of $B_0$ is multiplication by $(\frac{q^n-1}{q-1})^m$. Then

(i) the endomorphisms $\psi^n$ define a $\lambda$-ring structure on $B_0$;

(ii) in $B_0$ one has $\psi^n (q)=q^n$, so the map $\BZ[q]=\BZ[h]\mono B_0$ is a homomorphism of $\lambda$-rings; 

(iii) the diagram
\[
\xymatrix{
B_0\ar[r]^-\Delta\ar[d]_{\psi^n} & B_0\otimes_{\BZ[q]}B_0\ar[d]^{\psi^n\otimes\psi^n}\\
B_0\ar[r]^-\Delta & B_0\otimes_{\BZ[q]}B_0
}
\]
commutes, where $\Delta$ is the coproduct. 
\end{lem}

\begin{proof}
By the definition of $\psi^n:B_0\to B_0$, in the ring $B_0$ we have $\psi^n(q-1)=q^n-1$ (because $q-1=h$ is in the degree $1$ graded piece) and therefore
$\psi^n(q)=q^n$.

Let us prove (i). It is easy to check that $\psi^n\circ\psi^{n'}=\psi^{nn'}$. So by \S\ref{sss:Wilkerson}, it remains to check that for every prime $p$ the endomorphism of $B_0/pB_0$ induced by $\psi^p$ equals the Frobenius. This follows from \eqref{e:B_0 in terms of Int}, the fact that $\psi^p(h)\equiv h^p \mod p$, and Lemma~\ref{l:Fr=id}, which says that the Frobenius endomorphism of $\Int/p\Int$ equals the identity.

To prove (iii), recall that $B_0\otimes\BQ=\BQ[h,t]$, $\Delta (t)=t\otimes 1+1\otimes t$ (see Proposition~\ref{p:G=Spec B_0}(iii)) and 
$\psi^n(t)=\frac{q^n-1}{q-1}\cdot t$.
\end{proof}

\subsubsection{The morphisms $\Psi_n:G^!\to G^!$}   \label{sss:Psi_n}
The endomorphisms $\psi^n\in\End\BZ [h]$ and $\psi^n\in\End B_0$ induce maps $\Psi_n:\BA^1\to\BA^1$ and $\Psi_n:G^!\to G^!$. 
By Lemma~\ref{l:B_0 as lambda-ring}(ii), the diagram
\[
\xymatrix{
G^!\ar[r]^{\Psi_n} \ar[d] & G^!\ar[d]\\
\BA^1\ar[r]^{\Psi_n} & \BA^1
}
\]
commutes, so we get a morphism
\begin{equation}  \label{e:G^! to Psi_n^*G^!}
G^!\to\Psi_n^*G^!
\end{equation}
of schemes over $\BA^1$, where $\Psi_n^*G^!$ is the pullback of $G^!$ via $\Psi_n:\BA^1\to\BA^1$. 
Moreover, \eqref{e:G^! to Psi_n^*G^!} is a group homomorphism by Lemma~\ref{l:B_0 as lambda-ring}(iii).

\begin{lem}  \label{l:motivation of lambda-structure}
(i) Let $n\in\BN$. Let $\Psi_n^*H^!$ be the pullback of $H^!$ via $\Psi_n:\BA^1\to\BA^1$. Then there is a unique group homomorphism
\begin{equation}   \label{e:Psi_n^*H^! to H^!}
\Psi_n^*H^!\to H^!
\end{equation}
which makes the following diagram commute:
\begin{equation}   \label{e:diagram defining our map}
\xymatrix{
\Psi_n^*H^!\ar[r] \ar[rd]& H^!\ar[d]   \\
& \hat\BG_m\times\BA^1  
}
\end{equation}
Here the vertical arrow is the map \eqref{e:H to hat G_m} and the diagonal one is its pullback via $\Psi_n:\BA^1\to\BA^1$.

(ii) The homomorphisms \eqref{e:G^! to Psi_n^*G^!} and \eqref{e:Psi_n^*H^! to H^!}  are Cartier dual to each other.
\end{lem}

\begin{proof}
(i) $H^!$ is the formal group over $\BA^1$ given by the group law $z_1*z_2=z_1+z_2+(q-1)z_1z_2$. 
So $\Psi_n^*H^!$ is given by the group law $y_1*y_2=y_1+y_2+(q^n-1)y_1y_2$. The homomorphism \eqref{e:Psi_n^*H^! to H^!} is given by $z=\frac{q^n-1}{q-1}\cdot y$.

(ii) The Cartier dual of the vertical arrow of \eqref{e:diagram defining our map} is the homomorphism $f:\sR\times\BA^1\to G^!$ from \eqref{e:R to G}. So it suffices to check  commutativity of the diagram
\[
\xymatrix{
\Psi_n^*G^! & G^!\ar[l]   \\
& \sR\times\BA^1 \ar[ul]^{\Psi_n^*(f)} \ar[u]_f
}
\]
whose horizontal arrow is \eqref{e:G^! to Psi_n^*G^!}. This is equivalent to commutativity of the diagram
\begin{equation}     \label{e:Psi_n in terms of sR times A^1}
\xymatrix{
\sR\times\BA^1\ar[r]^-f \ar[d]_{\id_{\sR}\times\Psi_n} & G^!\ar[d]^{\Psi_n}\\
\sR\times\BA^1\ar[r]^-f & G^!
}
\end{equation}
and then (after passing to coordinate rings) to commutativity of the diagram
\[
\xymatrix{
\Int\otimes\BZ [h]\ar[d]_{\id\otimes\psi^n}\ & B_0 \ar[d]^{\psi^n}\ar[l] \\
\Int\otimes\BZ [h]& B_0 \ar[l]
}
\]
in which each horizontal arrow is the homomorphism \eqref{e:B_0 to Int[h]}. The commutativity of the latter diagram is clear from the definition of $\psi^n:B_0\to B_0$ from
Lemma~\ref{l:B_0 as lambda-ring}.
\end{proof}

\subsubsection{The $\delta$-ring $B_0\otimes\BZ_{(p)}$}  \label{sss:B_0 as delta-ring}
Fix a prime $p$. Let $\BZ_{(p)}$ be the localization of $\BZ$ at $p$. Let $\phi\in\End (B_0\otimes\BZ_{(p)})$ be induced by $\psi^p\in\End B_0$.
For every $b\in B_0\otimes\BZ_{(p)}$, the element $\delta (b):=\frac{\phi (b)-b^p}{p}$ belongs to $B_0\otimes\BZ_{(p)}$ by Lemma~\ref{l:B_0 as lambda-ring}(i).
The map $\delta :B_0\otimes\BZ_{(p)}\to B_0\otimes\BZ_{(p)}$ makes $B_0\otimes\BZ_{(p)}$ into a $\delta$-ring in the sense of \cite{JoyalDelta} and \cite{BS}. The subring
$\BZ_{(p)}[q]\subset B_0\otimes\BZ_{(p)}$ is a $\delta$-subring.

By the definition of $\psi^p$ (see  Lemma~\ref{l:B_0 as lambda-ring}), the element $t\in B_0\otimes\BZ_{(p)}$ satisfies the relation 
$\psi^p (t):=\frac{q^p-1}{q-1}\cdot t$ or equivalently,
\begin{equation}   \label{e:defining relation}
t^p+p\delta (t)=\Phi_p(q)\cdot t.
\end{equation}
On the other hand, let $C$ be the $\delta$-algebra over $\BZ_{(p)}[q]$ with a single generator (denoted by $t$) and the \emph{defining} relation \eqref{e:defining relation}. We claim that \emph{the canonical homomorphism $C\to B_0\otimes\BZ_{(p)}$ is an isomorphism.} Indeed, elements of the form
\[
\prod_i (\delta^i(t))^{d_i}, \quad \mbox{where } 0\le d_i<p  \mbox{ for all } i \mbox{ and } d_i=0 \mbox{ for }i\gg 0 
\]
generate\footnote{To see this, note that $\delta^(t)^p=\phi (\delta^i(t))-p\delta^{i+1}(t)=\delta^i(\Phi_p(q)\cdot t)-p\delta^{i+1}(t)$.} the $\BZ_{(p)}[q]$-module $C$ and form a basis of the $\BZ_{(p)}[q]$-module $B_0\otimes\BZ_{(p)}$ (the latter is similar to Lemma~\ref{l:generators of Int otimesZ_p}).

\subsubsection{Some generalizations} 
The generalizations discussed here are not used in the rest of the article.

(i) In \S\ref{sss:lambda-ring structure on Z [h]} we set $\psi^n(h):=(1+h)^n-1$. This choice of $\psi^n$ is motivated by our interest in the $q$-de Rham prism. On the other hand, one could set $\psi^n(h):=h^n$ and define $\psi^n:B_0\to B_0$ by setting $\psi^n(b)=h^{m(n-1)}$ for $b$ in the $m$-th graded piece of $B_0$. Then we would still get a $\lambda$-ring structure on $\BZ [h]$ and $B_0$; moreover, Lemmas~\ref{l:B_0 as lambda-ring} and \ref{l:motivation of lambda-structure} would remain valid.

(ii) In \S\ref{sss:B_0 as delta-ring} we considered the $\delta$-ring structure on $B_0\otimes\BZ_{(p)}$ corresponding to the endomorphism of $B_0\otimes\BZ_{(p)}$ that acts on the $m$-th graded piece as multiplication by $((1+h)^p-1)/h)^m$. If we replace $((1+h)^p-1)/h$ by any polynomial $f\in\BZ_{(p)}[h]$ congruent to $h^{p-1}$ modulo $p$ we would still get a $\delta$-ring structure on $B_0\otimes\BZ_{(p)}$ such that the elements $\delta^i(t)$, $i\ge 0$, generate $B_0\otimes\BZ_{(p)}$ over $\BZ_{(p)}[h]$.

\subsection{The group scheme $G^!$ in terms of Witt vectors.\,I}  \label{ss:G in terms of Witt}
\subsubsection{$\lambda$-schemes}
By a $\lambda$-scheme we mean a scheme $X$ equipped with a collection of endomorphisms $\Psi_n:X\to X$, $n\in\BN$, such that $\Psi_m\circ\Psi_n=\Psi_{mn}\,$, $\Psi_1=\id$, and for every prime $p$ the morphism $\Psi_p:X\otimes\BF_p\to X\otimes\BF_p$ equals $\Fr_{X\otimes\BF_p}\,$. (This definition is good enough for us because we will be dealing with schemes flat over $\BZ$.) Similarly to \S\ref{sss:group delta-scheme} we have the notion of group $\lambda$-scheme over a $\lambda$-scheme.

\subsubsection{Plan of \S\ref{ss:G in terms of Witt}-\ref{ss:G^! in terms of Witt}} 
$G^!$ is a group $\lambda$-scheme over the $\lambda$-scheme $\BA^1=\Spec\BZ [q]$ (see \S\ref{sss:lambda-ring structure on Z [h]}-\ref{sss:Psi_n}).
We will describe two realizations of  
this group $\lambda$-scheme in terms of Witt vectors, denoted by $G^{!?}$ and $G^{!!}$. The definitions of $G^{!?}$ and $G^{!!}$ are given in \S\ref{sss:G^!?} and \S\ref{sss:G^!!}, respectively. According to Propositions~\ref{p:G in terms of Witt} and \ref{p:G^!=G^!!}, the group $\lambda$-schemes $G^!$, $G^{!?}$, and $G^{!!}$ are canonically isomorphic.

Probably $G^{!!}$ is better than $G^{!?}$ (this opinion is influenced, in part, by my correspondence with Lance Gurney). 
However, let us start with $G^{!?}$, which is obtained by rescaling \S\ref{ss:R via W_big} in a straightforward way.

\subsubsection{Definition of $G^{!?}$} \label{sss:G^!?}
Let $W_{\bbig}$ be the ring scheme of ``big'' Witt vectors (so $W_{\bbig}\times\BA^1$ is a ring scheme over $\BA^1$).
Define $G^{!?}\subset W_{\bbig}\times\BA^1$ to be the following subgroup:
\begin{equation} \label{e:G^!? in terms of big Witt}
G^{!?}:= \{(w,q)\in W_{\bbig}\times\BA^1 \,|\, F_m (w)=[q-1]^{m-1}w \mbox{ for all } m\in\BN\}.
\end{equation}
For $n\in\BN$ define $\Psi_n:G^{!?}\to G^{!?}$ by 
\begin{equation}  \label{e:2 Psi_n(w,q)}
\Psi_n(w,q)=(\Big [\frac{q^n-1}{q-1}\Big ] \cdot w,q^n).
\end{equation}
It is easy to check that for each prime $p$ the morphism $\Psi_p:G^{!?}\otimes\BF_p\to G^{!?}\otimes\BF_p$ is equal to the Frobenius.
So $G^{!?}$ is a group $\lambda$-scheme over $\BA^1=\Spec\BZ [q]$.

\subsubsection{Remark}
Let $p$ be a prime and $W$ the ring scheme of $p$-typical Witt vectors. Similarly to the proof of Lemma~\ref{l:W_big & W}, one shows that the canonical epimorphism
$W_{\bbig}\epi W$ induces an isomorphism
\begin{equation}   \label{e:G^!? p-typically}
G^{!?}\otimes\BZ_{(p)}\iso \{(w,q)\in W\times\BA^1_{\BZ_{(p)}} \,|\, F (w)=[q-1]^{p-1}\cdot w\}.
\end{equation}

\begin{lem}   \label{l:flatness of G^!?}
$G^{!?}$ is flat over $\BZ [q]$.
\end{lem}

\begin{proof}
It suffices to show that the r.h.s of \eqref{e:G^!? p-typically} is flat over $\BZ_{(p)} [q]$.
The argument is parallel to 
that of \S\ref{sss:proof of flatness of W^F}, but the role of the elements
\[
x_n^p+px_{n+1}-x_n
\]
from 
\S\ref{sss:proof of flatness of W^F} is played by $x_n^p+px_{n+1}-(q-1)^{p^n (p-1)}x_n\,$.
\end{proof}

\subsubsection{A homomorphism $G^!\to W_{\bbig}\times\BA^1$}   \label{sss:G to W_big}
For any $\BZ [q]$-algebra $A$, an $A$-point of $G^!$ is an element $f\in 1+zA[[z]]\subset (A[[z]])^\times$ satisfying the functional equation 
\begin{equation} \label{e:f(x_1)f(x_2)}
f(z_1)f(z_2)=f(z_1+z_2+(q-1)z_1z_2).
\end{equation}
 Associating to such $f$ the formal power series 
\begin{equation}   \label{e:f(-x)}
f(-z)\in 1+zA[[z]]=W_{\bbig}(A),
\end{equation}
we get a group homomorphism $G^! (A)\to 1+zA[[z]]=W_{\bbig}(A)$ functorial in $A$, i.e., a homomorphism of group schemes over $\BA^1$
\begin{equation}  \label{e:G to W_big}
i: G^!\mono W_{\bbig}\times\BA^1\, .
\end{equation}
The morphism \eqref{e:G to W_big} is a closed immersion because \eqref{e:f(x_1)f(x_2)} is a closed condition.

\subsubsection{Relation to the homomorphism $\sR\to W_{\bbig}$}  \label{sss:[h]}
It is easy to check that after the specialization $q=2$ (i.e., $q-1=1$) the homomorphism \eqref{e:G to W_big} becomes the homomorphism
\begin{equation}  \label{e:2 R to W}
\sR\iso W_{\bbig}^F\mono W_{\bbig}\, 
\end{equation}
from \S\ref{sss:R to W} (the minus sign in \eqref{e:f(-x)} was introduced to ensure this). 
Moreover, one has  the following

\begin{lem}   \label{l:[h]}
(i)  The following diagram commutes: 
\begin{equation}  \label{e:[h]}
\xymatrix{
\sR\times\BA^1\ar[r] \ar@{^{(}->}[d] & G^!\ar@{^{(}->}[d]^i\\
W_{\bbig}\times\BA^1\ar[r]^{[q-1]} & W_{\bbig}\times\BA^1
}
\end{equation}
Here the upper horizontal arrow is \eqref{e:R to G}, the lower horizontal arrow is multiplication by the Teichm\"uller representative $[q-1]\in W_{\bbig}(\BZ [q])$, the right vertical arrow is \eqref{e:G to W_big}, and the left  vertical arrow comes from \eqref{e:2 R to W}. 

(ii) After base change to the open subset $\Spec\BZ [q,(q-1)^{-1}]\subset\Spec \BZ [q]=\BA^1$, the horizontal arrows of \eqref{e:[h]} become isomorphisms.
\end{lem}

\begin{proof}
Recall that for any $\BZ [q]$-algebra $A$, multiplication by $[q-1]$ in $W_{\bbig}(A)$ takes a formal series $g(x)\in 1+xA[[x]]=W_{\bbig}(A)$ to $g((q-1)x)$. The rest is straightforward.
\end{proof}

\subsubsection{Remarks}   \label{sss:compatibility with lambda-structure}
(i) By Lemma~\ref{l:Fr=id}, $\sR$ is a $\lambda$-scheme with $\Psi_n=\id$ for all $n$. 
Moreover, commutativity of \eqref{e:Psi_n in terms of sR times A^1} means that the upper horizontal arrow of \eqref{e:[h]} is a morphism of $\lambda$-schemes.

(ii) The lower horizontal arrow of \eqref{e:[h]} induces a morphism  $W_{\bbig}^F\times\BA^1\to G^{!?}$ of $\lambda$-schemes over $\BA^1$, which becomes an isomorphism over the locus $q\ne 1$.

\begin{prop}   \label{p:G in terms of Witt}
The homomorphism \eqref{e:G to W_big} induces an isomorphism
\begin{equation} \label{e:G in terms of big Witt}
G^!\iso G^{!?} 
\end{equation}
of $\lambda$-schemes over $\BA^1$. 
\end{prop}

\begin{proof}
The schemes $G^!$ and  $G^{!?}$ are flat over $\BZ [q]$ (for $G^{!?}$ this is Lemma~\ref{l:flatness of G^!?}). The morphism $i:G^!\to W_{\bbig}\times\BA^!$ is a closed immersion. So it remains to show that $i$ induces an isomorphism of $\lambda$-schemes $G^!_{q\ne 1}\iso G^{!?}_{q\ne 1}$, where $G^!_{q\ne 1}$ and $G^{!?}_{q\ne 1}$ are the restrictions of $G^!$ and $G^{!?}$ to the locus $q\ne 1$. This follows from Lemma~\ref{l:[h]} and \S\ref{sss:compatibility with lambda-structure}.
\end{proof}

\subsection{The group scheme $G^!$ in terms of Witt vectors.\,II}  \label{ss:G^! in terms of Witt}
This subsection is a non-$p$-typical version of \S\ref{sss:G_Q^!=G_Q^!!}. Part (iii) of Lemma~\ref{l:2 G^! to W_big} is somewhat surprising.

\subsubsection{Recollections on $G^!$}
$G^!$ is the group scheme over $\BA^1=\Spec\BZ [q]$ such that for any $\BZ [q]$-algebra $A$, $G^!(A)$ is the group of elements $f\in 1+zA[[z]]\subset (A[[z]])^\times$ satisfying the functional equation
\begin{equation} \label{e:f(z_1)f(z_2)}
f(z_1)f(z_2)=f(z_1+z_2+(q-1)z_1z_2).
\end{equation}
Recall that $H^0(G^!,\cO_{G^!})=B_0$, where $B_0$ is as in Proposition~\ref{p:G=Spec B_0}. The $\lambda$-scheme structure on $G^!$ was defined in 
\S\ref{sss:lambda-ring structure on Z [h]}-\ref{sss:Psi_n}.

\subsubsection{Definition of $G^{!!}$} \label{sss:G^!!}
Define $G^{!!}\subset W_{\bbig}\times\BA^1$ to be the following subgroup:
\begin{equation} \label{e:G^!! in terms of big Witt}
G^{!!}:= \{(w,q)\in W_{\bbig}\times\BA^1 \,|\, F_m (w)=\frac{[q]^m-1}{[q]-1}\cdot w \mbox{ for all } m\in\BN\},
\end{equation}
where $\frac{[q]^m-1}{[q]-1}:=1+[q]+\ldots+[q]^{m-1}$.
Define $\Psi_n:G^{!!}\to G^{!!}$ by the following very simple formula:
\begin{equation}  \label{e:3 Psi_n(w,q)}
\Psi_n(w,q)=(F_n(w),q^n).
\end{equation}
Then $G^{!?}$ is a group $\lambda$-scheme over $\BA^1=\Spec\BZ [q]$.

\subsubsection{Remark}
Let $p$ be a prime and $W$ the ring scheme of $p$-typical Witt vectors. Similarly to the proof of Lemma~\ref{l:W_big & W}, one shows that the canonical epimorphism
$W_{\bbig}\epi W$ induces an isomorphism
\begin{equation}   \label{e:G^!! p-typically}
G^{!!}\otimes\BZ_{(p)}\iso \{(w,q)\in W\times\BA^1_{\BZ_{(p)}} \,|\, F (w)=\Phi_p([q])\cdot w\}.
\end{equation}

\begin{lem} \label{l:2 G^! to W_big}
Equip $W_{\bbig}$ with the $\lambda$-scheme structure given by the Frobenius endomorphisms $F_n:W_{\bbig}\to W_{\bbig}$, $n\in\BN$.
Equip $G^!$ and $\BA^1=\Spec\BZ [q]$ with the $\lambda$-structure from \S\ref{sss:lambda-ring structure on Z [h]}-\ref{sss:Psi_n}.
Let $\pi :W_{\bbig}\epi\BG_a$ be the morphism that takes a Witt vector to its first component. Then

(i) there exists a unique morphism
\begin{equation}  \label{e:2 G^! to W_big}
G^!\to W_{\bbig}\times\BA^1 
\end{equation}
of $\lambda$-schemes over $\BA^1$ whose composition with the projection $W_{\bbig}\times\BA^1\to W_{\bbig}\overset{\pi}\longrightarrow\BG_a$
is given by the element $t\in B_0=H^0(G^! ,\cO_{G^!})$ from Proposition~\ref{p:G=Spec B_0};

(ii) the map \eqref{e:2 G^! to W_big} is a group homomorphism; 

(iii) the morphism \eqref{e:2 G^! to W_big} has the following explicit description: 
for any $\BZ [q]$-algebra $A$, it takes a formal series  $f\in 1+zA[[z]]$ 
satisfying \eqref{e:f(z_1)f(z_2)} to the formal series  $f (\frac{z}{z-1} )$
viewed as an element of $W_{\bbig}(A)$.
\end{lem}

\begin{proof}
The coordinate ring of the $\lambda$-scheme $W_{\bbig}$ is known to be the free $\lambda$-ring on a single generator $\pi$. This implies (i). Statement (ii) follows from (iii).

Let us prove (iii). Our map $G^!\to W_{\bbig}$ is given by the unique element of  $W_{\bbig}(B_0)$ whose $n$-th ghost component equals 
$\psi^n(t)=\frac{q^n-1}{q-1}\cdot t\in B_0$ (see the definition of $\psi^n$ in Lemma~\ref{l:B_0 as lambda-ring}).
Recall that the universal solution to  \eqref{e:f(z_1)f(z_2)} is given by 
\[
f(z)=(1+(q-1)z)^{t/(q-1)}.
\]
So it remains to check that for this $f$ one has
\[
-z\frac{d}{dz}\log f\Big (\frac{z}{z-1}\Big )=t\cdot\sum_{n=1}^\infty \frac{q^n-1}{q-1}\cdot z^n.
\]
This is straightforward; one uses that $1+(q-1)z/(z-1)=(1-qz)/(1-z)$.
\end{proof}

\begin{prop}   \label{p:G^!=G^!!}
The homomorphism \eqref{e:2 G^! to W_big} induces an isomorphism
\begin{equation} \label{e:G^!=G^!!}
G^!\iso G^{!!}, 
\end{equation}
where $G^{!!}\subset W_{\bbig}\times\BA^1$ is as in \S\ref{sss:G^!!}.
\end{prop}

\begin{proof} 
It suffices to show that \eqref{e:2 G^! to W_big} induces an isomorphism $G^{!}\otimes\BZ_{(p)}\iso G^{!!}\otimes\BZ_{(p)}$ for each prime $p$. The description of
$G^{!!}\otimes\BZ_{(p)}$ from \eqref{e:G^!! p-typically} allows one to prove this quite similarly to \S\ref{sss:G_Q^!=G_Q^!!}.
\end{proof}

\bibliographystyle{alpha}

\end{document}